\documentclass[preprint,11pt]{elsarticle}

\usepackage{amssymb}
\usepackage{enumerate}
\usepackage{amsmath}
\usepackage{amsfonts}
\usepackage{mathrsfs}

\usepackage[dvipsnames,usenames]{xcolor}
\usepackage{hyperref}
\usepackage{chngcntr}
\hypersetup{
colorlinks = true,
urlcolor = red,
linkcolor = blue,
}

\usepackage[
  top=3cm,
  bottom=3cm,
  left=2.5cm,
  right=2.5cm
]{geometry}

\newtheorem{theorem}{Theorem}[section]
\newtheorem{lemma}[theorem]{Lemma}
\newtheorem{e-proposition}[theorem]{Proposition}
\newtheorem{corollary}[theorem]{Corollary}
\newtheorem{e-definition}[theorem]{Definition}
\newtheorem{remark}{\it Remark\/}
\newtheorem{example}{\it Example\/}
\numberwithin{equation}{section}
\newenvironment{proof}{\par\noindent\textbf{Proof. }}{\hspace{\stretch{1}}$\square$\medskip\par}

\begin{document}

\begin{frontmatter}

 \author{Ch\'erif Amrouche $^a$}
 \ead{cherif.amrouche@univ-pau.fr}
 \address[authorlabel1]{Laboratoire de Math\'ematiques et Leurs Applications, UMR CNRS 5142\\ Universit\'e de Pau et des Pays de l'Adour,  64000 Pau,  France}
 \author{Mohand Moussaoui $^b$}
 \address[authorlabel2]{Laboratoire des \'Equations aux D\'eriv\'ees Partielles Non Lin\'eaires et
Histoire des Math\'ematiques, Ecole Normale Sup\'erieure, Kouba, Algeria}

\title{The Dirichlet Problem for the Laplacian \\ in Lipschitz Domains Revisited\\ \vspace{12pt} \large Dedicated to the memory of my friend Mohand Moussaoui}

\begin{abstract}  
The purpose of this work is to study the Dirichlet  problem:
$$
(\mathscr{L}_D)\ \ \ \  -\Delta u = f\quad \ \mbox{in}\ \Omega \quad
\mbox{and } \quad u = g \ \ \mbox{on }\Gamma,
$$
with data belonging to appropriate Sobolev spaces, when the domain $\Omega$ in $\mathbb{R}^N,$ with $N \geq 2$, is assumed to be only Lipschitz. Although this has been extensively studied by many authors, we would like to return to a number of fundamental questions and known results, such as the traces, the uniqueness and the maximal regularity of solutions. 
 
First, to treat non-homogeneous boundary conditions, we rigorously define the notion of traces for non regular functions. This approach replaces the non-tangential trace notion that has dominated the literature since the 80s. We identify a functional space 
\begin{equation*}
 E(\nabla;\, \Omega)\ =\ \left\{\,v\in H^{1/2}(\Omega);\ \nabla v\in  [\textit{\textbf H}^{\, 1/2}(\Omega)]'\,\right\},
\end{equation*}
which satisfies the embeddings $H^{1/2}_{00}(\Omega)\hookrightarrow E(\nabla;\, \Omega) \hookrightarrow H^{1/2}(\Omega)$ and the trace operator $\gamma: E(\nabla;\, \Omega)\rightarrow  L^2(\Gamma)$ is well defined, continuous and leads to a new
characterization of $H^{1/2}_{00}(\Omega)$ as precisely the kernel of this operator. Second, by using Grisvard's results and interpolation theory, we prove the following inequality:   let $\Omega$ be a polygonal domain, with $K$ angles $\omega_1, \ldots, \omega_K$ larger than $\pi$,  $\mathscr{A} = \{1 - \pi/\omega_k, \;   k = 1, \ldots, K\}$   and  $0 < s < 1$, then for any $v \in H^{2- s}(\Omega)\cap H^1_0(\Omega),$
\begin{equation}\label{inegH3demiabst}
 \Vert  v \Vert_{H^{2-s}(\Omega)} \leq C(\Omega)\Vert \Delta v \Vert_{H^{-s}(\Omega)} \quad \mathrm{iff} \quad s \notin \mathscr{A} ,
\end{equation}
where for $s = 1/2$ the space $H^{-1/2}(\Omega)$ denotes the dual space of $H^{1/2}_{00}(\Omega)$ and  $C_{P, L} (\Omega)$ depends only on $s$,  on Poincar\'e and Lipschitz constants of  $\Omega$. The same estimate holds for polyhedral domains with dihedral angle openings (toward the interior of $\Omega$) larger than $\pi$. By virtue of the inequality \eqref{inegH3demiabst}, the above characterization of $H^{1/2}_{00}(\Omega)$ and the uniqueness of $H^{1/2}(\Omega)$ solution to Problem $(\mathscr{L}_D)$,  we prove that  maximal regularity $H^{3/2}$ holds, in any bounded Lipschitz domain $\Omega$, for all right-hand sides in the dual of $H^{1/2}_{00}(\Omega)$.  This conclusion contradicts the prevailing claims in the literature since the 90s.  Third, we return to the very delicate question of the uniqueness of the solutions $W^{s, p}(\Omega)$ to  $(\mathscr{L}_D)$ problem, when $p$ lies between 1 and a limit value $p_0$ that depends only on $\Omega$.  Using explicit examples, we show that this question is poorly understood and that it has led to existence and regularity results which are only partially valid for the study of the $(\mathscr{L}_D)$ problem.

Finally, we revisit the classical Area Integral Estimate of Dahlberg, and of Kenig, Pipher, and Verchota. For a harmonic function $u$ in $\Omega$ vanishing at some interior point, the estimate asserts
\begin{equation}\label{ineg}
\int_\Gamma \vert u \vert^2 d\sigma \leq C \int_\Gamma \vert S(u)\vert^2 d\sigma \simeq C \int_\Omega \varrho \vert\nabla u\vert^2 dx,
\end{equation}
where $S(u)$ is the area integral of $u$ and $\varrho$ is the distance to the boundary. Using the inequality \eqref{inegH3demiabst} with $s = 1/2$ and an explicit function given by Ne$\mathrm{\check{c}}$as, we show that this inequality cannot hold in its stated form. However, we show that when the domain $\Omega$  is of class $\mathscr{C}^{1, 1}$, the above inequality is valid.
 Since the estimate \eqref{ineg} has been widely used to argue that  $H^{3/2}$-regularity  is unattainable for data in the dual of $H^{1/2}_{00}(\Omega)$, our counterexample provides a decisive clarification.

\medskip

\end{abstract}


\end{frontmatter}

\tableofcontents

\section{Introduction and main results}

The purpose of this work is to study the Dirichlet  problem:
$$
(\mathscr{L}_D)\ \ \ \  -\Delta u = f\quad \ \mbox{in}\ \Omega \quad
\mbox{and } \quad u = g \ \ \mbox{on }\Gamma,
$$
with data belonging to appropriate Sobolev spaces, where the domain $\Omega$ is assumed to be only Lipschitz. When $g = 0$ 
we denote this problem by $(\mathscr{L}_D^0)$ and when $f = 0$ we denote it by $(\mathscr{L}_D^H)$.  The Dirichlet problem has been extensively studied since the 1960s. In their classical work \cite{Lions},  Lions and Magenes provided a complete analysis for smooth domains and within the $L^2$-theory. Later, Grisvard \cite{Gri} and Ne\v{c}as  \cite{Necas} investigated the case where $\Omega$ is of class $\mathscr{C}^{r,1}$, with nonnegative integer $r$, while Grisvard  \cite{Gri1}, \cite{Gri2} also studied the particular case of polygons and polyhedra. As a consequence of Calder\'on-Zygmund theory on singular integrals and boundary layer potentials, it is well known that for every $f\in W^{m-2,p}(\Omega)$ and $g\in W^{m-1/p,p}(\Gamma)$ with integer $m \geq 1$ and $1 < p < \infty$, the problem $(\mathscr{L}_D)$ admits a unique solution $u\in W^{m,p}(\Omega)$ provided that $\Omega$ is of class $\mathscr{C}^{m-1,1}$ if $m \geq 2$ and of class $\mathscr{C}^{1}$ if $m = 1$. Moreover, if $f\in W^{s-2,p}(\Omega)$ and $g\in W^{s-1/p,p}(\Gamma)$ with $ m < s < m+1$, then $u\in W^{s,p}(\Omega)$ whenever $\Omega$ is of class $\mathscr{C}^{m,1}$.  Observe that the regularity of $\Omega$ can be weakened when $ m < s < m+1/p$.

During the 1980s, considerable attention shifted to the case where $\Omega$ is merely Lipschitz, a setting in which the situation changes
dramatically (see for instance \cite{Cos, Dahl, Dahlberg, Fabes1, Fabes2, Jer1, Jer2, Verchota1}). These problems continue to attract sustained interest up to the present day (see for example \cite{Fabes3, GK, J-K, MMS, Mik, Mitrea1, Mitrea2, McL}).

As we noted above, we know that if $\Omega $ is of class $\mathscr{C}^{1}$   and $1 < p < \infty$, then for any $f\in W^{-1,\, p}(\Omega)$ and $ g\in W^{1 - 1/p,\, p}(\Gamma),$ Problem $(\mathscr{L}_D)$ admits a unique solution $u\in W^{1,\, p}(\Omega).$ In the 1980s, Ne\v{c}as raised the question of solvability for Problem $(\mathscr{L}^0_D)$, {\it i.e}, with homogeneous boundary condition $g = 0$ in Lipschitz domains, when  $f\in W^{-1,\, p}(\Omega)$. The answer to this question was given in the celebrated paper of Jerison and Kenig  \cite{J-K}, Theorem A. Specifically, if $N \geq 3$, then for any $p > 3$, there exists a Lipschitz domain $\Omega$ and $f\in \mathscr{C}^{\infty}(\overline{\Omega})$ such that the solution $u\in H^1_0(\Omega)$ of Problem $(\mathscr{L}_D^0)$ does not belong to $W^{1,\, p}(\Omega)$ (for $N = 2$, the result holds for any $p > 4$). On the other hand, for every bounded Lipschitz  domain $\Omega$, there exists $q > 4$,  when $N = 2$ and  $q > 3$  (in fact when  $q > 2N/(N-1)$) when $N \geq 3$, depending on $\Omega$,  such that if $q' < p < q,$ then Problem $(\mathscr{L}^0_D)$ admits a unique solution $u\in W^{1,\, p}_0(\Omega)$ satisfying an estimate
$$
\Vert u \Vert_{W^{1,\, p}(\Omega)}\leq C \Vert f \Vert_{W^{-1,\, p}(\Omega)}.
$$
When $\Omega$ is of class $\mathscr{C}^1$, one can in fact take $q = \infty$. It is noteworthy that the exponents $4$ and $4/3$ in dimension 2,  respectively $3$ and $3/2$ in dimension 3,  corresponding respectively to limiting cases for the existence and uniqueness of solutions in $W^{1,\, p}_0(\Omega)$, are conjugate. This reflects the self-adjoint nature of the operator $\Delta : W^{1,\, p}_0(\Omega) \longrightarrow  W^{-1,\, p}(\Omega)$.

In the present paper, we focus on solving Problems $(\mathscr{L}_D^0)$ and $(\mathscr{L}_D^H)$ when $\Omega$ is Lipschitz.  For the existence of solutions, we restrict ourselves to the Hilbertian framework of Sobolev spaces $H^s(\Omega)$ with real  $s > 0$.  The $L^p$-theory is also discussed but will be treated in a subsequent work.  Before a more in-depth study, we will simply provide  criteria for the uniqueness of $W^{1, p}$ solutions for the problem $(\mathscr{L}_D)$, which allows us to revisit certain published results (see Section \ref{secuniqueness}). While these problems are not new and have been widely investigated using various techniques since the 1980s, we show that the limiting cases $s= 1/2$ and $s = 3/2$ for Problem $(\mathscr{L}_D^0)$ remain far from fully understood. In this direction, we establish new results concerning traces of functions in $H^s(\Omega)$, for $s \leq 1/2,$ under additional structural conditions. These results play a crucial role in achieving maximal $H^{3/2}$-regularity for right-hand sides belonging to the dual space of $H^{1/2}_{00}(\Omega)$.

It is worth noting that several arguments in the literature suggest that such maximal regularity fails, based on the so-called \emph{Area Integral Estimate}  \eqref{ineg}. However, we show that this inequality does not, in fact, hold in the general Lipschitz setting.

In the remainder of this introduction, we present our main results. Unless otherwise specified, $\Omega$ will denote a bounded Lipschitz
 domain in  $\mathbb{R}^N$, with $N \geq 2$. Throughout this manuscript the vector fields and the spaces of vector fields are denoted by bold fonts.
 
 \bigskip

\begin{center}
{\bf Main results}
\end{center}

The first result concerns norms equivalences,  see Section \ref{Preliminary results in weighted Sobolev}, which are one of the keys to establish the maximal $H^{3/2}$ regularity for the Laplace equation with a
homogeneous Dirichlet condition. 

\begin{theorem}[{\bf Norms Equivalences}]\label{c06-l1MR}

\noindent{i)} Let $0 \leq s \leq 1$. There exist constants $C_1> 0$  and $C_2> 0$ depending only on $s$ on the Poincar\'e constant of $\Omega$  such that for any $v\in H^s(\Omega)$
 \begin{equation*}
 \inf_{K\in\, \mathbb{R}} \Vert\,v+K\, \Vert_{H^{s}(\Omega)}\leq C_1\Vert\,\nabla v\Vert_{ [\widetilde{\textit{\textbf H}}\,^{1-s}(\Omega)]' } \leq C_2\Vert\, \varrho^{1-s}\nabla v\Vert_{{ \textit{\textbf L}}^2(\Omega) },
 \end{equation*}
 where $\varrho$ is the distance function to the boundary of $\Omega$ and $\widetilde{\textit{\textbf H}}\,^{1-s}(\Omega)$ is equal to $ {\textit{\textbf H}}\,^{1-s}_0(\Omega)$ if $s \neq 1/2$ and equal to  ${\textit{\textbf H}}\,^{1/2}_{00}(\Omega)$ if not.\\
ii) Let $v\in\mathscr{D}'(\Omega) $ and $0 \leq s \leq 1.$ Then we have the following implications 
\begin{equation*}
\sigma^{1-s}\nabla v\in  {\textit{\textbf L}}^2(\Omega) \Longrightarrow \nabla v\in  [\widetilde{\textit{\textbf H}}\,^{1-s}(\Omega)]'\Longrightarrow v \in H^s(\Omega),
\end{equation*}
 where $\sigma$ is the regularized distance function to the boundary of $\Omega$.\\
iii) For the critical case $s = 1/2$, we have also the following equivalence norms: 
 \begin{equation*}
 \Vert v \Vert_{H^{1/2}_{00}(\Omega)} \simeq \Vert \nabla v \Vert_{[H^{1/2}(\Omega)]'}, \quad for \; v\in H^{1/2}_{00}(\Omega),
\end{equation*} 
where 
\begin{equation*}\label{a3-e3}
H^{1/2}_{00}(\Omega)\ =\ \left\{v\in H^{1/2}(\Omega); \ \ v/ \sqrt \varrho \in L^2(\Omega)\right\}.
\end{equation*}
\end{theorem}

The second result is about traces of functions belonging to Sobolev spaces, which is crucial in the study of boundary value problems. We know that if $v\in H^s(\Omega)$ with $s > 1/2$ then the function $v$ has a trace which belongs to  $H^{s-1/2}(\Gamma)$. However, if $v\in H^{1/2}(\Omega)$ only, in general this function $v$ may have no trace. In the following theorem, we see  that  a very fine additional condition on $\nabla v$  allows to obtain such a trace, see Section \ref{traces}.

\begin{theorem}[{\bf Traces in the limit cases}]\label{TracesH1demigradH1demiprime1} 

i) The linear mapping $\gamma: v \mapsto v_{\vert \Gamma}$ defined on $\mathscr{D}(\overline{\Omega})$ can be extended by continuity to a linear and continuous mapping, still denoted $\gamma$, from $E(\nabla;\, \Omega)$ into $L^2(\Gamma)$.\\
ii) The kernel of $\gamma$ is equal to $H^{1/2}_{00}(\Omega)$, which gives the following characterisation:
$$
H^{1/2}_{00}(\Omega) =\ \left\{\,v\in H^{1/2}(\Omega);\ \nabla v\in  [\textit{\textbf H}^{\, 1/2}(\Omega)]'\, and\; v = 0 \;  on\, \Gamma \right\}.
$$ 
iii) If $\Omega$ is of class $\mathscr{C}^{1, 1}$, the mapping $\gamma: E(\nabla;\, \Omega) \rightarrow L^2(\Gamma)$ is surjective. 
\end{theorem}
As a consequence we get immediately the following implication: 
$$
v\in H^{3/2}(\Omega)\quad \mathrm{with}\quad \nabla^2 v\in  [\textit{\textbf H}^{\, 1/2}(\Omega)]' \quad \Longrightarrow \quad v_{\vert\Gamma} \in H^1(\Gamma) \quad \mathrm{and}\quad \partial_{\textit{\textbf n}} v \in L^2(\Gamma).
$$
Moreover, we have the following property: 
\begin{equation*}\label{traceH^{3/2}_{00}1}
v\in H^{3/2}_{00}(\Omega)\quad \Longrightarrow \quad v = 0 \quad \mathrm{in}\; H^1(\Gamma)\qquad  \mathrm{and}\qquad \partial_{\textit{\textbf n}} v    = 0 \quad \mathrm{in}\; L^2(\Gamma).
\end{equation*}
The notation $\nabla^k v$ denotes the derivatives $\partial^\alpha v$ with $\vert \alpha \vert = k$ and $\nabla^2 v$ denotes the Hessian of $v$. \medskip

Before exploring the questions of existence and regularity for the problem $(\mathscr{L}_D)$, it is important to first examine the case where the domain $\Omega$ is such that:
\begin{equation}\label{hyppoly}
\Omega\; \mathrm{is \, \, a \, \, polygon\, \, with} \, \,  n \, \, \mathrm{angles\, \, satisfying}:\quad \pi  <  \omega_1\leq \ldots \leq \omega_n. 
\end{equation}
Recall that the following kernel space
$$
\mathscr{N}_0(\Omega) =: \{\varphi \in L^{2}(\Omega); \;  \Delta \varphi = 0 \;\; \mathrm{in}\,  \Omega \; \; \mathrm{and} \;\;   \varphi = 0 \; \;on \, \Gamma\}
$$
which is not trivial contrarily to the case where $\Omega$ is $\mathscr{C}^{1, 1}$ or convex is of dimension $n$. We denote by $z_{1}, \ldots, z_{n}$ a base of this space (see Section \ref{poly} for more details). 

\begin{theorem} [{\bf Solutions in $H^{s}(\Omega)$ with $\Omega$ polygon}] \label{regpolthetasm+} Let $\Omega$ satisfying \eqref{hyppoly}. Setting  $\alpha_k = \pi/\omega_k$ for $k = 1, \ldots, n$ and by convention $\alpha_0 = 1$. Then,\\
i) the following operators are isomorphisms:
\begin{equation*}
\ \Delta : H^{2-\theta}(\Omega) \cap H^1_0(\Omega)\longrightarrow {H}^{-\theta}(\Omega),\quad  \mathrm{for} \; \theta \in \, ] 1 - \alpha_n, 1[, \; \mathrm{ with }\; \theta \neq 1/2,  
\end{equation*}
and
\begin{equation*}
 \Delta : H^{3/2}_0(\Omega)\longrightarrow  [{H}^{1/2}_{00}(\Omega)]'  \quad  \mathrm{for} \; \theta = 1/2.
\end{equation*}
ii) For any fixed $k = 0, \ldots, n-1$, $\theta \in \, ] 1 - \alpha_k, 1 - \alpha_{k+1}[$ and $f\in H^{-\theta}(\Omega)$ satisfying the following compatibility condition
\begin{equation*}
\forall \varphi \in \langle z_{k+1}, \ldots, z_{n}\rangle, \quad \langle f, \varphi\rangle = 0, 
\end{equation*}
Problem  $(\mathscr{L}_D^0)$ has a unique solution $u \in H^{2-\theta}(\Omega)$.\\
iii) For critical values $\theta = 1 - \alpha_k $, with $k = 1, \ldots, n$,   the operator
\begin{equation*}\label{CC}
 \Delta : H^{1 + \alpha_k}(\Omega) \cap H^1_0(\Omega)\longrightarrow M_{1 - \alpha_k}(\Omega)
\end{equation*}
is an isomorphism, where $M_\theta(\Omega) : = [(\mathscr{N}_0(\Omega))^\bot,  H^{-1}(\Omega)]_{\theta}$ and $(\mathscr{N}_0(\Omega))^\bot$ is the orthogonal subspace of
$\mathscr{N}_0(\Omega).$ Moreover  
$$
\bigcap_{r < 1 - \alpha_k}H^{-r}(\Omega) \hookrightarrow M_{1 - \alpha_k}(\Omega)\hookrightarrow H^{-1 + \alpha_k}(\Omega),$$
where the topology of $ M_{1 - \alpha_k}(\Omega)$ is finer than that of $ H^{-1 + \alpha_k}(\Omega)$.
\end{theorem}

\begin{remark}\label{rem3demi+eps} \upshape The results in Point i)  (also valid in the case of polyhedra) are not new, unlike those in Point ii) and Point iii). Note that when $\alpha_n$ is near 1/2, the  domain $\Omega$ is close to a cracked domain and the expected regularity in this case is better than $ H^{3/2}$. That means that for any nonconvex polygon there exists $\varepsilon = \varepsilon(\omega_n)\in \, ]0, 1/2[$ depending on  $\omega_n$ (in fact, $\varepsilon(\omega_n) = \alpha_n - 1/2$) such that for any $ 0 < s < \varepsilon$ and $f \in H^{s -1/2}(\Omega)$, the  $H^1_0(\Omega)$ solution of Problem $(\mathscr{L}_D^0)$ belongs to $ H^{s + 3/2}(\Omega)$. This result is not valid for general Lipschitz domain. However, we shall see that the $H^{3/2}(\Omega)$ regularity holds for RHS in the dual of $H^{1/2}_{00}(\Omega)$  for any bounded  Lipschitz domain $\Omega$.
\end{remark}

The next result of this manuscript addresses the $L^p$-theory. Unlike the case where the domain $\Omega$ is regular, the kernel space of harmonic functions $W^{s, p}(\Omega)$ that are zero on the boundary of $\Omega$, we denote by $W^{s, p}_0(\Omega)\cap \mathscr{H}$,  is generally non-trivial for a large collection of $s$ and $p$, as shown by the following explicit example. Indeed, let us  consider the following domain:
\begin{equation*}\label{domOmeg}
\Omega = \{(r, \theta);\; 0 < r < 1,\quad 0 < \theta < \pi /a\}, \quad \mathrm{with}\; a > 1/2\quad  \mathrm{and}\quad  \mathrm{close \, \, to } \, \, 1/2.
\end{equation*}
We can easily verify that the following function
\begin{equation*}\label{sing}
z(r, \theta) = (r^{-a} - r^{a})\mathrm{sin}(a\theta)
\end{equation*}
is harmonic in $\Omega$ with $z = 0$ on $\Gamma$. Moreover, given $s$ in the interval $[0, 3/2[$, it is easy to verify that the function $z$ belongs to $ W^{s, p}(\Omega) $ for any $p$ such that $1< p < 2/ (a + s)$. In particular for $s = 1$,  we have $z \in   W^{1, p}(\Omega) = L^p_1(\Omega) $ for any $p$ such that $1< p < 2/ (a + 1)$,  where $2/ (a + 1)$ is strictly less than  $4/3 $ but close  to $4/3$. We can also see  that $z$ belongs to $W^{3/2 - \varepsilon, 1}(\Omega)$ for any $\varepsilon > a - 1/2$ and $z \in L^q(\Omega)$ for any $q  < 2/a$.  We give in Section  \ref{secuniqueness} another example when $N =  3$. Note that this kernel space may even be of infinite dimension when $N \geq 3$.  

Let us now recall the following result established by Jerison and Kenig in \cite{J-K}, Proposition 5.17: let $\Omega$ be a bounded Lipschitz domain and suppose that $v$ satisfies the following properties:
\begin{equation*}\label{KerJK}
v \in L^p_s(\Omega)\; \; \mathrm{for \, some} \; s > 1/p, \quad \mathrm{with}\; \Delta v = 0\; \;  \mathrm{in}\,  \Omega  \quad \mathrm{and}\quad   v = 0\; \; \mathrm{on}\; \Gamma,
\end{equation*} 
then $v$ is identically zero.   This result is in fact true for $\mathscr{C}^1$ domains but not for Lipschitz domains in general, as shown in the example above with  $s = 1$ and $1 < p <   2/ (a + 1)$. 

This shows that this issue is poorly understood. Unfortunately, this lack of understanding has led to existence and regularity results that are only partially valid for the study of the problem $(\mathscr{L}_D)$.
 
In Section \ref{secuniqueness}  we prove the following:

\begin{theorem} [{Uniqueness Criteria}] \label{thunicityintro}  Let $\Omega$ be a bounded Lipschitz  domain of $\mathbb{R}^N$ with $N \geq 2$. 

\noindent i) Let \(u \in H^{1/2}(\Omega)\) be harmonic in \(\Omega\) and satisfy
\(u|_{\Gamma} = 0\). Then

\[u \equiv 0 \quad \text{in } \Omega.\]

\noindent ii) More generally, 
\begin{equation*}\label{carkerdim2}
\forall \, 1/2 < \alpha < (N+1)/2, \quad 
W^{\alpha, 2N/(N+2\alpha-1)}_0(\Omega) \cap \mathscr{H}= \{0\},
\end{equation*}
or equivalently
\begin{equation*}\label{carkerdim3}
\forall \, 1 < p <  2, \quad 
W^{1/2 + N(1/p - 1/2), \, p}_0(\Omega) \cap \mathscr{H}= \{0\},
\end{equation*}
where $ \mathscr{H} $ is the space of harmonic functions in $\Omega$.

\noindent iii) For any $ p <  2N/(N+1)$, there exists a bounded Lipschitz  domain $\Omega$ such that
\begin{equation*}\label{carkerdim4}
W^{1, p}_0(\Omega) \cap \mathscr{H}\neq  \{0\}.
\end{equation*}
More generally, let $1 < p < 2$ be fixed. For any $1 < r < p$, there exists a bounded Lipschitz  domain $\Omega$ such that
\begin{equation*}\label{carkerdim5}
W^{1/2 + N(1/p - 1/2), \, r}_0(\Omega) \cap \mathscr{H}\neq \{0\}.
\end{equation*}

\noindent iv) For any bounded Lipschitz  domain $\Omega$ of $\mathbb{R}^N$, with $N \geq 2$, there exists $p_0(\Omega) < 2N/(N+1)$ such that
\begin{equation}\label{carkerLp1bintro}
W^{1, p}_0(\Omega) \cap \mathscr{H} =  \{0\} \, \; \mathrm{if} \; \; p \geq p_0(\Omega)\quad  and  \quad W^{1, p}_0(\Omega) \cap \mathscr{H}\neq  \{0\} \, \; \mathrm{if} \; \; 1< p < p_0(\Omega).
\end{equation}
When $\Omega$ is $\mathscr{C}^1$ or convex,  the exponent $p_0(\Omega)$ may be taken to be equal $1$. 

\noindent v) Moreover if $\Omega$ is a polygonal (resp. 3D polyedral), we have
\begin{equation*}\label{defp02D}
p_0(\Omega) = 2\omega^\star/(\pi + \omega^\star) \quad \mathrm{with} \quad \omega^\star < 2\pi \; \mathrm{is \, \, the\, \,  larger\, \,  angle\, \,   of} \; \Omega, \quad \mathrm{for}\; N = 2. 
\end{equation*}
and
\begin{equation*}\label{defp03D}
p_0(\Omega) = 3\omega^\star/(\pi + \omega^\star) \quad \mathrm{with} \quad \omega^\star < \pi \; \mathrm{is \, \, the\, \,  larger\, \, polar\, \, angle\, \,   of} \; \Omega, \quad \mathrm{for}\; N = 3. 
\end{equation*}
\end{theorem}

The kernels above are very useful, not only for studying uniqueness, but also to study the regularity of the solutions to the Dirichlet problem for the Laplacian.  In \cite{J-K} Theorem 0.5, the authors write that for any bounded Lipschitz domain $\Omega$ in $\mathbb{R}^N$, there is an exponent  $p_1$ with $p_1 > 4$ when $N = 2$ and $p_1 > 3$ when $N \geq 3$, such that if $p_1' < p < p_1 $ and $f\in W^{-1, p}(\Omega)$,  then the inhomogeneous Dirichlet problem $(\mathscr{L}^0_D)$ has a unique solution $u \in W^{1, p}(\Omega)$. In fact, the maximal interval for the exponent $p$ to get a unique solution in $W^{1, p}(\Omega)$ decreases, in the sense of inclusion, as the dimension $N$ increases and the exponent $p_1$ given in Theorem 0.5 satisfies, more precisely, the inequality  
$$
p_1 > 2N/(N-1)\quad \mathrm{ for\, \,  any}\quad N\geq 2.
$$  

Another very important consequence of Theorem \ref{thunicityintro} is that Theorems 1.1 and 1.3 of \cite{J-K}, as stated, need to be reformulate, since conditions (b) and (c) are inappropriate. As these two theorems, concerning the solvability in $L^p_s(\Omega)$ for Problem ($\mathscr{L}_D^0$),  are widely used, we felt it was important to revisit them,  see Section \ref{secuniqueness} for more details. For further information, see also  the Appendix.\\

In \cite{Necas}, Ne\v{c}as proved the following property (see Theorem 2.2 Section 6): if $\varrho^{\alpha/p} u\in L^p(\Omega)$ and $\varrho^{\alpha/p}\nabla u \in L^p(\Omega)$, with $0 \leq \alpha < p - 1$, then $u_{\vert\Gamma} \in L^p(\Gamma)$ and 
\begin{equation*}\label{inegNec}
\int_\Gamma \vert u \vert^p d\sigma \leq C(\Omega) (\int_\Omega \varrho^\alpha\vert  u \vert^p dx + \int_\Omega \varrho^\alpha\vert \nabla u \vert^p dx),
\end{equation*}
where $\varrho$ is the distance function to the boundary of $\Omega$. However, if $\alpha = p - 1$, the above inequality does not hold in general, as proved in a counter example with $\Omega= \,  ]0, 1/2[\,  \times\,  ]0, 1/2[ $. In particular, if $ \sqrt \varrho \, \nabla u \in L^2(\Omega)$, corresponding to the case  $p = 2$ and $\alpha = 1$, in which case we know  that $u \in H^{1/2}(\Omega)$ (see Theorem \ref{c06-l1}), the function $u$ may have no trace in $L^2(\Gamma)$.  At higher order and for example if $u\in H^{3/2}(\Omega)$, there is a $\mathscr{C}^1$ domain $\Omega \subset \mathbb{R}^2$ and a function $u\in H^{3/2}(\Omega)$ whose trace on $\Gamma$ does not have a tangential derivative in $L^2(\Gamma)$ (see Proposition 3.2 in  \cite{J-K}). However that is the case when the domain $\Omega$ is of class $\mathscr{C}^{1, 1}.$

What about if in addition the function $u$ is harmonic? In \cite{Dahl}, see Corollary, Section 6,  the author proved the following property: let $\textit{\textbf x}_0$ a fixed point in $\Omega$, then there exists a constant $C > 0$ such that if $u$ is a harmonic function in $\Omega$ and vanishes at $\textit{\textbf x}_0$, then 

\begin{equation}\label{inegaltraceL2Gamma1bis}
C^{-1} \Vert u \Vert_{L^2(\Gamma)} \leq \Vert \sqrt  \varrho\,  \nabla u  \Vert_{L^{2}(\Omega)} \leq C \Vert u \Vert_{L^2(\Gamma)}.
\end{equation}

In \cite{Dahlberg}, the authors showed in the same context,  see pages 1428 and 1429,  but with different proof, a result close to the one mentioned above:  there exists $C > 0$ such that
\begin{equation*}\label{ineg1bis}
\int_\Gamma \vert u \vert^2 d\sigma \leq C \int_\Gamma \vert S(u)\vert^2 d\sigma = C \int_\Omega \delta \vert\nabla u\vert^2 dx,
\end{equation*}
without any specification of spaces of such function and with unjustified formal calculations. Here $S(u)$ is the area integral of $u$ and $\delta$ is an adaptative distance to the boundary, equivalent  to the distance $\varrho$ to the boundary $\Gamma$. \medskip

Recall that if $u\in L^2(\Omega)$, then we have the following implications: 
$$
\sqrt \varrho\,  \nabla u \in L^2(\Omega) \Longrightarrow u \in H^{1/2}(\Omega) \quad \mathrm{and}\quad \sqrt \varrho\,  \nabla^2 u \in L^2(\Omega) \Longrightarrow u \in H^{3/2}(\Omega)
$$
and the reverse implications hold if moreover $u$ is a harmonic function, see Theorem 3.2 and Theorem 3.8 in \cite{AM} or Theorem 4.2 in \cite {J-K}. Here $\nabla^2 u$ denotes the Hessian matrix of $u$. In addition,  we have  the following equivalence norms for harmonic functions:
\begin{equation*}
\begin{array}{rl}
\Vert u \Vert_{H^{1/2}(\Omega)} \approx & \Vert u \Vert_{L^2(\Omega)} + \Vert \sqrt  \varrho\,  \nabla u  \Vert_{L^{2}(\Omega)} \quad \mathrm{for}\; u \in H^{1/2}(\Omega)\cap \mathscr{H}, \\
\Vert u \Vert_{H^{3/2}(\Omega)} \approx & \Vert u \Vert_{L^2(\Omega)} + \Vert \sqrt  \varrho\,  \nabla^2 u  \Vert_{L^{2}(\Omega)} \quad \mathrm{for}\; u \in H^{3/2}(\Omega)\cap \mathscr{H}. 
\end{array}
\end{equation*}
Inequalities \eqref{inegaltraceL2Gamma1bis} would then result in the equivalence of the following norms:
\begin{equation*}
\Vert u \Vert_{H^{1/2}(\Omega)} \approx  \Vert u \Vert_{L^2(\Gamma)} \quad \mathrm{for}\; u \in H^{1/2}(\Omega)\cap \mathscr{H} 
\end{equation*}
and would imply that any harmonic function $H^{1/2}(\Omega)$ has a trace  in $L^2(\Gamma)$.  Consequently, the gradient of any harmonic function $u$ in $H^{3/2}(\Omega)$ would have a trace in $\textit{\textbf L}^2(\Gamma)$, {\it i.e} $u\in H^1(\Gamma)$ and $\partial_\textit{\textbf n} u \in L^2(\Gamma)$.\medskip

In the same spirit, the first  inequality in  \eqref{inegaltraceL2Gamma1bis} would imply the following property: let $u$ be a harmonic function in $\Omega$ satisfying $u(\textit{\textbf x}_0) = 0$ and $\nabla u(\textit{\textbf x}_0) = {\bf 0}$ at some point $\textit{\textbf x}_0\in \Omega$, then
\begin{equation}\label{inegH32intro}
\Vert u \Vert_{H^1(\Gamma)} \leq C(\Omega)\Big(\int_\Omega \varrho\vert \nabla^2 u \vert^2 dx\Big)^{1/2}.
\end{equation}
And as above,  we would have the equivalence of the following norms:
\begin{equation*}
\Vert u \Vert_{H^{3/2}(\Omega)} \approx  \Vert u \Vert_{H^1(\Gamma)} \quad \mathrm{for}\; u \in H^{3/2}(\Omega)\cap \mathscr{H}.
\end{equation*}

In the following proposition, we give a counter-example which shows that the inequality \eqref{inegH32intro} and then  the first  inequality in  \eqref{inegaltraceL2Gamma1bis}  cannot in general be satisfied for Lipschitz domain $\Omega$, see Proposition \ref{ConterexampleH1demitrace1b}. However it is the case when $\Omega$ is $\mathscr{C}^{1, 1}$  as we can see in Section \ref{traces2}.\smallskip

\begin{e-proposition}  [{\bf Counter Example}]  \label{ConterexampleH1demitrace1}For any $\varepsilon > 0$, there is a Lipschitz domain $\Omega_\varepsilon \subset \mathbb{R}^2$ and a harmonic function $w_\varepsilon \in H^{3/2}(\Omega_\varepsilon), $ with $\sqrt \varrho_\varepsilon\, \nabla^2 w_\varepsilon  \in L^2(\Omega_\varepsilon) $ and where $\varrho_\varepsilon$ is the distance to the boundary $\Gamma_\varepsilon$, such that the following family 
$$
( \Vert w_\varepsilon \Vert_{H^{3/2}(\Omega_\varepsilon)} + \Vert \varrho_\varepsilon \nabla^2 w_\varepsilon \Vert_{L^{2}(\Omega_\varepsilon)})_\varepsilon,
$$ is bounded with respect $\varepsilon$ and
\begin{equation*}
\lim_{\varepsilon\rightarrow  0}\Vert w_\varepsilon \Vert_{H^1(\Gamma_\varepsilon)} = + \infty .
\end{equation*}
\end{e-proposition}
Of course, the above statement can be replaced by the following:
for any $\varepsilon > 0$, there is a Lipschitz domain $\Omega_\varepsilon \subset \mathbb{R}^2$ and a harmonic function $w_\varepsilon \in H^{1/2}(\Omega_\varepsilon)$, with $\sqrt \varrho_\varepsilon\, \nabla w_\varepsilon  \in L^2(\Omega_\varepsilon)$, such that the following family 
$$
( \Vert w_\varepsilon \Vert_{H^{1/2}(\Omega_{\varepsilon})} +  \Vert \varrho_\varepsilon \nabla w_\varepsilon \Vert_{L^{2}(\Omega_{\varepsilon})})_\varepsilon,
$$
is bounded with respect $\varepsilon$ and
\begin{equation*}
\Vert w_\varepsilon \Vert_{L^2(\Gamma_\varepsilon)} \rightarrow \infty \quad as \; \varepsilon\rightarrow  0.
\end{equation*}
As a consequence of Proposition \ref{ConterexampleH1demitrace1}, arguments in the literature relying on
inequality \eqref{inegH32intro} to conclude the failure of \(H^{3/2}\)-regularity are
invalid.\medskip

This brings us to the question of the maximal regularity in the case of Lipschitz domains. It is well known that for any $1/2 < s < 3/2$, Problem $(\mathscr{L}_D^0)$ has a unique solution $u \in H^{s}(\Omega)$ for every $f \in  H^{s-2}(\Omega)$. Moreover if $f\in L^2(\Omega)$ (or even if $f\in H^{-s}(\Omega)$ for any $s < 1/2$), then there exists a unique solution $u\in H^{3/2}_0(\Omega)$ to Problem $(\mathscr{L}_D^0)$. But these both assumptions on $f$ are too strong. So it would be interesting to characterize the range of  $H^{3/2}_0(\Omega)$ by the Laplacian operator. One of our main results, which was not proved yet so far as we know, is given by the next theorem, see Section 
\ref{Inhomogeneous Problem}.

\begin{theorem} [{\bf Solutions in $H^{3/2}_0(\Omega)$ and in $H^{1/2}_{00}(\Omega)$}] \label{IsoDeltaH3/2H1/21}

i) The operators 
\begin{equation}\label{isoH3demiintro}
\Delta : H^{3/2}_0(\Omega) \longrightarrow [{H}^{1/2}_{00}(\Omega)]'\quad and \quad \Delta : H^{1/2}_{00}(\Omega) \longrightarrow  [H^{3/2}_0(\Omega)]'
\end{equation}
are isomorphisms.\\
ii) For any $f\in [H^{1/2}(\Omega)]' $ satisfying the compatibility condition 
\begin{equation*}
 \forall  \varphi \in H^{1}(\Omega)\cap \mathscr{H}, \quad \langle f, \, \varphi  \rangle = 0,
\end{equation*} 
where 
$$
H^{1}(\Omega)\cap \mathscr{H}\ =\ \left\{ v\in H^{1}(\Omega); \ \ \Delta v=0  \;\; in\;\; \Omega\right\},
$$
there exists a unique solution $u\in  H^{3/2}_{00}(\Omega)$ such that $\Delta u = f$ in $\Omega$.  In addition to the Dirichlet  boundary condition $u = 0$, the normal derivative of this solution satisfies  $\partial_{\textit{\textbf n}}u = 0$. \\
\end{theorem}

In particular, this shows that the  regularity \(H^{3/2}\) holds for all bounded Lipschitz domains, in contrast with previous claims in the literature.  According to Remark \ref{rem3demi+eps}, this maximal regularity $H^{3/2}$ is not surprising. Moreover, recall that the operator
\begin{equation}\label{isoW12N/N-1intro}
\Delta : W^{1, 2N/(N- 1)}_0(\Omega) \longrightarrow W^{-1, 2N/(N- 1)}(\Omega)
\end{equation}
is an isomorphism for $N \geq 2$. As by Sobolev embeddings, we know that $ H^{3/2}_0(\Omega) \hookrightarrow W^{1, 2N/(N- 1)}_0(\Omega)$ and  $[{H}^{1/2}_{00}(\Omega)]'  \hookrightarrow W^{-1, 2N/(N- 1)}(\Omega)$, the first operator in \eqref{isoH3demiintro} can be considered as a regularity result with respect the isomorphism \eqref{isoW12N/N-1intro}.

From Theorem \ref{TracesH1demigradH1demiprime1}, the second isomorphism in Point i) above means that for any $f\in [H^{3/2}_0(\Omega)]'$, there exists a unique solution $u\in H^{1/2}_{00}(\Omega)$ satisfying 
$$
\Delta u = f\quad \mathrm{ in }\; \Omega\quad \mathrm{ and }\quad u = 0 \quad  \mathrm{on}\;  \Gamma.
$$ 
Furthermore, we deduce by interpolation that for any $0 < s < 1$, the following operator 
\begin{equation*}\label{isoHsb}
\Delta : H^{3/2- s}_0(\Omega) \rightarrow H^{-1/2-s}(\Omega)
\end{equation*}
is an isomorphism.

Concerning Problem $(\mathscr{L}_D^H)$, using harmonic analysis techniques, many authors have established existence results (see \cite{Jer1} and \cite{J-K}). In the case where $g\in L^2(\Gamma)$, it is proved in \cite{Dahl77} the existence of a unique harmonic function such that $u$ tends  nontangentially to $g$ {\it a.e} on $\Gamma$ (see \cite{Dahl77})  and $u$ satisfies
$$
\Vert u^\star\Vert_{L^2(\Gamma)} \leq C \Vert g \Vert_{L^2(\Gamma)}.
$$
Here, the nontangential maximal function $u^\star$ is defined by
$$
\textit{\textbf z}\in \Gamma, \quad u^\star(\textit{\textbf z}) = \sup_{\textit{\textbf x}\, \in \, \Gamma(\textit{\textbf z})} \vert u(\textit{\textbf x})\vert,
$$
where $\Gamma(\textit{\textbf z})  = \{\textit{\textbf x}\in \mathbb{R}^N; \; \vert \textit{\textbf x} - \textit{\textbf z}\vert < C \varrho(\textit{\textbf x})\},$ for a suitable constant $C > 1$ is a nontangential cone with vertex at $\textit{\textbf z}$. Recall that the notation $\varrho(\textit{\textbf x})$ denotes the distance from $\textit{\textbf x}\in \Omega$ to $\Gamma$. Similarly, when $g\in H^1(\Gamma)$, there exists a unique harmonic function such that $u = g$ on $\Gamma$ and $u$ tends  nontangentially to $g$ {\it a.e} on $\Gamma$ (see \cite{Jer1})  with the estimate
$$
\Vert (\nabla u)^\star\Vert_{L^2(\Gamma)} \leq C \Vert g \Vert_{H^1(\Gamma)}.
$$
This idea of a non tangential maximal function likely originates from the half-space: in \cite{Stein}, the author consider the Poisson integral of $g\in L^p(\mathbb{R}^{N-1})$, defined by
\begin{equation}\label{intPoisson}
u(\textit{\textbf x}\,', x_N) = P_{x_{N}} \star g (\textit{\textbf x}\,'), \quad \mathrm{with}\quad  P_{x_{N}}(\textit{\textbf x}\,') = c_N\frac{x_N}{(\vert \textit{\textbf x}\,' \vert ^2 + x_N^2)^{(N-1)/2}}, \quad \textit{\textbf x}\,' \in \mathbb{R}^{N-1}, \quad x_N > 0,
\end{equation}
which is harmonic in the half-space $\mathbb{R}^{N}_+$ and satisfies 
$$
\lim_{x_N\rightarrow 0} u(\textit{\textbf x}\,', x_N)  = g(\textit{\textbf x}\,'), \quad \mathrm{for\, \,  almost\, \, every}\;  \textit{\textbf x}\,' \in \mathbb{R}^{N-1}.
$$
Furthermore
$$
\lim_{
\begin{array}{c}
\textit{\textbf x}\rightarrow \textit{\textbf z}\\  \textit{\textbf x}\in \Gamma(\textit{\textbf z})
\end{array}}
 u(\textit{\textbf x})  = g( \textit{\textbf x}\, '^0), \quad \mathrm{for\, \,  almost\, \, every}\;  \textit{\textbf z} = ( \textit{\textbf x}\, '^0, 0)  \in \partial \mathbb{R}^{N}_+.
$$
However the relation \eqref{intPoisson} does not allow to define the trace of $u$  on $ \mathbb{R}^{N-1}$. Nevertheless, we can study the problem $(\mathscr{L}^0_D)$ for data $g$ in $ L^p(\mathbb{R}^{N-1})$. The appropriate functional framework is that of weighted fractional Sobolev spaces, which is more precise than that of homogeneous fractional Sobolev spaces. For simplicity, we consider the case $p = 2$. Problem $(\mathscr{L}^0_D)$ then has a unique solution $v$ that satisfies:

\begin{equation*}\label{expsol}
\widehat{v}(\xi', x_N) = e^{-2\pi\vert \xi'\vert x_N} \widehat{g}(\xi'), \quad \mathrm{for}\quad \xi' \in \mathbb{R}^{N-1}, \; x_N > 0.
\end{equation*}
Moreover 
$$
\int_{\mathbb{R}^{N}_+} \frac{\vert v \vert^2}{1+\vert \textit{\textbf x}\vert} d\textit{\textbf x} < \infty \quad \mathrm{and}\quad \vert v \vert_{H^{1/2}(\mathbb{R}^{N}_+)} = \frac{1}{\pi} \Vert g \Vert_{L^2(\mathbb{R}^{N-1})} < \infty,
$$
where $\vert \cdot \vert_{H^{1/2}(\mathbb{R}^{N}_+)}$ is the semi-norm $H^{1/2}(\mathbb{R}^{N}_+)$. Unlike the function $u$ defined above, on the boundary $\Gamma =  \mathbb{R}^{N-1}$ we have $v = g$ in the sense $L^2(\Gamma)$. Observe that 
$$
\widehat{v}(\xi', x_N) = \widehat{P_{x_{N}}}(\xi')\, \,  \widehat{g}(\xi'),  \quad \mathrm{for}\quad \xi' \in \mathbb{R}^{N-1}, \; x_N > 0
$$
{\it i.e} $v = u$ in $\mathbb{R}^{N}_+$.

We will give here a new proof of existence results in the case of boundary data in $L^2(\Gamma)$ or in $H^1(\Gamma)$, see Section \ref{SectionDPL},  which is essentially based in the case $g \in L^2(\Gamma)$ on the first isomorphism given in Theorem \ref{IsoDeltaH3/2H1/21} and on the following variant of  Ne\v{c}as' property:  
\begin{equation*}\label{uH10DeltauH1/2prime1}
\varphi\in H^1_0(\Omega)\quad \mathrm{ and} \quad\Delta \varphi  \in\ [H^{1/2}(\Omega)]' \quad\Longrightarrow \quad \partial_{\textit{\textbf n}} \varphi\in L^2(\Gamma),
\end{equation*}
(see Theorem \ref{d02-110118-th1a1} for how to extend this property for $\varphi\in H^1(\Omega)$).

\begin{theorem} [{\bf Homogeneous Problem in $H^{1/2}(\Omega)$ and in $H^{3/2}(\Omega)$}]\label{ThIsogL2Gamma1} i) For any $g\in L^{2}(\Gamma)$, Problem $(\mathscr{L}_D^H)$ has a unique solution $u\in H^{1/2}(\Omega)$. Moreover $\sqrt \varrho\, \nabla u  \in {\textit{\textbf L}}^{2}(\Omega)$ and there exists a constant $C(\Omega)$ such that 
\begin{equation*}
 \Vert u \Vert_{H^{1/2}(\Omega)} + \Vert\sqrt \varrho\, \nabla u \Vert_{\textit{\textbf L}^2(\Omega)}  \leq \ C(\Omega) \Vert g \Vert_{L^{2}(\Gamma)}.
\end{equation*}
ii) This solution satisfies the following relation: for any $\varphi \in H^2(\Omega)\cap H^1_0(\Omega)$, we have
\begin{equation*}
\int_\Omega u \Delta \varphi dx = \int_\Gamma g\partial_{\textit{\textbf n}}\varphi d\sigma.
\end{equation*}
iii) Moreover $u$ satisfies also the following property: for any positive integer $k$
$$
\varrho^{k + 1/2}\nabla^{k + 1}u \in \textit{\textbf L}^2(\Omega).
$$
iv) For any $g\in H^{1}(\Gamma)$, the problem $(\mathscr{L}_D^H)$ has a unique solution $u\in H^{3/2}(\Omega)$. Moreover $\sqrt \varrho \, \nabla^2 u \in {\textit{\textbf L}}^2(\Omega)$ and there exists a constant $C(\Omega)$ such that :
\begin{equation*}
 \Vert u \Vert_{H^{3/2}(\Omega)} + \Vert\sqrt \varrho\, \nabla^2 u \Vert_{\textit{\textbf L}^2(\Omega)} \leq \ C(\Omega) \Vert g \Vert_{H^{1}(\Gamma)}.
\end{equation*}
The solution $u$ satisfies also the following property: for any positive integer $k$
$$
\varrho^{k + 1/2}\nabla^{k + 2}u \in \textit{\textbf L}^2(\Omega).
$$
\end{theorem}

\begin{remark}\upshape i) In addition to the existence and uniqueness result given in Point i) above, the sense of the boundary condition $u = g$ is as usual for boundary value problems the one given by the trace $L^2(\Gamma)$ (see Remark \ref{remrelevbis}, Point iv)) and not in the non-tangential sense as it has been found in the literature since the 80s (see further comments in Section \ref{SectionDPL}).\\
ii) From Theorem \ref{IsoDeltaH3/2H1/21} and Point i) above, we deduce the following characterization of $H^{1/2}_{00}(\Omega)$ as follows: let $u\in H^{1/2}(\Omega)$, then
$$
u \in H^{1/2}_{00}(\Omega) \Longleftrightarrow \Delta u \in [H^{3/2}_0(\Omega)]' \quad \mathrm{and}\quad u = 0 \; \; \mathrm{on} \; \Gamma.
$$
With this characterization, we can also get the following: let $u\in H^{3/2}_0(\Omega)$, then
$$
u \in H^{3/2}_{00}(\Omega) \Longleftrightarrow \Delta u \in [H^{1/2}(\Omega)]' \quad \mathrm{and}\quad \partial_\textit{\textbf n} u\ = 0 \; \; \mathrm{on}\; \Gamma.
$$
iii) In Point i) and Point iii) above, the properties $\sqrt \varrho\, \nabla u  \in {\textit{\textbf L}}^{2}(\Omega)$ and  $\varrho^{k + 1/2}\nabla^{k + 1}u \in \textit{\textbf L}^2(\Omega)$ are a direct consequence of the harmonicity of $u$, as we can see in Theorem \ref{inegharm}.\\
iv) In a forthcoming paper, we will prove similar results in the case where the domain is of class $\mathscr{C}^{1,1}$ and the Dirichlet boundary condition $g$ belongs to $ H^2(\Gamma)$.
\end{remark}

We now give extensions of the classical Ne\v{c}as' property, that will be very useful for the study of the homogeneous Neumann problem and also for the  Dirichlet-to-Neumann operator for the Laplacian: if $u\in H^1(\Omega)$ with $\Delta u\in L^2(\Omega)$, then we have the following equivalence
\begin{equation*}\label{Necclass}
u\in H^1(\Gamma)\;  \Longleftrightarrow\;  \partial_{\textit{\textbf n}} u\in L^2(\Gamma) \;  \Longleftrightarrow \; \nabla u\in \textit{\textbf L}^2(\Gamma).
\end{equation*}
In the next theorem, see Section \ref{SectionDPL}, we relax the condition regarding the Laplacian:

\begin{theorem}  [{\bf Ne\v{c}as Property, New Version}] \label{d02-110118-th1a1}
Let 
$$
u\in H^1(\Omega)\quad \mathrm{ with} \quad \Delta u  \in\ [H^{1/2}(\Omega)]'.
$$
i) If $u\in H^1(\Gamma)$, then $\partial_{\textit{\textbf n}} u\in L^2(\Gamma)$ and
\begin{equation*}
 \Vert \partial_{\textit{\textbf n}} u\Vert_{ \textit{\textbf L}^2(\Gamma)}\ \leq\ C(\Omega)\left(\inf_{k\in \mathbb{R}}\Vert u + k \Vert_{H^1(\Gamma)} + \Vert \Delta u \Vert_{[H^{1/2}(\Omega)]'}\right), 
\end{equation*}
where the constant $C(\Omega)$ depends on the Lipschitz  constant of $\Omega.$

\noindent ii)  If $\partial_\textit{\textbf n} u\in L^2(\Gamma)$, then $u\in H^1(\Gamma)$ and we have the following estimate
\begin{equation*}
 \inf_{k\in \mathbb{R}}\Vert u + k\Vert_{H^1(\Gamma)}  \ \leq\ C(\Omega)\left( \Vert \partial_\textit{\textbf n} u\Vert _{ \textit{\textbf L}^2(\Gamma)} + \Vert  \Delta u\Vert _{[H^{1/2}(\Omega)]')}\right),
\end{equation*}
where the constant $C(\Omega)$ depends on the Lipschitz  constant of $\Omega.$\\
iii)  If $u\in H^1(\Gamma)$ or $ \partial_{\textit{\textbf n}}u\in L^2(\Gamma)$, then $u\in H^{3/2}(\Omega)$.
\end{theorem}

\begin{remark}\upshape  In a forthcoming paper, using the properties of the  Dirichlet-to-Neumann operator, we will give different regularity results for the Neumann problem.
\end{remark}
  
\section{Functional framework}  
  
 \noindent{\bf Definition}: The domain $\Omega$ is of class  $\mathscr{C}^{0,1}$, respectively  $\mathscr{C}^{k-1,1}$ with $k \geq 2$, if for every $\textit{\textbf x} \in \Gamma$, there exist a neighbourhood $V$ of $\textit{\textbf x} $ in $\mathbb{R}^N$ and a system of local charts $(\textit{\textbf y}',y_N) \in \mathbb{R}^N$  such that \\
i)  $V$ is a cylinder of the form:
$$
V =  \left\{\,(\textit{\textbf y}',y_N);\ \vert \textit{\textbf y}' \vert < \delta, \quad - b < y_N < b \right\},
$$
for some positive numbers $\delta$ and $b$,\\
ii) there exists a Lipschitz, respectively  $\mathscr{C}^{k-1,1}$, function $\psi$ defined in $B' = \{ \textit{\textbf y}'\in  \mathbb{R}^{N-1};\; \vert \textit{\textbf y}' \vert \leq \delta\}$ and such that for any $\textit{\textbf y}' \in B', $ we have $ \vert \psi (\textit{\textbf y}')\vert \leq b/2$ and
\begin{equation*}
\Omega\cap V = \left\{\,(\textit{\textbf y}',y_N)\in V;\;  y_N <  \psi (\textit{\textbf y}') \right\},\;
\Gamma\cap V =  \left\{\,(\textit{\textbf y}',y_N)\in V;\;  y_N = \psi (\textit{\textbf y}') \right\}.
\end{equation*}
Instead of neighbourhoods $V$, we can consider balls: for each point $\textit{\textbf x}_0\in\Gamma$, there exist $\delta > 0$ and 
$\xi\in C^{k-1,1}( \mathbb{R}^{N-1})$ with $\mathrm{supp}\,\xi$ compact such that, upon relabeling and reorienting the coordinates axes if necessary, we have
$$
\Omega\cap B(\textit{\textbf x}_0,\delta)\ = \ \left\{\,(\textit{\textbf x}',x_N)\in B(\textit{\textbf x}_0,\delta);\ \ x_N<\xi(\textit{\textbf x}') \right\}.
$$

Because $\Gamma$ is compact and  $\mathscr{C}^{k-1,1}$,  there exist $M$ sets    $U_1,..., U_M$, covering the boundary $\Gamma$, and $\theta_0, ..., \theta_M\in \mathscr{D}(\mathbb{R}^N)$  such that 
$$
\forall r=0,\ldots, M,\quad 0\leq \theta_r\leq 1, \ \quad \ \sum_{r=0}^M\theta_r=1\ \textrm{ in }\  D \supset \overline{\Omega},
$$
$$
\quad\forall r=1,\ldots, M, \quad \mathrm{supp}\, \theta_r \ \textrm{ compact   }\subset U_r, \quad \mathrm{supp}\, \theta_0\subset \Omega.
$$

\subsection{Spaces $H^s(\Omega)$}

In the rest of this section, we assume that $\Omega$ is a bounded Lipschitz domain of $\mathbb{R}^N$, with $N \geq 2$. Given a function $u$ in $\Omega$, we denote  $u_r = u\theta_r,$ for  $r = 0,\ldots, M$. Using a linear mapping, we transform the coordinate $\textit{\textbf x}$ into $(\textit{\textbf y}'_r, y_{rN})$; the regularity properties  in $\Omega$ or on the boundary do not change. So we can assume that the system $(\textit{\textbf y}'_r, y_{rN})$ coincides with the original system and for any $r = 1,\ldots, M$, we set
$$
u_{\xi_{r}} (\textit{\textbf x}') = u(\textit{\textbf x}', \xi_r(\textit{\textbf x}'))\quad \mathrm{or\,\, more\,\, simply}\quad u_{\xi} (\textit{\textbf x}') = u(\textit{\textbf x}', \xi(\textit{\textbf x}')), 
$$
where $ \xi\in \mathscr{C}^{k-1,1}(\mathbb{R}^{N-1})$ with $\mathrm{supp}\,\xi$ compact. 

For $1\leq j\leq N-1$, we introduce the following functions
$$
(\textit{\textbf x}',x_N)\in\Omega,\ \  v_{j}(\textit{\textbf x}',x_N)\ =\ \frac{\partial u}{\partial x_j}(\textit{\textbf x}',x_N)+\frac{\partial \xi}{\partial x_j}(\textit{\textbf x}')\frac{\partial u}{\partial x_N}(\textit{\textbf x}',x_N)
$$
$$
\textit{\textbf x}'\in\mathbb{R}^{N-1},\ \ \ v_{{j\xi}}(\textit{\textbf x}')\ =\ \frac{\partial u}{\partial x_j}(\textit{\textbf x}',\xi(\textit{\textbf x}'))+\frac{\partial \xi}{\partial x_j}(\textit{\textbf x}')\frac{\partial u}{\partial x_N}(\textit{\textbf x}',\xi(\textit{\textbf x}'))
$$
and
$$
\textit{\textbf x}'\in\mathbb{R}^{N-1},\ \ \ \omega_{\xi}(\textit{\textbf x}')\ =\ (\nabla u\cdot{\textit{\textbf n}})(\textit{\textbf x}',\xi(\textit{\textbf x}')), \quad \mathrm{with}\quad  \textit{\textbf n}\ =\ \frac{(-\nabla'\xi,1)}{\sqrt{1+|\nabla'\xi|^2}}. 
$$
Recall that $u\in H^s(\Gamma)$ with $0 \leq s \leq 1$ means that $u_{\vert \Gamma}\in L^2(\Gamma)$ and $u_{\xi}\in H^s(\mathbb{R}^{N-1})$. 

We now review the definitions of some Sobolev spaces and some important properties that will be useful later. Recall first the following Sobolev space: for $s\in \mathbb{R}$,
$$
H^s(\mathbb{R}^N)\ =\ \left\{\, v\in \mathscr{S}'(\mathbb{R}^N);\ (1+|\xi|^2)^{s/2}\widehat{v}\in L^2(\mathbb{R}^N) \right\}
$$
which is a Hilbert space for the norm:
$$
\Vert u\Vert_{H^s(\mathbb{R}^N)} = \left(\int_{\mathbb{R}^N} (1+|\xi|^2)^{s}\vert\widehat{v}\vert^2 dx\right)^{1/2}. 
$$
Here the notation $\widehat{v}$ denotes the Fourier transform of $v$. For each non negative real $s$, define
$$
H^s(\Omega)\ =\ \left\{\, v|_{\Omega};\ \ v\in H^s(\mathbb{R}^N) \right\},
$$
with the usual quotient norm
$$
\Vert u\Vert_{H^s(\Omega)} = \inf\{\Vert v\Vert_{H^s(\mathbb{R}^N)} ;\, v|_{\Omega} = u\; \mathrm{in}\; \Omega\}.
$$
If $m\in\mathbb{N}$, then
$$
H^m(\Omega)\ =\ \left\{\, v\in L^2(\Omega);\ D^\lambda v \in L^2(\Omega)\ \ \textrm{ for } 0<|\lambda|\leq m\right\},
$$
and we have the equivalence:
$$
\Vert u\Vert_{H^m(\Omega)} \simeq \left(\sum_{0\leq \vert\lambda\vert \leq m}\Vert  D^\lambda u\Vert^2_{L^2(\Omega)}\right)^{1/2}.
$$
Recall that there exits a bounded linear extension operator $P: H^s(\Omega) \rightarrow H^s(\mathbb{R}^N)$ for every non negative real number $s$. So,  by complex interpolation we have
$$
H^s(\Omega)\ =\ \left[H^m(\Omega),L^2(\Omega) \right]_\theta \quad \mathrm{with}\quad  s=(1-\theta)m\ \textrm{ and }0<\theta<1,
$$
with equivalence norms.

According to \cite{Ada} and \cite{Gri}, when $s = m + \sigma$, with $0 < \sigma < 1$, $H^s(\Omega)$ can be equipped with an equivalent and intrinsic norm
$$
\Vert u\Vert_{H^s(\Omega)} = (\Vert u\Vert_{H^m(\Omega)}^2 + \vert u\vert^2_{H^s(\Omega)})^{1/2},
$$
where
$$
\vert u\vert_{H^s(\Omega)} =  \big(\sum_{\vert\lambda\vert = m}\int_\Omega\int_\Omega\frac{\vert D^\lambda u(x) - D^\lambda u(y)\vert^2}{\vert x - y \vert^{N + 2\sigma}}dx\,dy\big)^{1/2}.
$$

\subsection{Spaces $H^s_0(\Omega)$ and  $H^s_{00}(\Omega)$}\label{introHs}

This leads us to introduce the following space
$$
H^s_0(\Omega)\ =\ \overline{\mathscr{D}(\Omega)}^{\, ||\,.\,||_{H^s(\Omega)}}\quad  \textrm{ with } s\geq 0, 
$$
\textit{i.e.}, the closure of the space $\mathscr{D}(\Omega)$ for the norm $||\,.\, ||_{H^s(\Omega)}$. Let us also recall that for any $ 0\leq s\leq 1/2$, the space $\mathscr{D}(\Omega) \textrm{ is dense in } H^s(\Omega)$. That means that $H^s(\Omega) = H^s_0(\Omega)$   for $ 0\leq s\leq 1/2.$  Moreover, we have the following properties:\\
\textit{\textbf {i}}) Let $u\in H^s_0(\Omega)$ with $0 \leq s \le 1$. Then 
\begin{equation*}\label{a2-e5}
\frac{u}{\varrho^s}\in L^2(\Omega)\quad  \textrm{ when } \quad s\neq 1/2 \qquad \mathrm{and}\qquad \Vert \frac{u}{\varrho^s} \Vert_{L^2(\Omega)} \leq C(\Omega) \vert u \vert_{H^s(\Omega)}.
\end{equation*}
\textit{\textbf {ii}}) More generally, let $u\in H^s_0(\Omega)$ with $s>0$ and such that $ s - 1/2$ is not an integer. Then
\begin{equation}\label{a2-e6}
\forall  |\lambda|\leq s, \quad \frac{D^\lambda u}{\varrho^{s-|\lambda|}}\in L^2(\Omega),
\end{equation}
with similar inequalities as above.

Let us to introduce the following space: for $ s \geq 0$
$$
 \widetilde{H}^s(\Omega)\ =\ \left\{\, v\in H^s(\Omega);\, \widetilde{v}\in H^s(\mathbb{R}^N)  \right\},
 $$
 where $ \widetilde{v}$ is the extension of $v$ by zero outside $\Omega$. The space $ \widetilde{H}^s(\Omega)$ is a Hilbert for the norm 
 $$
 \Vert u \Vert_{\widetilde{H}^s(\Omega)} =  \Vert \widetilde{u} \Vert_{H^s(\mathbb{R}^N)}
 $$
 and satisfies the following property:
 \begin{equation}\label{a2-e7}
\widetilde{H}^s(\Omega)=H^s_0(\Omega)\quad \mathrm{when}\quad s\notin\left\{1/2\right\}+\mathbb{N}.
\end{equation}

Another way to characterize the space $H^s_0(\Omega)$, for $s>1/2$ and $s\notin\left\{1/2\right\}+\mathbb{N}$, is given by
\begin{equation*}\label{a2-e8}
u\in H^s_0(\Omega)\ \Longleftrightarrow\ u\in H^s(\Omega)\ \textrm{ and }\ \frac{\partial^j u}{\partial \textit{\textbf {n}}\, ^j}=0,\ \  0\leq j\leq s-1/2,
\end{equation*}
where $\textit{\textbf {n}}$ is the outward normal vector to the boundary of $\Omega$.
For the case  $s = 3/2$, we have $H^{3/2}_0(\Omega) = H^{3/2}(\Omega)\cap H^{1}_0(\Omega)$. The interpolation between two spaces $H^s_0(\Omega)$ is somewhat different from the one between two spaces $H^s(\Omega)$. Indeed, if $s_1>s_2\geq 0$ such that $ s_1, s_2\notin\left\{1/2\right\}+\mathbb{N} $, then we have
\begin{equation*}\label{a3-e1}
[H^{s_1}_0(\Omega), H^{s_2}_0(\Omega)]_\theta=H^{(1-\theta)s_1+\theta s_2}_0(\Omega)\ \ \ \  \textrm{ if } \ (1-\theta)s_1+\theta s_2 \notin\left\{1/2\right\}+\mathbb{N} 
\end{equation*}
and 
\begin{equation}\label{a3-e2}
[H^{s_1}_0(\Omega), H^{s_2}_0(\Omega)]_\theta=H^{(1-\theta)s_1+\theta s_2}_{00}(\Omega)\ \ \ \  \textrm{ if } \ (1-\theta)s_1+\theta s_2 \in\left\{1/2\right\}+\mathbb{N} 
\end{equation}
where the space $H^s_{00}(\Omega)$ is defined as follows: For any $\mu\in\mathbb{N}$, 
\begin{equation*}\label{a3-e3}
H^{\mu+1/2}_{00}(\Omega)\ =\ \left\{\, u\in H^{\mu+ 1/2}_{0}(\Omega);\ \ \frac{D^\lambda u}{\varrho^{1/2}}\in L^2(\Omega), \ \ \forall |\lambda|=\mu \,\right\}.
\end{equation*}
This is a strict subspace of $H^{\mu+1/2}_{0}(\Omega)$ with a strictly finer topology and $\mathscr{D}(\Omega) $ is dense in  $H^{\mu+1/2}_{00}(\Omega)$  for this finer topology. 
 
The property (\ref{a2-e6}) admits a reciprocal one if $ s \notin\left\{1/2\right\}+\mathbb{N}$:
 \begin{equation*}\label{a3-e4}
 u\in H^s_0(\Omega)\ \Longleftrightarrow\ u\in  L^2(\Omega)\ \textrm{ and } \frac{D^\lambda u}{\varrho^{s-|\lambda|}}\in L^2(\Omega),\ \  \forall |\lambda|\leq s. 
 \end{equation*}
Regarding the property (\ref{a2-e7}), we have 
\begin{equation}\label{a3-e5}
\widetilde{H}^s(\Omega)=H^s_{00}(\Omega)\quad \mathrm{when}\quad s\in\left\{1/2\right\}+\mathbb{N}.
\end{equation}
So, by \eqref{a3-e2}-\eqref{a3-e5}, we have three equivalent definitions for the spaces $H^{1/2}_{00}(\Omega)$ and $H^{3/2}_{00}(\Omega)$ respectively. Later on, we will look at two other ways to define each of these two spaces.

We now have a look at their dual spaces. We set for $s\geq 0$,
\begin{equation*}\label{a3-e6}
\widetilde{H}^{-s}(\Omega)\ =\ \left[ \widetilde{H}^s(\Omega)\right]'\ \ \textrm { and } \ \ 
H^{-s}(\Omega)\ =\ \left[ {H}^s_0(\Omega)\right]',
\end{equation*}
and note that if $ s \notin\left\{1/2\right\}+\mathbb{N}$, then
$ \widetilde{H}^{-s}(\Omega)\ =\ H^{-s}(\Omega)$. And  since  $\mathscr{D}(\Omega)$ is dense in $\widetilde{H}^s(\Omega)$, then the dual spaces could be identified with subspaces of  $\mathscr{D}'(\Omega)$.

\section{Equivalent norms in fractional Sobolev spaces}\label{Preliminary results in weighted Sobolev}

In the rest of this work, we will assume that $\Omega$ is a bounded Lipschitz domain of $\mathbb{R}^N$, with $N \geq 2$, unless otherwise stated.

In Theorem 4.2 of \cite{J-K}, it is stated, for $v$ harmonic in $\Omega$, $0\leq s \leq 1$ and $k$ be a nonnegative integer, the following equivalence:
\begin{equation*}\label{normeqharm}
v\in H^{k+s}(\Omega)\Longleftrightarrow v, \nabla^k v \; \mathrm{and }\; \varrho^{1-s} \nabla^{k +1}v  \in L^2(\Omega).
\end{equation*}
In this section, we will improve and extend significantly this result and provide several equivalent norms that are very useful for establishing new regularity results for boundary value problems.

\begin{lemma}\label{densityHtilde} Let $0 < s < 1$ and 
$$
E_s(\Omega) : = \{v\in L^2(\Omega); \; \nabla v \in  [\widetilde{\textit{\textbf H}}\,^{1-s}(\Omega)]' \}.
$$
We have the following properties:\\
i) if  $1/2 < s < 1$, then
$$
\mathscr{D}(\overline{\Omega}) \quad is\; dense \;in \quad E_s(\Omega),
$$
ii) if  $0 < s \leq 1/2$, then
$$
\mathscr{D}(\Omega) \quad is\; dense \;in \quad E_s(\Omega).
$$
\end{lemma}

\begin{proof}To prove the above density results, we will use Hahn-Banach's theorem. Let $\ell\in [E_s(\Omega)]'$. So there exist $f\in L^2(\Omega)$ and $\textit{\textbf F}\in \widetilde{\textit{\textbf H}}\,^{1-s}(\Omega)$  such that 
$$
\forall v \in E_s(\Omega), \quad \left\langle \ell, v \right\rangle = \int_\Omega f v dx+  \left\langle \textit{\textbf F}, \nabla v \right\rangle_{\widetilde{\textit{\textbf H}}\,^{1-s}(\Omega)\times [\widetilde{\textit{\textbf H}}\,^{1-s}(\Omega)]'}.
$$
{\bf i) Case $1/2 < s < 1$}. Suppose that $\ell_{\vert\mathscr{D}(\overline{\Omega}) } = 0$. In order to establish the claim of the lemma, we just need to show that $\ell$ is identically zero. Indeed for any $\varphi \in \mathscr{D}(\mathbb{R}^N)$, we have
$$
0 = \left\langle \ell, \varphi_{\vert\Omega} \right\rangle = \int_{\mathbb{R}^N }(\widetilde{f} \varphi +  \widetilde{\textit{\textbf F}}\cdot \nabla \varphi)dx,
$$
where $\widetilde{f}$ and $\widetilde{\textit{\textbf F}}$ are the extensions by zero outside of $\Omega$ of $f$ and $\textit{\textbf F}$ respectively. Since $\widetilde{f} \in L^2(\mathbb{R}^N)$ and $\widetilde{\textit{\textbf F}}\in \textit{\textbf H}\,^{1-s}(\mathbb{R}^N)$ we deduce that 
$$
\widetilde{f} = \mathrm{div}\,  \widetilde{\textit{\textbf F}}\quad \mathrm{in}\;  \mathbb{R}^N.
$$
Clearly $\textit{\textbf F} \in  \textit{\textbf H}\,^{1-s}(\Omega)$ and div $\textit{\textbf F}  \in L^2(\Omega)$, so $\textit{\textbf F}\cdot \textit{\textbf n}\in H^{-1/2}(\Gamma)$. Since $\mathrm{div}\, \widetilde{\textit{\textbf F}} \in L^2(\mathbb{R}^N)$, we then deduce that $\textit{\textbf F}\cdot \textit{\textbf n}= 0$ on $\Gamma$. So the vector field $\textit{\textbf F}$ belongs to the following space
$$
\textit{\textbf H}_0^{\,1-s}(\mathrm{div}\, ; \, \Omega):=  \{\textit{\textbf F} \in  \textit{\textbf H}\,^{1-s}(\Omega); \; \mathrm{div }\,  \textit{\textbf F}  \in L^2(\Omega)\; \; \mathrm{and}\; \; \textit{\textbf F}\cdot \textit{\textbf n}= 0\; \mathrm{on} \, \Gamma\}.
$$
Now, as in the proof of Theorem 1.3 in \cite{Temam}, we can show that $\mathscr{D}(\Omega)^N $ is dense in the space $\textit{\textbf H}_0^{\,1-s}(\mathrm{div}\, ; \, \Omega)$. Let us then consider  a  sequence $(\textit{\textbf F}_k)$, for $k \in \mathbb{N}^\star$, of vector fields belonging to $\mathscr{D}(\Omega)^N $  and such that $\textit{\textbf F}_k \rightarrow \textit{\textbf F}$ in $\textit{\textbf H}_0^{\,1-s}(\mathrm{div}\, ; \, \Omega)$ when $k\rightarrow \infty$.  For every $v \in E_s(\Omega)$, we have
$$
\left\langle \ell, v \right\rangle = \lim_{k\rightarrow \infty}\big(\int_\Omega v\, \mathrm{div}\, \textit{\textbf F}_k \, dx+ \langle \nabla v,  \textit{\textbf F}_k\rangle  \big) =   0.
$$

\noindent{\bf ii) Case $0 < s \leq 1/2 $}. Assume now that $\ell_{\vert\mathscr{D}(\Omega) } = 0$. So we get the relation $f = \mathrm{div}\,  \textit{\textbf F}$ in $\Omega$.  Using again the same arguments as in in the proof of Theorem 1.3 in  \cite{Temam}, we find that 
$$
\mathscr{D}(\Omega)^N \quad \mathrm{is\; dense \;in} \quad \{\textit{\textbf F}\in  \widetilde{\textit{\textbf H}}\,^{1-s}(\Omega);\; \mathrm{div }\,  \textit{\textbf F}  \in L^2(\Omega)\}.
$$ 
And we finish the proof as above.
\end{proof}

Recall that for any $0 < s < 1$, we have the following implication:
 \begin{equation*}\label{impliGri}
 v\in H^{s}(\Omega)\quad  \Longrightarrow \quad\nabla v\in [\widetilde{\textit{\textbf H}}\,^{1-s}(\Omega)]' 
 \end{equation*} 
(see Theorem 1.4.4.6 and Remark  1.4.4.7 in \cite {Gri}). Moreover it is easy to show by interpolation that for any $0 < s < 1$ we have the following inequality:
\begin{equation*}\label{inegH1-stildeprime}
\forall v\in H^{s}(\Omega), \quad \,\Vert \nabla v\Vert_{ [\widetilde{\textit{\textbf H}}\,^{1-s}(\Omega)]' }\leq \Vert v \Vert_{H^{s}(\Omega)},
\end{equation*}
with the continuity constant equal to 1. In the theorem below, we study the validity of the inverse inequality. Before that, let us recall  the following property given in  Lemma 3.1, Chapter 6 in \cite{Necas} (see also Theorem 2 - Chapter VI in Stein \cite{Stein}):  there exists a function $\sigma$ belonging to $\mathscr{C}^\infty(\Omega) \cap \mathscr{C}^{0,1}(\overline{\Omega})$ and such that for any $x\in \Omega $ and for any multi-index $\lambda$
\begin{equation*}\label{sigma}
C_1 \varrho(x) \leq \sigma (x) \leq C_2 \varrho (x) \quad \mathrm{and}\quad \vert D^\lambda \sigma \vert \leq C \sigma^{1 - \vert  \lambda\vert}.
\end{equation*}
In the following we will use one or other of the functions $\varrho$ or $\sigma$ alternatively, depending on our needs. 

\begin{theorem}\label{c06-l1}
\noindent{i)} Let $0 \leq s \leq 1$. There exist constants $C_1> 0$  and $C_2> 0$ depending only on $s$, on the Poincar\'e constant of $\Omega$  such that for any $v\in H^s(\Omega)$
 \begin{equation}\label{d02-7118-e1}
 \inf_{K\in\, \mathbb{R}} \Vert v+K\Vert _{H^{s}(\Omega)}\leq C_1\Vert \nabla v\Vert _{ [\widetilde{\textit{\textbf H}}\,^{1-s}(\Omega)]' } \leq C_2\,\Vert  \sigma^{1-s}\nabla v\Vert _{{ \textit{\textbf L}}^2(\Omega) }.
 \end{equation}
ii) Let $v\in\mathscr{D}'(\Omega) $ and $0 \leq s \leq 1.$ Then we have the following implications 
\begin{equation}\label{implicationsigmanabla}
\sigma^{1-s}\nabla v\in  {\textit{\textbf L}}^2(\Omega) \Longrightarrow \nabla v\in  [\widetilde{\textit{\textbf H}}\,^{1-s}(\Omega)]'\Longrightarrow v \in H^s(\Omega).
\end{equation}
 \end{theorem}
 
\begin{remark}\upshape If we assume that $v \in L^2(\Omega)$, instead of $v\in\mathscr{D}'(\Omega) $, we can replace in \eqref{implicationsigmanabla} the condition $\sigma^{1-s}\nabla v\in  {\textit{\textbf L}}^2(\Omega)$ by the condition $\varrho^{1-s}\nabla v\in  {\textit{\textbf L}}^2(\Omega) $, where $\varrho$ is the distance function to the boundary of $\Omega$. 
\end{remark}
 \begin{proof} {\bf Step 1.} Firstly, recall that
 $$
 \widetilde{H}\,^{1-s}(\Omega) = H^{1-s}_0(\Omega) \quad \mathrm{when}\; s \not= 1/2 \quad \mathrm{and}\quad  \widetilde{H}\,^{1/2}(\Omega) = H^{1/2}_{00}(\Omega).
 $$
 Recall also the following estimate (see \cite{AG})
\begin{equation}\label{inegL2H-1}
  \forall  v\in L^2(\Omega), \quad \inf_{K\in\,\mathbb{R}} \Vert v+K\Vert _{L^{2}(\Omega)}\leq C_1\, \Vert \nabla v\Vert _{ \textit{\textbf H}\,^{-1}(\Omega)}
\end{equation}
and the Poincar\'e Wirtinger inequality: 
 \begin{equation}\label{inegH1L2}
   \forall  v\in H^1(\Omega), \quad \inf_{K\in\,\mathbb{R}} \Vert v+K\Vert_{H^{1}(\Omega)}\leq C_2\,\Vert \nabla v\Vert _{ \textit{\textbf L}^{2}(\Omega)},
 \end{equation}
 where the constants $C_1$ and $C_2$ depend only on the Poincar\'e constant of $\Omega$. That means that the operators
  $$
  \nabla :  L^2(\Omega){/\mathbb{R}}\rightarrow { \textit{\textbf H}\,^{-1}(\Omega)} \quad \mathrm{and}\quad  \nabla :  H^1(\Omega){/\mathbb{R}}\rightarrow { \textit{\textbf L}\,^{2}(\Omega)}
  $$
  have their image closed in $\textit{\textbf H}\,^{-1}(\Omega)$ and in $\textit{\textbf L}\,^{2}(\Omega)$ respectively. These are respectively characterized as follows:
  $$
  \textit{\textbf V}^\bot = \{\textit{\textbf f}\in \textit{\textbf H}\,^{-1}(\Omega);\; \langle \textit{\textbf f}, \textit{\textbf v}\rangle = 0,\quad \forall  \textit{\textbf v}\in \textit{\textbf V}\}
  $$
  and
  $$
  \textit{\textbf H}^\bot = \{\textit{\textbf f}\in \textit{\textbf L}\,^{2}(\Omega);\; \langle \textit{\textbf f}, \textit{\textbf v}\rangle = 0, \quad\forall  \textit{\textbf v}\in \textit{\textbf H}\},
  $$
where
$$
\textit{\textbf V} =\{ \textit{\textbf v}\in \textit{\textbf H}\,^{1}_0(\Omega); \; \mathrm{div}\,  \textit{\textbf v} = 0\} \quad \mathrm{and}\; \textit{\textbf H} =\{ \textit{\textbf v}\in \textit{\textbf L}\,^{2}(\Omega); \; \mathrm{div}\,  \textit{\textbf v} = 0, \; \textit{\textbf v}\cdot \textit{\textbf n}= 0\}.
$$
So the operators 
\begin{equation}\label{isoderham}
  \nabla :  L^2(\Omega){/\mathbb{R}}\rightarrow  \textit{\textbf V}^\bot \quad \mathrm{and}\quad  \nabla :  H^1(\Omega){/\mathbb{R}}\rightarrow \textit{\textbf H}^\bot 
\end{equation}
are isomorphisms. Observe that the first isomorphism in \eqref{isoderham}, and also the second one, is nothing more than one of the variants of De Rham's theorem:  for any 
$\textit{\textbf f}\in \textit{\textbf H}\,^{-1}(\Omega)$ satisfying the condition 
$$
\langle \textit{\textbf f}, \textit{\textbf v}\rangle = 0,\quad \forall  \textit{\textbf v}\in \textit{\textbf V},
$$
there exists $\chi \in L^2(\Omega)$, unique up an additive constant, such that $\nabla \chi = \textit{\textbf f} \, $ in $\Omega$ (see \cite{AG} for instance).\medskip

\noindent{\bf Step 2.} {\bf i)} We will prove the first inequality in \eqref{d02-7118-e1}.  Since $H^1(\Omega)$ is dense in $L^2(\Omega)$, the quotient space $H^1(\Omega){/\mathbb{R}}$ is also dense in $L^2(\Omega){/\mathbb{R}}$. So from \eqref{isoderham}, we deduce by interpolation that  the following operator
\begin{equation*}\label{grad-13}
  \nabla :  [H^1(\Omega){/\mathbb{R}},\,  L^2(\Omega){/\mathbb{R}}\, ]_\theta \rightarrow [\textit{\textbf H}^\bot, \, \textit{\textbf V}^\bot]_\theta 
\end{equation*}
is then an  isomorphism for any $ 0 < \theta < 1$. 

We denote by $\nabla^{-1}$ the inverse operator of the operator $\nabla$. In particular, from  \eqref{inegL2H-1} and \eqref{inegH1L2} we have the following inequalities:
\begin{equation}\label{grad-11}
\forall  \textit{\textbf f} \in \textit{\textbf V}^\bot, \quad \Vert\nabla^{-1} \textit{\textbf f}\,  \Vert_{ L^2(\Omega){/\mathbb{R}}} \leq C_1 \Vert \textit{\textbf f}\,  \Vert_{ \textit{\textbf H}\,^{-1}(\Omega)}
\end{equation}
and
\begin{equation}\label{grad-12}
\forall  \textit{\textbf f} \in \textit{\textbf H}^\bot, \quad \Vert \nabla^{-1} \textit{\textbf f}\,  \Vert_{ H^1(\Omega){/\mathbb{R}}} \leq C_2 \Vert\textit{\textbf f} \,\Vert_{ \textit{\textbf L}\,^{2}(\Omega)}.
\end{equation}
 Using \eqref{grad-11}-\eqref{grad-12}, we deduce by interpolation the following estimate
\begin{equation}\label{ineg3theta}
\forall  \textit{\textbf f} \in [\textit{\textbf H}^\bot, \, \textit{\textbf V}^\bot]_\theta, \quad \Vert\nabla^{-1} \textit{\textbf f} \,\Vert_{  H^{1-\theta}(\Omega){/\mathbb{R}}} \leq C(\Omega) \Vert\textit{\textbf f} \,\Vert_{[ \widetilde{\textit{\textbf H}}\,^{\theta}\Omega)]'},
\end{equation}
where $C(\Omega) =  C_1^\theta C_2^{1-\theta}$. We used here the identities
$$
 [H^1(\Omega){/\mathbb{R}},\,  L^2(\Omega){/\mathbb{R}}\, ]_\theta =  H^{1-\theta}(\Omega){/\mathbb{R}}\quad \mathrm{and}\quad [L^2(\Omega),\, H^{-1}(\Omega)]_\theta = {[ \widetilde{H}\,^{\theta}\Omega)]'} 
 $$
 and also the interpolation inequality (see Adams \cite{Ada} page 222, Berg-Lofstr$\mathrm{\ddot{o}}$m \cite{BL} Theorem 4.1.2 and Triebel \cite{Tri} Remark 3 page 63). Since
$$
 \textit{\textbf V}^\bot = \{ \nabla \chi; \; \chi \in  L^2(\Omega)\} \quad \mathrm{and}\quad  \textit{\textbf H}^\bot = \{ \nabla \chi; \; \chi \in  H^1(\Omega)\},
 $$
 we deduce that
$$
[ \textit{\textbf H}^\bot, \textit{\textbf V}^\bot]_\theta =   \{ \nabla \chi; \; \chi \in  H^{1-\theta}(\Omega)\}. 
$$
The first inequality in \eqref{d02-7118-e1} is then a consequence of \eqref{ineg3theta} when $0 < \theta < 1$,  resp. of \eqref{inegL2H-1} and \eqref{inegH1L2} for  $\theta = 0 $ or $\theta = 1$.\medskip

 \noindent{\bf ii)}  We will now prove the first implication of \eqref{implicationsigmanabla}  and the second inequality of  \eqref{d02-7118-e1}.  We can suppose $0 < s < 1$. Recall  the following property: 
 \begin{equation}\label{gr}
 [v\in\mathscr{D}'(\Omega) \quad \mathrm{ with}\quad   \nabla v \in \textit{\textbf H}\,^{-1}(\Omega)]\;  \Longrightarrow \;  v \in L^2(\Omega),
 \end{equation} 
 see Proposition 2.10 in \cite{AG}.
 
Let $v\in\mathscr{D}'(\Omega)$ with $\sigma^{1-s}\nabla v\in {\textit{\textbf L}}^2(\Omega) $. So $v\in H^1_{loc}(\Omega)$ and then for any $\varphi\in\boldsymbol{\mathscr{D}}(\Omega)$, we have for any $j = 1, \ldots , N$
  $$
 \left|\left< \frac{\partial v}{\partial x_j}, \varphi\right>\right| = \left|\int_\Omega \sigma^{1-s}\, \frac{\partial v}{\partial x_j}\frac{ \varphi}{\sigma^{1-s}} \, dx\right|\leq C \Vert \sigma^{1-s}\, \frac{\partial v}{\partial x_j}\Vert_{L^2(\Omega)}\Vert \varphi\Vert_{\widetilde{H}\,^{1-s}(\Omega)}.
 $$
 By the density of $\boldsymbol{\mathscr{D}}(\Omega)$ in $\widetilde{H}\,^{1-s}(\Omega)$, this shows that $\nabla v\in \left[\widetilde{\textit{\textbf H}}\,^{1-s}(\Omega) \right]'$ and we get the first implication of \eqref{implicationsigmanabla} and also  the second inequality in \eqref{d02-7118-e1}. Since $\left[\widetilde{\textit{\textbf H}}\,^{1-s}(\Omega) \right]'$ is included in ${\textit{\textbf H}}\,^{-1}(\Omega) $ we have in addition  $v \in L^2(\Omega)$ by using \eqref{gr}.
 \medskip 
 
\noindent{\bf iii)} To finish the proof of the theorem, we need to verify that the second implication in \eqref{implicationsigmanabla} holds. Let  $v\in\mathscr{D}'(\Omega) $ such that  $ \nabla v\in  [\widetilde{\textit{\textbf H}}\,^{1-s}(\Omega)]'$. We know from \eqref{gr} that $v\in L^2(\Omega)$, so $v\in E_{s}(\Omega)$. Using then the density results of Lemma \ref{densityHtilde} and the first inequality in \eqref{d02-7118-e1},  there exists a sequence $(v_n) \subset D(\overline{\Omega})$ such that $v_n$ tends to $v$ in $E_{s}(\Omega)$ and there exists a constant $K_n$ such that 
$$
\Vert v_n + K_n \Vert_{L^2(\Omega)}  + \vert v_n \vert_{H^s(\Omega)} \leq C \Vert \nabla v_n\Vert_{[\widetilde{H}^{1-s}(\Omega)]'}.
$$
So $v_n + K_n \rightharpoonup w$ in $H^s(\Omega)$. So $\nabla v = \nabla w$ and then $v = w + K$ (where $K = \lim_{n\rightarrow \infty} K_n$, which gives $v \in H^s(\Omega))$.
\end{proof} 
 
 Clearly we have the following results:
 
 \begin{corollary}\label{ident-espaces} We have the following properties:
\begin{equation*}
 \{v\in\mathscr{D}'(\Omega); \,\, \sigma^{1-s}\nabla v\in  {\textit{\textbf L}}^2(\Omega)\} \subset  \{v\in\mathscr{D}'(\Omega); \,\, \ \nabla v\in  [\widetilde{\textit{\textbf H}}\,^{1-s}(\Omega)]'\} = H^s(\Omega).
 \end{equation*}
  \end{corollary}

 \begin{corollary}\label{equivnormH1demi} 
 i) For any $v\in H^{s}(\Omega)$ satisfying $\int_\Omega v = 0$, with $0 \leq s \leq 1$,  we have the following inequality: 
 \begin{equation}\label{equivnorms}
 \Vert v \Vert_{H^{s}(\Omega)} \leq C(\Omega) \Vert \nabla v \Vert_{[ \widetilde{\textit{\textbf H}}\,^{1-s}\Omega)]'}.
\end{equation}
ii) For any $v\in H^{1+s}(\Omega)\cap H^{1}_0(\Omega) $ we have the following inequality: 
 \begin{equation}\label{equivnormsH3/20} 
  \Vert v \Vert_{H^{1+s}(\Omega)} \leq C(\Omega) \Vert \nabla^2 v \Vert_{[ \widetilde{\textit{\textbf H}}\,^{1-s}\Omega)]'},
   \end{equation}
 where the constants involving in \eqref{equivnorms} and in \eqref{equivnormsH3/20} depend only on $s$, on the Poincar\'e constant of $\Omega$.

\end{corollary}

\begin{proof}  Let us observe that if $v\in H^{1+s}(\Omega)\cap H^{1}_0(\Omega) $, then for any $j = 1 \ldots , N$  the derivate $\frac{\partial v} {\partial x_j}$ satisfies the assumptions of Point i). So  Point ii) is a simple consequence of Point i). To establish the inequality \eqref{equivnorms}, note that for any $v\in H^{s}(\Omega)$ satisfying the condition $\int_\Omega v = 0$, we have
\begin{equation*}
 \Vert v \Vert^2_{H^{s}(\Omega)} = \Vert v \Vert^2_{L^2(\Omega)} +   \vert v \vert^2_{H^s(\Omega)} 
 =\inf_{K\in\, \mathbb{R}} \Vert v+K\Vert^2_{H^s(\Omega)}
\end{equation*}
since
\begin{equation*}
 \inf_{K\in\, \mathbb{R}} \Vert v+K\Vert ^2_{L^2(\Omega)} =  \Vert v \Vert^2_{L^2(\Omega)} \quad \mathrm{and}\quad  \vert\,v+K\, \vert^2_{H^s(\Omega)} =  \vert\,v\, \vert^2_{H^s(\Omega)}.
  \end{equation*}
The required estimate \eqref{equivnorms} is then a consequence of the first inequality in \eqref{d02-7118-e1}.
 \end{proof}

In the same spirit, we have the following norms equivalence results.

\begin{theorem}\label{equivnormH1demi00}  We have the following equivalence norms: 
 \begin{equation}\label{equivnorms00}
 \Vert v \Vert_{H^{1/2}_{00}(\Omega)} \simeq \Vert \nabla v \Vert_{[\textit{\textbf H}\,^{1/2}(\Omega)]'} \quad for\;  v\in H^{1/2}_{00}(\Omega) 
\end{equation} 
 and in particular, 
 \begin{equation}\label{equivnormsH3/200b} 
  \Vert v \Vert_{H^{3/2}_{00}(\Omega)} \simeq \Vert \nabla^2 v \Vert_{[\textit{\textbf H}\,^{1/2}(\Omega)]'} \quad  for \; v\in H^{3/2}_{00}(\Omega).
 \end{equation}
\end{theorem}

\begin{proof} {\bf i)} Let $v\in  \mathscr{D}(\Omega)$ and $\widetilde{v}\in   \mathscr{D}(\mathbb{R}^N)$ its extension by $0$ outside of $\Omega$. Then for any $\boldsymbol{\varphi} \in \mathscr{D}(\mathbb{R}^N)$ we have
$$
   \vert \langle \nabla \, \widetilde{v},\,  \boldsymbol{\varphi}\rangle \vert  = \vert \int_{\mathbb{R}^N}  \widetilde{v}\,  \mathrm{div}\, \boldsymbol{\varphi}\, dx \vert = \vert \int_{\Omega}  v\,  \mathrm{div}\, \boldsymbol{\varphi} \, dx\vert = \vert \langle \nabla \, v,\,  \boldsymbol{\varphi}\rangle_{[\textit{\textbf H}^{\,1/2}(\Omega)]'\times \textit{\textbf H}^{\,1/2}(\Omega)} \vert,
$$
But
$$
\vert \langle \nabla \, v,\,  \boldsymbol{\varphi}\rangle_{[\textit{\textbf H}^{\,1/2}(\Omega)]'\times \textit{\textbf H}^{\,1/2}(\Omega)} \vert \leq \Vert \nabla \, v\Vert_{ [\textit{\textbf H}^{\,1/2}(\Omega)]'} \Vert \boldsymbol{\varphi} \Vert_{\textit{\textbf H}^{\, 1/2}(\mathbb{R}^N)}.
$$
Using successively the density of $\mathscr{D}(\mathbb{R}^N)$ in $H^{1/2}(\mathbb{R}^N)$ and the density of  $\mathscr{D}(\Omega)$ in $H^{1/2}_{00}(\Omega)$  we deduce that
\begin{equation*}\label{ineqHundemiprimeb}
 \Vert  \nabla \, \widetilde{v} \Vert_{\textit{\textbf H}^{\,-1/2}(\mathbb{R}^N)} \leq  \Vert  \nabla \, v \Vert_{[\textit{\textbf H}^{\,1/2}(\Omega)]'}.
\end{equation*}
\noindent{\bf ii)} Let now $v\in \mathscr{D}(\Omega)$ and $\varphi \in H^{1/2}(\Omega)$. So
$$
\int_\Omega \varphi \,  \partial_j v \, dx =  \int_{\mathbb{R}^N} (P \varphi) \,    \partial_j \widetilde{v} \, dx,
$$
where $P$ is any continuous extension operator from $H^{\,1/2}(\Omega)$ into $H^{1/2}(\mathbb{R}^N)$. Using then the density of $\mathscr{D}(\Omega)$ in $H^{1/2}_{00}(\Omega)$, the continuity of the extension operator by zero from $ H^{1/2}_{00}(\Omega)$ into $H^{-1/2}(\mathbb{R}^N)$ and the continuity of the partial derivative operator $\partial_j$ from $  H^{1/2}_{00}(\Omega)$ into  $[H^{1/2}(\Omega)]'$, we deduce the following relation:  for any $\varphi \in H^{1/2}(\Omega)$  and $v\in H^{1/2}_{00}(\Omega)$
$$
 \langle \partial_j v,\, \varphi\rangle_{[H^{1/2}(\Omega)]'\times H^{1/2}(\Omega)} = \langle \partial_j \widetilde{v},\, P\varphi\rangle_{[H^{-1/2}(\mathbb{R}^N)\times H^{1/2}(\mathbb{R}^N)}.
$$
As a consequence we have
\begin{equation*}
\begin{array}{rl}
\Vert \nabla \, v\Vert_{ [\textit{\textbf H}^{\,1/2}(\Omega)]'} &= \displaystyle \sup_{\boldsymbol{\varphi}\in \textit{\textbf H}^{\,1/2}(\Omega), \, \boldsymbol{\varphi} \not= {\bf 0}}\dfrac{ \vert \langle \nabla \, v,\,  \boldsymbol{\varphi}\rangle_{[\textit{\textbf H}^{\,1/2}(\Omega)]'\times \textit{\textbf H}^{\,1/2}(\Omega)} \vert}{ \Vert \boldsymbol{\varphi} \Vert_{\textit{\textbf H}^{\, 1/2}(\Omega)}}\\
& \leq C  \displaystyle  \sup_{\boldsymbol{\varphi}\in \textit{\textbf H}^{\,1/2}(\Omega), \, \boldsymbol{\varphi} \not= {\bf 0}}\dfrac{ \vert \langle \nabla \, \widetilde{v},\,  P\boldsymbol{\varphi}\rangle\vert}{ \Vert P\boldsymbol{\varphi} \Vert_{\textit{\textbf H}^{\, 1/2}(\mathbb{R}^N)}} \\\\
& \leq  C \Vert  \nabla \, \widetilde{v} \Vert_{\textit{\textbf H}^{\,-1/2}(\mathbb{R}^N)}.
\end{array}
\end{equation*}
We have thus established the following equivalence:
\begin{equation}\label{equivnormsOmRn}
 \Vert \nabla v \Vert_{[\textit{\textbf H}\,^{1/2}(\Omega)]'} \simeq  \Vert  \nabla \, \widetilde{v} \Vert_{\textit{\textbf H}^{\,-1/2}(\mathbb{R}^N)}
 \quad \mathrm{for}\;  v\in H^{1/2}_{00}(\Omega). 
 \end{equation} 

 \noindent{\bf iii)} Recall that 
 $$
 \Vert q \Vert_{H^{-1}(\mathbb{R}^N)} + \Vert \nabla q \Vert_{\textit{\textbf H}\,^{-1}(\mathbb{R}^N)}  \simeq  \Vert q \Vert_{L^2(\mathbb{R}^N)} \quad \mathrm{for}\;  q \in L\; ^2(\mathbb{R}^N).
$$
As   
 $$
 \Vert q \Vert_{H^1(\mathbb{R}^N)} =  (\Vert q \Vert^2_{L^{2}(\mathbb{R}^N)} + \Vert \nabla q \Vert^2_{\textit{\textbf L}\,^{2}(\mathbb{R}^N)})^{1/2}  
$$
we get by interpolation, or by using the Fourier transforms,  the following equivalence norms:
$$
 \Vert q \Vert_{H^{-1/2}(\mathbb{R}^N)} + \Vert \nabla q \Vert_{\textit{\textbf H}\,^{-1/2}(\mathbb{R}^N)}  \simeq  \Vert q \Vert_{H^{1/2}(\mathbb{R}^N)}  \quad \mathrm{for}\;  q \in H^{1/2}(\mathbb{R}^N),
$$
and also the following one:
$$
 \Vert q \Vert_{L^{2}(\mathbb{R}^N)} + \Vert \nabla q \Vert_{\textit{\textbf H}\,^{-1/2}(\mathbb{R}^N)}  \simeq   \Vert q \Vert_{H^{1/2}(\mathbb{R}^N)} \quad \mathrm{for}\;  q \in H^{1/2}(\mathbb{R}^N).
$$

\noindent{\bf iv)} Now using the last above equivalence norms and \eqref{equivnormsOmRn}, we deduce that
$$
\Vert v\Vert_{H^{1/2}_{00}(\Omega)} = \Vert \widetilde{v}\Vert_{H^{1/2}(\mathbb{R}^N)}\simeq   \Vert v \Vert_{L^{2}(\Omega)} + \Vert  \nabla \, v \Vert_{[\textit{\textbf H}^{\,1/2}(\Omega)]'} \quad \mathrm{for}\;  v\in H^{1/2}_{00}(\Omega). 
$$
However, since the embedding $H^{1/2}_{00}(\Omega)\hookrightarrow L^2(\Omega)$ is compact, we can prove that for any $v\in H^{1/2}_{00}(\Omega)$
$$
\Vert v \Vert_{L^{2}(\Omega)} \leq C \Vert  \nabla \, v \Vert_{[\textit{\textbf H}^{\,1/2}(\Omega)]'}
$$
and then the required equivalence of norms \eqref{equivnorms00}.

\noindent{\bf v)} Applying \eqref{equivnorms00}, we deduce that we have
\begin{equation*}\label{equivnormsH3/200} 
  \Vert \nabla v \Vert_{\textit{\textbf H}\,^{1/2}_{00}(\Omega)} \simeq \Vert \nabla^2 v \Vert_{[\textit{\textbf H}\,^{1/2}(\Omega)]'}\quad \mathrm{for}\;  v\in H^{3/2}_{00}(\Omega).
 \end{equation*}
The equivalence of norms \eqref{equivnormsH3/200b} is then a consequence of the following:
\begin{equation*}\label{equivnormsGH3/200} 
  \Vert \nabla v \Vert_{\textit{\textbf H}\,^{1/2}_{00}(\Omega)} \simeq \Vert  v \Vert_{H^{3/2}_{00}(\Omega)}\quad  \mathrm{for}\; v\in H^{3/2}_{00}(\Omega).
 \end{equation*}
\end{proof}

Let us define the following space
\begin{equation*}\label{spaceK_0}
K_2(\Omega) = \left\{ v \in L^2(\Omega);\ \varrho^{2}\Delta v \in
L^{2}(\Omega) \right\},
\end{equation*}
which is a Hilbert space for the norm 
$$
\Vert v \Vert_{K_2(\Omega)} = (\Vert v \Vert^2_{L^2(\Omega)} + \Vert \varrho^{2} \Delta v\Vert^2_{L^{2}(\Omega)})^{1/2}.
$$

\begin{lemma}\label{DensityK_0} The space 
$$
\mathscr{D}\big(\overline{\Omega}\big)\quad  is \, \, dense\, \,  in \; \; K_2(\Omega).
$$
\end{lemma}

\begin{proof}
Let $\ell \in \left[K_2(\Omega)\right]'$ be such that 
$$
\forall v \in \mathscr{D}(\overline\Omega),\ \ \ \left<\, \ell, v\, \right>\ =\ 0.
$$
We know that there exist $f \in L^2(\Omega) $ and $g \in L^2_{\varrho^{-2}}(\Omega)$ such that for any $v\in K_2(\Omega)$, 
$$
\left<\, \ell, v\, \right>\, =\ \int_\Omega f v \, dx + \int_\Omega g\Delta  v \, dx.
$$
Let $\widetilde{f}\in L^2(\mathbb{R}^N)$ and $\widetilde{g}\in L^{2}(\mathbb{R}^N)$ the extension functions by zero of respectively $f$ and $g$. Then for any $v\in \mathscr{D}(\mathbb{R}^N)$, we have
$$
\ \int_{\mathbb{R}^N}  \widetilde{f} v \, dx + \int_{\mathbb{R}^N}  \widetilde{g}\Delta  v \, dx = 0,
$$
\textit{i.e.} $- \Delta  \widetilde{g}\ =\ \widetilde{f}$ in $\mathbb{R}^N$. Consequently, we deduce that $\widetilde{g}\in H^{2}(\mathbb{R}^N)$.  It means that $g\in H^{2}_{0}(\Omega)$. As  $\mathscr{D}(\Omega)$ is dense in $H^{2}_{0}(\Omega)$, there exists $g_k\in \mathscr{D}(\Omega)$ such that 
$$
g_k\  \ \longrightarrow \ \ g \ \ \textrm{ in } \ H^{2}(\Omega).
$$
Then, for any $v\in K_2(\Omega)$, we have
$$
\left<\, \ell, v\, \right>\, =\ \lim_{k\rightarrow\infty} (- \int_\Omega v\Delta g_k \, dx  + \int_\Omega g_k\Delta  v \, dx) = 0,
$$
which ends the proof.
\end{proof}

 \begin{remark}\upshape In fact we will see below that the space $\mathscr{D}(\Omega)$ is dense in $K_2(\Omega)$.
 \end{remark}

Let us introduce the following space:
$$
  \mathscr{Q}^2_{0}(\Omega)\ =\ \left\{ v\in L^2(\Omega);\ \ \varrho \nabla v\in \textit{\textbf L}^2(\Omega), \;    \varrho^2 \nabla^2 v\in \textit{\textbf L}^2(\Omega\right\}.
$$
Recall that  $\mathscr{D}(\Omega)$ is dense in $\mathscr{Q}^2_{0}(\Omega)$, the proof is similar to that of  Proposition II.6.2  in \cite{Lions}.\smallskip

We are now in position to state, more generally that in proposition above, some converse implications of Theorem \ref{c06-l1}.

\begin{e-proposition}\label{RegulDelta} We have the following identity
\begin{equation*}\label{Q20K0}
\mathscr{Q}^2_{0}(\Omega) = K_2(\Omega),
\end{equation*}
algebraically and topologically. In particular if $v$ is a harmonic function,  for any positive integer $k$ we have the following property
\begin{equation}\label{Hm-sb}
v\in L^2(\Omega)\Longrightarrow \sigma^{ k}\nabla^{k} v\in  {\textit{\textbf L}}^2(\Omega)\quad \mathrm{for}\; k = 1, 2
\end{equation}
and
\begin{equation*}\label{Hm-sa}
\Vert\sigma^{ k}\nabla^{k} v\Vert_{ {\textit{\textbf L}}^2(\Omega)} \leq C\Vert  v\Vert_ {L^2(\Omega)},
\end{equation*}
where the constant $C$ depends on the Lipschitz constant of $\Omega$.
\end{e-proposition}

\begin{proof} {\bf Step 1.}  We claim that 
\begin{equation}\label{inegQ20K0}
\forall v\in \mathscr{Q}^2_{0}(\Omega), \quad \Vert  v  \Vert_{  \mathscr{Q}^{2}_{0}(\Omega)} \leq  C(\Vert v  \Vert_{L^2(\Omega)} +\Vert \varrho^{2} \Delta v  \Vert_{L^2(\Omega)}). 
\end{equation}
Since $\mathscr{D}(\Omega)$ is dense in $\mathscr{Q}^2_{0}(\Omega)$, it suffices to prove this inequality for any $v\in \mathscr{D}(\Omega)$. 

Integrating by parts, we get for any $v \in \mathscr{D}(\Omega)$:
$$
\int_\Omega \vert \varrho \nabla v\vert^2 \, dx = - 2\int_\Omega \varrho v \nabla\varrho \cdot \nabla v \, dx - \int_\Omega \varrho^{2} v \Delta v \, dx.
$$
Using Cauchy-Schwarz inequality, we obtain
\begin{equation*}
 \Vert  \varrho \nabla v \Vert_{  \textit{\textbf L}^2(\Omega)} \leq  C(\Vert   v \Vert_{L^2(\Omega)} +\Vert \varrho^{2} \Delta v  \Vert_{L^2(\Omega)}). 
\end{equation*}
Integrating two times by parts
\begin{equation*}
\begin{array}{rl}
\!\!\!\Vert \varrho^{2} \nabla^2  v \Vert^2_{\textit{\textbf L}^2(\Omega)} & =  - 4 \displaystyle\int_\Omega \varrho^{3}  \dfrac{\partial \varrho}{\partial x_i} \dfrac{\partial v}{\partial x_j}   \dfrac{\partial^2 v}{\partial x_i\partial x_j} \, dx  -  \int_\Omega \varrho^{4}  \dfrac{\partial v}{\partial x_j} \dfrac{\partial \Delta  v}{\partial x_j} \, dx\\\\
&=  - 4\displaystyle \int_\Omega \varrho^{3}  \dfrac{\partial \varrho}{\partial x_i} \dfrac{\partial v}{\partial x_j}   \dfrac{\partial^2 v}{\partial x_i\partial x_j} \, dx + 4 \int_\Omega \varrho^{3}  \dfrac{\partial \varrho}{\partial x_j} \dfrac{\partial v}{\partial x_j}  \Delta v \, dx\;  + \displaystyle \int_\Omega \varrho^{4} \vert \Delta v \vert^2 \, dx\\\\
&  \leq  C( \Vert  \varrho \nabla v \Vert_{  L^2(\Omega)} \Vert \varrho^{2} \nabla^2  v \Vert_{L^2(\Omega)} ) + \Vert \varrho^{2} \Delta  v \Vert^2_{L^2(\Omega)} 
\end{array}
\end{equation*}
and using the previous inequality and Cauchy-Schwarz inequality, we  get 
\begin{equation*}
\Vert \varrho^{2} \nabla^2  v \Vert_{\textit{\textbf L}^2(\Omega)}  \leq  C(\Vert   v  \Vert_{L^2(\Omega)} +\Vert \varrho^{2} \Delta  v \Vert_{L^2(\Omega)}). 
\end{equation*}
We finally deduce the required inequality.\medskip

\noindent{\bf Step 2.}  It just remains to prove the inclusion of $K_2(\Omega)$ into $\mathscr{Q}^2_{0}(\Omega) $. So let $v\in K_2(\Omega)$ and $v_k \in \mathscr{D}\big(\overline{\Omega}\big)$ such that $v_k \longrightarrow v$ in $K_2(\Omega)$. From \eqref{inegQ20K0}, we get the following estimate:
\begin{equation*}
\Vert  v_k \Vert_{  \mathscr{Q}^{2}_{0}(\Omega)} \leq  C(\Vert v_k  \Vert_{L^2(\Omega)} +\Vert \varrho^{2} \Delta v_k  \Vert_{L^2(\Omega)}).
\end{equation*}
Therefore, the sequence $(v_k)_k$ is bounded in $ \mathscr{Q}^{2}_{0}(\Omega)$ and $v_k \rightharpoonup v^\star $  in  $\mathscr{Q}^{2}_{0}(\Omega)$ and $v^\star = v \in \mathscr{Q}^{2}_{0}(\Omega)$.
 \end{proof}
 
\begin{theorem}\label{inegharm} Let $v\in\mathscr{D}'(\Omega) $ be such that $\Delta v = 0$ in $\Omega$, $m$ a non negative integer and $0\leq \theta < 1$ a  real number. Then we have the following properties:\\
 For any nonnegative integer $k$,
\begin{equation*}\label{Hm-s}
v\in H^{m+ \theta}(\Omega)\Longrightarrow \sigma^{1-\theta + k}\nabla^{m + 1 + k} v\in  {\textit{\textbf L}}^2(\Omega)
\end{equation*}
and
\begin{equation*}\label{Hm-sa}
\Vert \sigma^{1-\theta + k}\nabla^{m + 1 + k} v\Vert_{ {\textit{\textbf L}}^2(\Omega)} \leq C(\Omega)\Vert  v\Vert_ {H^{m + \theta }(\Omega)},
\end{equation*}
where $C(\Omega)$ depends only on the Lipschitz  constant of $\Omega$.
\end{theorem}

\begin{proof} {\bf  i)} Suppose $\theta = 0$. If $m = 0$ the result is given by Proposition \ref{RegulDelta}.  If $m$ is a positive integer we reitere the same reasoning on the derivatives of v of order $m$.\smallskip

\noindent{\bf ii)} Assume $0 < \theta < 1$. It suffices to prove the result for $m = 0$. So let $v$ be harmonic. Recall that 
$$
 [H^1(\Omega), L^2(\Omega)]_{1-\theta} =  H^\theta(\Omega) \; \mathrm{and}\;  [\mathscr{H}\cap H^1(\Omega), \mathscr{H}\cap L^2(\Omega)]_{1-\theta} =  \mathscr{H}\cap H^\theta(\Omega),
$$
where $\mathscr{H}$ is the space of harmonic functions in $\Omega$ (see \cite{J-K} page 183 for the last identity).
But we know from \eqref{Hm-sb} that
$$
v\in L^{2}(\Omega)\Longrightarrow \sigma\nabla v\in  {\textit{\textbf L}}^2(\Omega) \quad \mathrm{with}\quad \Vert \sigma\nabla v \Vert_{{\textit{\textbf L}}^2(\Omega)} \leq C_1 \Vert v \Vert_{L^{2}(\Omega)}
$$
and
$$
v\in H^{1}(\Omega)\Longrightarrow \nabla v\in  {\textit{\textbf L}}^2(\Omega)\quad \mathrm{with}\quad \Vert \nabla v \Vert_{{\textit{\textbf L}}^2(\Omega))} \leq  \Vert v \Vert_{H^{1}(\Omega)},
$$
where $C_1$ depends on the Lipschitz  constant of $\Omega$. So by interpolation we get the following implication
$$
v\in H^\theta(\Omega)\Longrightarrow \sigma^{1-\theta}\nabla v\in {\textit{\textbf L}}^2(\Omega),
$$
with the corresponding estimate. Similarly
$$
v\in L^{2}(\Omega)\Longrightarrow \sigma^2\nabla^2 v\in  {\textit{\textbf L}}^2(\Omega) \quad \mathrm{and}\quad v\in H^{1}(\Omega)\Longrightarrow \sigma \nabla^2 v\in  {\textit{\textbf L}}^2(\Omega).
$$
We get again by interpolation the following implication
$$
v\in H^\theta(\Omega)\Longrightarrow \sigma^{2-\theta}\nabla^2 v\in  L^2(\Omega),
$$
with the corresponding estimate.\end{proof}

\section{Traces in the limit cases $H^{1/2}(\Omega)$ and $H^{3/2}(\Omega)$ and for non smooth functions} \label{traces}
The questions of traces of functions belonging to Sobolev spaces are fundamental in the study of boundary value problems. Classically, we know that the linear mapping $\gamma: u \mapsto u_{\vert \Gamma}$ is continuous from $H^s(\Omega)$ into $H^{s-1/2}(\Gamma)$ for $1/2< s < 3/2$ and this property is wrong for $s = 1/2$ or $s = 3/2$. Moreover if $u\in H^s(\Omega)$  for $s > 3/2$ (resp. $u\in H^{3/2}(\Omega)$), then $u_{\vert \Gamma}\in H^1(\Gamma)$ (resp. $u_{\vert \Gamma}\in H^{1-\varepsilon}(\Gamma)$ for any $\varepsilon > 0$). We will investigate in this section the crucial limit cases $s = 1/2$, $s = 3/2$ and what additional condition should be added for them to have a trace.

\subsection{Traces in the limit cases $H^{1/2}(\Omega)$ and $H^{3/2}(\Omega)$}

Recall that for any  $v\in H^{s}(\Omega)$, with $0 < s < 1$, we have $\nabla v \in \textit{\textbf H}\, ^{s-1}(\Omega) $ if $s \neq 1/2$ and $\nabla v \in [\textit{\textbf H}\, ^{1/2}_{00}(\Omega)]' $ if $s = 1/2$, this last dual space being bigger that the dual space $[\textit{\textbf H}\, ^{1/2}(\Omega)]' $ (see Theorem 1.4.4.6 and Proposition 1.4.4.8 in \cite{Gri}).

\begin{lemma}\label{DensityDOmegabarH1/2GradH1/2'}
The space $\mathscr{D}(\overline\Omega)$ is dense in the following space:
\begin{equation}\label{defEnablaOmega}
 E(\nabla;\, \Omega)\ =\ \left\{\,v\in H^{1/2}(\Omega);\ \nabla v\in  [\textit{\textbf H}^{\, 1/2}(\Omega)]'\,\right\}.
\end{equation}
\end{lemma}

\begin{proof} The proof of the density of  $\mathscr{D}(\overline\Omega)$ in $E(\nabla;\, \Omega)$ is similar to that of $\mathscr{D}(\overline\Omega)$ in $H^{1}(\Omega)$, but little bit more complicated. It suffices to consider the case where $\Omega = \mathbb{R}^N_+$ is the half space. 
\medskip

\noindent{\bf Step 1}. {\it We will prove that the functions of $E(\nabla;\, \Omega)$ with compact support is dense in $E(\nabla;\, \Omega)$.} Let $\psi \in \mathscr{D}(\mathbb{R}^N)$, with 
\begin{equation*} 
\psi(x) = \begin{cases} 1 \quad \; \mathrm{if}\;\;\vert x \vert  \leq 5/4\\
0 \quad \; \mathrm{if}\;\;\vert x \vert  \geq 7/4
\end{cases}
\end{equation*}
and define 
$$
\mathrm{for\, \, any} \, k\in \mathbb{N}^*,\quad \psi_k (x) = \psi(x/k).
$$
For $v\in  H^{1/2}(\mathbb{R}^N_+)$, posing $v_k = \psi_k v$, we can prove by some direct calculations the following estimate:
\begin{equation}\label{estimH1/2H1/2primedemiespace}
\Vert v_k - v \Vert_{H^{1/2}(\mathbb{R}^N_+)}\leq C\big( \Vert v\Vert_{H^{1/2}(\mathbb{R}^N_+\cap B_k^c)}  + \frac{1}{\sqrt k}  \Vert v\Vert_{L^{2}(\mathbb{R}^N_+)} \big),
\end{equation}
where $B_k^c$ is the the complementary of $B_k$ in the whole space. Besides for any $\varphi \in \mathscr{D}(\mathbb{R}^N_+)$ and $j = 1, \ldots, N$, we have
\begin{equation*}
\langle \frac{\partial}{\partial x_{j}}( v_k - v),\, \varphi \rangle = \langle \frac{\partial v}{\partial x_{j}}, \, \psi_k \varphi - \varphi \rangle + \int_{\mathbb{R}^N_+}v \varphi \frac{\partial \psi_k}{\partial x_{j}} \, dx
\end{equation*}
and then by using \eqref{estimH1/2H1/2primedemiespace}, we get
\begin{equation*}\label{inegbracketH12H1/2prime}
\begin{array}{rl}
 \vert \langle \frac{\partial}{\partial x_{j}}( v_k - v),\, \varphi \rangle \vert & \leq C  \Vert \frac{\partial v}{\partial x_{j}}\Vert_{ [\textit{\textbf H}^{\, 1/2}(\mathbb{R}^N_+)]'} \big( \Vert \varphi_{H^{1/2}(\mathbb{R}^N_+\cap B_k^c)}  + \frac{1}{\sqrt k}  \Vert \varphi\Vert_{L^{2}(\mathbb{R}^N_+)} \big) \\\\
&  + \; \frac{C}{k} \Vert v \Vert_{L^{2}(\mathbb{R}^N_+)}\Vert \varphi \Vert_{L^{2}(\mathbb{R}^N_+)},
\end{array}
\end{equation*}
Hence, 
$$
\frac{\partial v_k}{\partial x_{j}}  \rightharpoonup \frac{\partial v}{\partial x_{j}} \quad \mathrm{in}\; [H^{\, 1/2}(\mathbb{R}^N_+)]'.
$$
Our goal is to prove the strong convergence. For that, we observe that $\frac{\partial v_k}{\partial x_{j}} = \psi_k \frac{\partial v}{\partial x_{j}} + v\frac{\partial \psi_{k}}{\partial x_{j}}$ and $v\frac{\partial \psi_{k}}{\partial x_{j}} \rightarrow 0 $ in $L^2(\mathbb{R}^N_+)$ and then in $ [H^{\, 1/2}(\mathbb{R}^N_+)]'$. In addition, since for any $\varphi \in H^{1/2}(\mathbb{R}^N_+)$
$$
 \vert \langle \psi_k \frac{\partial v}{\partial x_{j}} ,\, \varphi \rangle\vert = \vert \langle \frac{\partial v}{\partial x_{j}} ,\, \psi_k\varphi \rangle \vert  \leq \Vert  \frac{\partial v}{\partial x_{j}} \Vert_{[H^{\, 1/2}(\mathbb{R}^N_+)]'} \Vert \psi_k \varphi\Vert_{H^{\, 1/2}(\mathbb{R}^N_+)}
$$
we have the following estimate
$$
\Vert \psi_k \frac{\partial v}{\partial x_{j}} \Vert_{[H^{\, 1/2}(\mathbb{R}^N_+)]'} \leq \Vert  \frac{\partial v}{\partial x_{j}} \Vert_{[H^{\, 1/2}(\mathbb{R}^N_+)]'} \sup_{\varphi\in H^{1/2}(\mathbb{R}^N_+), \varphi \not= 0}\frac{\Vert \psi_k \varphi\Vert_{H^{\, 1/2}(\mathbb{R}^N_+)}}{\Vert  \varphi\Vert_{H^{\, 1/2}(\mathbb{R}^N_+)}}.
$$
As 
$$
\limsup_{k\rightarrow \infty}\Vert \psi_k \frac{\partial v}{\partial x_{j}} \Vert_{[H^{\, 1/2}(\mathbb{R}^N_+)]'} \leq \Vert \frac{\partial v}{\partial x_{j}} \Vert_{[H^{\, 1/2}(\mathbb{R}^N_+)]'},
$$
we have also the same inequality for the norm of $\frac{\partial v_k}{\partial x_{j}}$ and then we deduce the desired strong convergence.\medskip

\noindent{\bf Step 2}.  {\it Extension to} $ \mathbb{R}^N$. It follows from Step 1 that we can suppose, without loss of generality, that $v\in H^{1/2}(\mathbb{R}^N_+)$ with compact support.

For $h > 0$  we set $\tau_h v(\textit{\textbf x})  = v_h (\textit{\textbf x}) = v(\textit{\textbf x}', x_N + h)$ and we introduce the following function
\begin{equation*}
\alpha_h(\textit{\textbf x}) = \begin{cases} 1 \quad \; \mathrm{if}\;\; x_N > 0\\
0 \quad \; \mathrm{if}\;\; x_N < - h
\end{cases}
\end{equation*}
with $\alpha_h \in \mathscr{C}^1(\mathbb{R}^N)$. We set $w_h = \alpha_h \tau_h Pu$, where $P :  H^{1/2}(\mathbb{R}^N_+) \rightarrow H^{1/2}(\mathbb{R}^N)$ is a bounded linear extension operator. Clearly,  if $v \in H^{1/2}(\mathbb{R}^N_+)$, using Lebesgue's dominated convergence theorem, then we have $  v_h \rightarrow v$ in $H^{1/2}(\mathbb{R}^N_+)$ and $  w_h{_{\vert \mathbb{R}^N_+}} \rightarrow v$ in $H^{1/2}(\mathbb{R}^N_+)$  as  $h  \rightarrow 0$. Moreover, for any $\varphi \in \mathscr{D}(\mathbb{R}^N_+)$ and $j = 1, \ldots, N$,
\begin{equation*}
\vert \langle \frac{\partial w_h}{\partial x_{j}} ,\, \varphi \rangle \vert = \vert \langle \frac{\partial u}{\partial x_{j}},\,   \tau_{-h} \varphi\rangle \vert  \leq \Vert \frac{\partial u}{\partial x_{j}}\Vert_{[H^{\, 1/2}(\mathbb{R}^N_+)]'} \Vert \varphi \Vert_{H^{\, 1/2}(\mathbb{R}^N_+)}.
\end{equation*}
Using the density of $\mathscr{D}(\mathbb{R}^N_+)$ in $H^{\, 1/2}(\mathbb{R}^N_+)]$, we deduce that 
\begin{equation*}
\frac{\partial w_h}{\partial x_{j}} \in [H^{\, 1/2}(\mathbb{R}^N_+)]'\quad \mathrm{and}\quad \Vert \frac{\partial w_h}{\partial x_{j}}\Vert_{[H^{\, 1/2}(\mathbb{R}^N_+)]'} \leq  \Vert \frac{\partial u}{\partial x_{j}}\Vert_{[H^{1/2}(\mathbb{R}^N_+)]'}
\end{equation*}
(where the last inequality can be obtained by interpolation between $L^2(\mathbb{R}^N_+)$ and $H^{-1}(\mathbb{R}^N_+)$).
Besides, for any $\varphi \in \mathscr{D}(\mathbb{R}^N_+)$ and $j = 1, \ldots, N$, we have even if it means extending $\varphi$ by zero outside the half-space,
\begin{equation*}
\vert \langle \frac{\partial w_h}{\partial x_{j}} -  \frac{\partial u}{\partial x_{j}} ,\, \varphi \rangle \vert = \vert\langle \frac{\partial u}{\partial x_{j}},\,   \tau_{-h}\varphi - \varphi \rangle \vert  \leq  \Vert \frac{\partial u}{\partial x_{j}}\Vert_{[H^{1/2}(\mathbb{R}^N_+)]'} \Vert   \tau_{-h}\varphi - \varphi \Vert_{H^{\, 1/2}(\mathbb{R}^N_+)} 
\end{equation*}
where the last norm above tends to $0$  when $h\rightarrow 0$. That gives the strong convergence  
$$
 \frac{\partial w_h}{\partial x_{j}} \rightarrow \frac{\partial u}{\partial x_{j}} \qquad \mathrm{ in }\; [H^{1/2}(\mathbb{R}^N_+)]'.
$$
\medskip
\noindent{\bf Step 3}.  {\it Regularization}. To finish, we will approximate $w_h$, with $h$ fixed, by the functions $\varphi_k = w_h\star \varrho_k$, where we use the sequence of mollifiers $(\varrho_k)_k$. It is easy to verify that 
$$
\varphi_k \rightarrow w_h \qquad \mathrm{ in }\; H^{1/2}(\mathbb{R}^N)  \qquad \mathrm{ and } \qquad  \frac{\partial \varphi_k}{\partial x_{j}} \rightarrow  \frac{\partial w_h }{\partial x_{j}} \qquad \mathrm{ in }\; [H^{1/2}(\mathbb{R}^N )]'
$$ 
as  $k  \rightarrow \infty$. 
\end{proof}

\begin{theorem}\label{TracesH1demigradH1demiprime} i) The linear mapping $\gamma: u \mapsto u_{\vert \Gamma}$ defined on $\mathscr{D}(\overline{\Omega})$ can be extended by continuity to a linear and continuous mapping, still denoted $\gamma$, from $E(\nabla;\, \Omega)$ into $L^2(\Gamma)$. \\
ii) The kernel of  $\gamma: u \mapsto u_{\vert \Gamma}$ from $E(\nabla;\, \Omega)$ into $L^2(\Gamma)$ is equal to $H^{1/2}_{00}(\Omega)$.

\end{theorem}

\begin{proof} i) For any $v\in \mathscr{D}(\overline{\Omega})$ 
\begin{equation*}
 \int_\Gamma (\textbf{\textit h}\cdot\textbf{\textit n})\vert v \vert^2 \, dx = 2 \int_\Omega v\nabla v \cdot \textbf{\textit h}  \, dx+  \int_\Omega \vert v \vert^2 \mathrm{div}\,  \textbf{\textit h} \, dx,
\end{equation*}
where $\textbf{\textit h}\in \mathscr{C}^\infty(\overline{\Omega})$ is such that   $\textbf{\textit h} \cdot\textbf{\textit n}\geq \alpha > 0$ a.e  on $\Gamma$ (see Lemma 1.5.1.9 in \cite{Gri}).
Consequently, we have the following estimate:
\begin{equation*}
 \Vert v \Vert^2_{L^2(\Gamma)}\leq C (\Vert \nabla v \Vert_{ [\textbf{\textit H}\, ^{1/2}(\Omega)]'} \Vert v \Vert_{H^{1/2}(\Omega)} + \Vert v \Vert^2_{L^2(\Omega)} ),
\end{equation*}
which means that
\begin{equation*}
 \Vert v \Vert_{L^2(\Gamma)}\leq C \Vert  v \Vert_{E(\nabla;\, \Omega)}.
\end{equation*}
The required property is finally a consequence of the density of $\mathscr{D}(\overline{\Omega})$ in $E(\nabla;\, \Omega)$. \smallskip

\noindent ii) Observe that $H^{1/2}_{00}(\Omega)$ is included in $E(\nabla;\, \Omega)$. So by using the density of $\mathscr{D}(\Omega)$ in $H^{1/2}_{00}(\Omega)$,  we have the following inclusion: $H^{1/2}_{00}(\Omega) \subset Ker \, \gamma$.  

Conversely, let $u \in E(\nabla;\, \Omega)$ with $u = 0$ on $\Gamma$.  With the same calculations as in the proof of the first point  of Theorem \ref{equivnormH1demi00}, the extension by $0$ of $u$ outside of $\Omega$ satisfies: for any $j = 1, \ldots, N$ any $\varphi \in \mathscr{D}(\mathbb{R}^N)$
$$ 
\langle  \frac{\partial \widetilde{u}}{\partial x_{j}}, \varphi\rangle =  - \int_{\mathbb{R}^N}  \widetilde{u}\frac{\partial \varphi}{\partial x_{j}} \, dx =  - \int_{\Omega}  u\frac{\partial \varphi}{\partial x_{j}} \, dx = \langle  \frac{\partial u}{\partial x_{j}}, \varphi\rangle_{ [H^{\, 1/2}(\Omega)]'\times H^{\, 1/2}(\Omega)}.
$$
So we have
$$\vert\langle  \frac{\partial \widetilde{u}}{\partial x_{j}}, \varphi\rangle \vert \leq \Vert \frac{\partial u}{\partial x_{j}} \Vert_{[H^{\, 1/2}(\Omega)]'}\Vert \varphi \Vert_{H^{\, 1/2}(\Omega)} \leq  \Vert \frac{\partial u}{\partial x_{j}} \Vert_{[H^{\, 1/2}(\Omega)]'}\Vert \varphi \Vert_{H^{\, 1/2}(\mathbb{R}^N)}.
$$
We deduce that 
$$
 \nabla  \widetilde{u}\in  \textit{\textbf H}^{\, -1/2}(\mathbb{R}^N)
 $$
 which implies that $ \widetilde{u}\in  H^{\, 1/2}(\mathbb{R}^N)$ and then $u\in H^{1/2}_{00}(\Omega)$. 
\end{proof}

\begin{remark}\label{Surj}\upshape i) A closely related result can be found in \cite{Tar}, where
the author shows, using the inequality:
\begin{equation*}
\Vert \gamma_0 (v) \Vert_{L^2(\mathbb{R}^{N-1})} \leq C \Vert v \Vert_{L^2(\mathbb{R}^N)} \Vert \partial_N v \Vert_{L^2(\mathbb{R}^N)}
\end{equation*}
that the trace operator $\gamma_0: v \mapsto v_{\vert \Gamma}$ is continuous from the interpolation space $ [H^1(\mathbb{R}^N), L^2(\mathbb{R}^N)]_{1/2, 1}$ into $L^2(\mathbb{R}^{N-1})$. \medskip

\noindent ii) What about the characterization of the range of $E(\nabla;\, \Omega)$ by the linear mapping $\gamma: u \mapsto u_{\vert \Gamma}$? Is this range equal or strictly included in $L^2(\Gamma)$? We will provide some answers in Section \ref{traces2}.
\end{remark}

\begin{corollary}\label{TracesbH1demigradH1demiprime} i) The linear mapping $\gamma: u \mapsto (u_{\vert \Gamma},  \partial_{\textit{\textbf n}} u)$ is  continuous from $E(\nabla^2;\, \Omega)$ into $H^1(\Gamma)\times L^2(\Gamma)$, where
\begin{equation*}\label{defEnabla2Omega}
 E(\nabla^2;\, \Omega)\ =\ \left\{\,v\in H^{3/2}(\Omega);\ \nabla ^2v\in  [\textit{\textbf H}^{\, 1/2}(\Omega)]'\,\right\}.
\end{equation*}
ii) The kernel of  $\gamma$ from $E(\nabla^2;\, \Omega)$ into  $H^1(\Gamma)\times L^2(\Gamma)$ is equal to $H^{3/2}_{00}(\Omega)$.
\end{corollary}

\subsection{Traces for non smooth functions}

As we recalled earlier, $H^{1/2}(\Omega)$ functions, or even less regular ones, generally have no traces. Nevertheless, if their Laplacian satisfies certain properties, and in particular if they are harmonic, then we can define their trace in a certain sense. This will be the subject of this subsection.

To do this, we shall use a regularity result for a Dirichlet problem for the bi-Laplacian. This one will allow us to deduce an interesting Green's formula and then to define traces for non smooth functions. A more complete study on the  bi-Laplacian will be published later.
 
 So let us recall some existence and regularity results of the following problem:
 \begin{equation*}
 (\mathscr{B}_D)\quad
 \begin{cases}
 \Delta^2 u\ =\ 0 \quad\textrm{ in }\ \Omega,\\
   u\ =\ g_0 \quad\textrm{ on }\ \Gamma,\\
  \partial_{\textit{\textbf n}} u\ =\ g_1 \quad\textrm{ on }\ \Gamma.
 \end{cases}
 \end{equation*}
We know that if $g_0\in H^1(\Gamma)$ and $g_1 \in L^2(\Gamma)$,  the problem $(\mathscr{B}_D)$ has a unique solution $u\in H^{3/2}(\Omega)$.  
We know also (see \cite{AP}) that if 
\begin{equation*}\label{compa1}
g_0\in H^1(\Gamma),  \;  g_1 \in L^2(\Gamma)\quad \mathrm{and} \quad \nabla_{\tau}g_0 + g_1\boldsymbol{n}\in \boldsymbol{H}^{s}(\Gamma), \quad \mathrm{for\, \, some} \quad 0 < s  < 1
\end{equation*} 
then this solution belongs to  $H^{3/2 + s}(\Omega)$ withe the following estimate:
\begin{equation}\label{inegtraceDelta2x}
\Vert w \Vert_{H^{3/2 + s}(\Omega)} \leq C(\Omega) (\Vert g_0 \Vert_{H^1(\Gamma)} + \Vert g_1 \Vert_{L^2(\Gamma)} + \Vert \nabla_{\tau}g_0 + g_1\boldsymbol{n}\Vert_{\boldsymbol{H}^{s}(\Gamma)}).
\end{equation}

\begin{remark}\label{remrelev} \upshape  Let us introduce the following Hilbert space
$$
0\leq s < 1, \quad \boldsymbol{H}_{\textit{\textbf N}}^{s}(\Gamma) = \{ \boldsymbol{\mu} \in   \boldsymbol{H}^{s}(\Gamma); \; \boldsymbol{\mu}_\tau = {\bf 0} \} 
$$
equipped with the norm $\Vert\cdot \Vert_{\boldsymbol{H}^{s}(\Gamma)}$.
Clearly 
$$
\boldsymbol{\mu}\in \boldsymbol{H}_{\textit{\textbf N}}^{s}(\Gamma) \Longleftrightarrow  \boldsymbol{\mu}  = g\textit{\textbf n} \quad \mathrm{with}\quad g\in L^2(\Gamma) \quad  and \quad  g\textit{\textbf n}\in   \boldsymbol{H}^{s}(\Gamma).
$$
From above we deduce  that for any $\boldsymbol{\mu}\in \boldsymbol{H}_{\textit{\textbf N}}^{s}(\Gamma) $, with $0 \leq s < 1$, there exists a function $u\in H^{s+3/2}(\Omega)\cap H^1_0(\Omega)$ such that $\nabla u =  \boldsymbol{\mu}$ on $\Gamma$ and satisfying the estimate \eqref{inegtraceDelta2x}. 

Since $\textit{\textbf n} \in L^\infty(\Gamma)$ only, what conditions must $g\in L^2(\Gamma)$ satisfy so that the product $g\textit{\textbf n}$ belongs to $\boldsymbol{H}^{s}(\Gamma)$?
 \end{remark}
 
The example below provides an answer to this question. 
 
\begin{example}\upshape
Let us consider  the following Lipschitz domain:
$$
\Omega = \{(r, \theta);\; 0 < r < 1,\quad 0 < \theta < 3 \pi/2\}
$$
and different parts of its boundary:
$$
\Gamma_C = \{(r, \theta);\;  r  = 1\}, \; \Gamma_0 =\{(r, 0);\;  0 < r  < 1\}, \,  \Gamma_1 = \{(r, 3 \pi/2);\;  0 < r  < 1\}.
$$ 
It is easy to verify that for any $0 < s < 1$, if $g\in H^{s}(\Gamma)$ and 
\begin{equation*}\label{exampleintro}
 g_{\vert{\Gamma_0}}\in H^{s}_{00}(\Gamma_{0}),\; g_{\vert{\Gamma_1}}\in H^{s}_{00}(\Gamma_1), \; \mathrm{when }\; s = 1/2
\end{equation*}
and 
\begin{equation*}\label{exampleintrobis}
g_{\vert{\Gamma_0}}\in H^{s}_{0}(\Gamma_{0}),\; g_{\vert{\Gamma_1}}\in H^{s}_{0}(\Gamma_1), \; \mathrm{when }\; 1/2 < s < 1
\end{equation*}
then $ g\textit{\textbf n}\in   \boldsymbol{H}^{s}(\Gamma)$.
\end{example}

\begin{lemma} \label{densitetraces1} The following space 
\begin{equation*}
\left\{(v, \partial_{\textit{\textbf n}} v) \in H^1(\Gamma)\times L^2(\Gamma); \; v \in H^2(\Omega)\right\}
\end{equation*}
is dense in $H^1(\Gamma)\times L^2(\Gamma)$.
\end{lemma}
Adapting the proof of this result (see \cite{Necas}, Lemma 4.4, Chapter 5, page 274), we can prove easily the following lemma.

\begin{lemma} \label{densitetraces2}  The following space 
\begin{equation*}\label{DefX0}
X_0(\Gamma) = \left\{ \partial_{\textit{\textbf n}} v\in L^2(\Gamma); \; v \in H^2(\Omega)\cap H^1_0(\Omega)\right\}
\end{equation*}
is dense in $L^2(\Gamma)$.
\end{lemma}

Using the above lemma, we deduce immediately the following density result:
\begin{equation*}\label{densityH12T}
\boldsymbol{H}_{\textit{\textbf N}}^{1/2}(\Gamma) \; \; \mathrm{is \, \, dense\, \,  in }\;  \; \boldsymbol{L}_{\textit{\textbf N}}^{2}(\Gamma) := \{ \boldsymbol{\mu} \in   \boldsymbol{L}^{2}(\Gamma); \; \boldsymbol{\mu}_\tau = {\bf 0} \}. 
\end{equation*}
This space, as well as its dual space $ [\boldsymbol{H}_{\textit{\textbf N}}^{1/2}(\Gamma)]'$, can therefore be considered as space  of normal fields of regularity $1/2$ and $- 1/2$ respectively and for any $0 < s < 1/2$,  

\begin{equation}\label{inclusions}
\boldsymbol{H}_{\textit{\textbf N}}^{1/2}(\Gamma) \hookrightarrow \boldsymbol{H}_{\textit{\textbf N}}^{s}(\Gamma) \hookrightarrow \boldsymbol{L}_{\textit{\textbf N}}^{2}(\Gamma) \hookrightarrow [\boldsymbol{H}_{\textit{\textbf N}}^{s}(\Gamma)]'\hookrightarrow [\boldsymbol{H}_{\textit{\textbf N}}^{1/2}(\Gamma)]',
\end{equation}
with density.

Introduce now the following Hilbert space: for $-1/2 <  s <  1/2$
\begin{equation*}
M^s(\Omega) = \left\{ v \in H^s(\Omega);\ \Delta v  \in
[H^{3/2}_{0}(\Omega)]' \right\},
\end{equation*}
with its norm
$$
\Vert v\Vert_{M^{s}(\Omega)} = (\Vert v\Vert_{H^s(\Omega)}^2 +
\Vert \Delta v\Vert^2_{[H^{3/2}_{0}(\Omega)]'})^{1/2}.
$$
We will explain below the choice in the above definition of the dual space $[H^{3/2}_{0}(\Omega)]'$. And to simplify the notations, we set $M(\Omega) := M^0(\Omega)$.

\begin{lemma}\label{DensityM(Om)}   For any  $-1/2 <  s <  1/2$ the space 
$$
\mathscr{D}\big(\overline{\Omega}\big)\quad  is \, \, dense\, \,  in \; \; M^s(\Omega).
$$
\end{lemma}

\begin{proof} Let $\ell \in \left[M^{s}(\Omega)\right]'$ be such that 
$$
\forall v \in \mathscr{D}(\overline\Omega),\ \ \ \left<\, \ell, v\, \right>\ =\ 0.
$$
We know that there exist $f \in H^{-s}(\Omega) $ and $g \in H^{3/2}_{0}(\Omega)$ such that for any $v\in M^s(\Omega)$, 
$$
\left<\, \ell, v\, \right>\, = \langle f,\,  v\rangle_{H^{-s}(\Omega)\times H^{s}(\Omega)}  + \langle \Delta v ,\,  g \rangle_{[H^{3/2}_{0}(\Omega)]'\times H^{3/2}_{0}(\Omega)}.
$$
Let $\widetilde{f}\in H^{-s}(\mathbb{R}^N)$ and $\widetilde{g}\in H^{1}(\mathbb{R}^N)$  the extension functions by zero of respectively $f$ and $g$ (in fact $\widetilde{g}\in H^{r}(\mathbb{R}^N)$ for any $1 \leq r <  3/2$). Then for any $v\in \mathscr{D}(\mathbb{R}^N)$, we have
$$
\langle \widetilde{f},\,  v\rangle+ \int_{\mathbb{R}^N}  \widetilde{g}\, \Delta  v \, dx = 0,
$$
\textit{i.e.} 
$$  
- \Delta  \widetilde{g}\ = \widetilde{f}  \quad \mathrm{in}\; \mathbb{R}^N \quad\mathrm{and}\quad     \widetilde{g}  - \Delta  \widetilde{g}  \in H^{-s}(\mathbb{R}^N).
$$
Consequently, we deduce that $\widetilde{g}\in H^{2-s}( \mathbb{R}^N)$.  It means that $g\in H^{2 - s}_{0}(\Omega)$. As  $\mathscr{D}(\Omega)$ is dense in $ H^{2 - s}_{0}(\Omega)$, there exists $g_k\in \mathscr{D}(\Omega)$ such that 
$$
g_k\  \ \longrightarrow \ \ g \ \ \textrm{ in } \  H^{2 - s}_{0}(\Omega) \quad  \mathrm{as}\; k \rightarrow \infty.
$$
Then, for any $v\in M^s(\Omega)$, we have
$$
\left<\, \ell, v\, \right>\, =\ \lim_{k\rightarrow\infty} (- \int_\Omega v\Delta g_k \, dx   + \langle \Delta  v, g_k\rangle ) = 0,
$$
which ends the proof.
\end{proof}

\begin{remark}\label{RemdenMs} \upshape The above proof can not be extended to the case $s =  1/2$ since the extension by zero of $f\in [H^{1/2}(\Omega)]'$, resp. $f\in H^{1/2}(\Omega)$ does not belong to $H^{-1/2}(\mathbb{R}^N)$, resp. $H^{1/2}(\mathbb{R}^N)$. However, this remains true for $s \leq -1/2$, provided that, in the definition of $M^s(\Omega)$, the space $H^s(\Omega)$ is replaced by the dual of ${H^s_{00}(\Omega)} $ when $- s \in \{1/2\} + \mathbb{N}$.
\end{remark}

\begin{theorem} \label{lsb3b} Let $-1/2 < s < 1/2$. The linear mapping ${\boldsymbol\gamma}_\textit{\textbf n}: v \longmapsto v\textit{\textbf n}_{\vert\Gamma} $ defined on
$\mathscr{D}\big(\overline{\Omega}\big)$ can be extended to a linear and continuous mapping, still denoted by ${\boldsymbol\gamma}_\textit{\textbf n}$:
\begin{equation*}\label{deftrace}
\begin{array}{rl}
{\boldsymbol\gamma}_\textit{\textbf n} : M^s(\Omega) &\longrightarrow
[{\boldsymbol{H}_{\textit{\textbf N}}^{1/2-s}(\Gamma)}]'\\
 v & \longmapsto v\textit{\textbf n}_{\vert\Gamma}
\end{array}
\end{equation*}
Moreover, we have the Green's formula: For all $v \in M^s(\Omega)$ and  $ \varphi \in H^{2-s}(\Omega)\cap H^{1}_0(\Omega)$,
\begin{equation} \label{esb3am}
\langle  v , \,  \Delta \varphi  \rangle_{H^{s}(\Omega)\times H^{-s}(\Omega)} -  \langle \Delta v ,\,  \varphi \rangle_{[H^{3/2}_{0}(\Omega)]'\times H^{3/2}_{0}(\Omega)} = \langle v\textit{\textbf n},\, \nabla \varphi \rangle_{\Gamma},
\end{equation}
where $ \langle v\textit{\textbf n},\, \nabla \varphi \rangle_{\Gamma}$ denotes the duality brackets between $[{\boldsymbol{H}_{\textit{\textbf N}}^{1/2-s}(\Gamma)}]'$ and $ \boldsymbol{H}_{\textit{\textbf N}}^{1/2-s}(\Gamma)$.
\end{theorem}

\begin{proof}  Clearly the above Green's formula holds for any $v \in \mathscr{D}\big(\overline{\Omega}\big)$ and any  $ \varphi \in H^{2-s}(\Omega)\cap H^{1}_0(\Omega)$. Therefore, 
\begin{equation*} \label{esb3ba}
\vert \int_\Gamma v \partial_{\textit{\textbf n}}\varphi \, d\sigma \vert= \vert \int_\Gamma v \textit{\textbf n}\cdot \nabla \varphi \, d\sigma \vert \leq \Vert v \Vert_{M^s(\Omega)} \Vert \varphi \Vert_{H^{2-s}(\Omega)}.
\end{equation*}
Now, let $\boldsymbol{\mu} \in   \boldsymbol{H}_{\textit{\textbf N}}^{1/2-s}(\Gamma)$. From Remark \ref{remrelev} we know that there exists a function $\varphi\in H^{2-s}(\Omega)\cap H^1_0(\Omega)$ such that $\nabla \varphi =  \boldsymbol{\mu}$ on $\Gamma$ with the estimate 
$$
\Vert \varphi \Vert_{H^{2-s}(\Omega)} \leq C \Vert \boldsymbol{\mu}\Vert_{\boldsymbol{H}^{1/2-s}(\Gamma)}.
$$ 
Hence the above inequalities imply that for any $v \in \mathscr{D}\big(\overline{\Omega}\big)$
\begin{equation*} \label{esb3bb}
\vert \int_\Gamma v \textit{\textbf n}\cdot \boldsymbol{\mu}\vert \leq C \Vert v \Vert_{M^s(\Omega)} \Vert \boldsymbol{\mu}\Vert_{\boldsymbol{H}^{1/2-s}(\Gamma)}.
\end{equation*}
From \eqref{inclusions}, $v \textit{\textbf n}\vert_\Gamma\in [{\boldsymbol{H}_{\textit{\textbf N}}^{1/2-s}(\Gamma)}]'$ and then
\begin{equation*} \label{esb3bb}
\Vert v \textit{\textbf n}\vert_\Gamma\Vert_{[\boldsymbol{H}_{\textit{\textbf N}}^{1/2-s}(\Gamma)]'} \leq C \Vert v \Vert_{M^s(\Omega)}.
\end{equation*}
Therefore the linear mapping ${\boldsymbol\gamma}_\textit{\textbf n}: v \longmapsto v \textit{\textbf n}\vert_\Gamma $ defined on $\mathscr{D}\big(\overline{\Omega}\big)$  is continuous for the norm of  $M^s(\Omega)$. As  $\mathscr{D}\big(\overline{\Omega}\big)$ is dense  in $M(\Omega)$, the  mapping ${\boldsymbol\gamma}_\textit{\textbf n}$ can be extended to a linear continuous mapping still called ${\boldsymbol\gamma}_\textit{\textbf n}$ from $M(\Omega) $ into $[{\boldsymbol{H}_{\textit{\textbf N}}^{1/2-s}(\Gamma)}]'$.
\end{proof}

\begin{remark}\label{remrelevbis}\upshape 

\noindent i) Recall that when $\Omega$ is of class $\mathscr{C}^{1,1}$, the linear mapping $v \longmapsto v |_\Gamma$ defined on
$\mathscr{D}\big(\overline{\Omega}\big)$ can be extended to a linear
continuous mapping
$$
M(\Omega) \longrightarrow H^{-1/2}(\Gamma)
$$
and  we have the Green's formula (see Lemma 2 in \cite{ARB}): for all $v \in M(\Omega)$ and  $ \varphi \in H^{2}(\Omega)\cap H^{1}_0(\Omega)$,
\begin{equation*} \label{GFC11}
\int_{\Omega} v  \Delta \varphi\, dx -  \langle \Delta v ,\,  \varphi \rangle_{[H^{3/2}_{0}(\Omega)]'\times H^{3/2}_{0}(\Omega)}   = \langle v  ,
\partial_{\textit{\textbf n}}\varphi \rangle_{H^{-1/2}(\Gamma) \times H^{1/2}(\Gamma) }.
\end{equation*}
When the domain $\Omega$ is only Lipschitz, Theorem \ref{lsb3b} does not allow us to define a trace $H^{-1/2}(\Gamma)$ for functions of $M(\Omega)$; it merely establishes that for any $v \in M(\Omega)$, there exists a unique linear form, denoted by $v{\textit{\textbf n}}$, belonging to $[{\boldsymbol{H}_{\textit{\textbf N}}^{1/2}(\Gamma)}]' $ and satisfying Green formula \eqref{esb3am}.\smallskip

\noindent ii) In \cite{Gri1}, Lemma 3.2, the author gives a close result in the case where $\Omega$ is a polygonal domain. Let us define for any $s\in \mathbb{N}$ the following spaces: \begin{equation}\label{Ws}
W_s(\Omega) =  \{\varphi \in H^{2+s}(\Omega)\cap H^{1}_0(\Omega); \, \Delta \varphi \in H^{s}_0(\Omega)\},
\end{equation}
which is a Hilbert space for the graph norm and 
$$K_s(\Omega) =  \left\{ v \in H^{-s}(\Omega);\ \Delta v \in L^2(\Omega) \right\}.
$$
Then the following mapping $v \longmapsto v\vert_\Gamma $ defined on
$\mathscr{D}\big(\overline{\Omega}\big)$ can be extended to a linear
continuous mapping
$$
K_s(\Omega) \longrightarrow[X_s(\Gamma)]',
$$
such that
\begin{equation} \label{esb3bgri}
\forall v \in K_s(\Omega)\quad \mathrm{ and } \quad   \forall \varphi \in W_s(\Omega),\quad \langle v, \Delta \varphi\rangle -  \int_{\Omega} \varphi  \Delta v\, dx = \langle v ,\, \partial_{\textit{\textbf n}}\varphi\rangle_{[X_s(\Gamma)]'\times X_s(\Gamma)}.
\end{equation}
where $X_s(\Gamma)$ is the space described by the normal derivative $\partial_{\textit{\textbf n}}\varphi$ when $\varphi$ browse the  space $W_s(\Omega)$.

This result allows us to define rigorously  the kernels 
$$
 \{v \in H^{-s}(\Omega); \, \Delta v = 0\;  \mathrm{in}\, \Omega\; \; \mathrm{and }\, \;  v = 0\;  \mathrm{on}\, \,  \Gamma \}
$$ 
and thus to apply the Fredholm alternative in the study of the regularity of solutions to the Dirichlet problem for the Laplacian.

According to Remark \ref{RemdenMs} and the density of  $\mathscr{D}\big(\overline{\Omega})$ in $M^s(\Omega)$, for any $-\infty < s < 1/2$, we have again the following Green formula: for $s > - 1/2$,
\begin{equation} \label{esb3bgriMs}
\forall v \in M_{-s}(\Omega)\quad \mathrm{ and } \quad   \forall \varphi \in W_s(\Omega),\quad \langle v, \Delta \varphi\rangle -  \int_{\Omega} \varphi  \Delta v\, dx = \langle v ,\, \partial_{\textit{\textbf n}}\varphi\rangle_{[X_s(\Gamma)]'\times X_s(\Gamma)},
\end{equation}
where the space $W_s(\Omega)$ is defined as in \eqref{Ws}, but with $H^{s}_0(\Omega)$ replaced by $H^{s}(\Omega)$ when $-1/2 < s < 1/2$ and by $H^{s}_{00}(\Omega)$ when $s \in \{1/2\} + \mathbb{N}$. The formula \eqref{esb3bgriMs} thus extends to the case where $ s > - 1/2$ and $M_{-s}(\Omega)$,  the formula \eqref{esb3bgri} established for  $ s \in \mathbb{N}$ and $v \in K_{s}(\Omega)$, this last space being included in $M_{-s}(\Omega)$.

Now, since $\mu\in X_s(\Gamma)$ iff $\mu\textit{\textbf n} \in \boldsymbol{H}^{1/2-s}_{\textit{\textbf N}}(\Gamma)$ for any  $- 1/2 < s < 1/2$,  we deduce that  if $v \in M_{s}(\Omega) $ 
\begin{equation*}
 \forall \mu\in X_s(\Gamma),\quad \langle v,\, \mu \rangle_{[X_s(\Gamma)]'\times X_s(\Gamma)} =   \langle v\textit{\textbf n},\, \mu \textit{\textbf n} \rangle_{[{\boldsymbol{H}_{\textit{\textbf N}}^{1/2-s}(\Gamma)}]'\times {\boldsymbol{H}_{\textit{\textbf N}}^{1/2-s}(\Gamma)}} .
\end{equation*}
Formula \eqref{esb3bgriMs} gives meaning to the trace of $v \in M_{-s}(\Omega)$, for any $s > - 1/2$, as an element of $[X_s(\Gamma)]'$, but $X_s(\Gamma)$ and $[X_s(\Gamma)]'$ have the inconvenience of being abstract. But as noted above, it makes it possible, in particular, to make sense of the traces of harmonic functions belonging to $H^{s}(\Omega)$ with $-\infty < s < 1/2$.  One of the advantages of Formula \eqref{esb3am} is that the duality brackets on the right-hand side are more precise. 
 \smallskip 

\noindent iii) As in \cite{Fabes3} (see the relation (9.3)), we can define this notion of trace in the following way.  For any $v \in M^s(\Omega)$, with $-1/2 < s < 1/2$, we define and denote by 
$\widetilde{v \textit{\textbf n}}$ the linear functional in $[{\boldsymbol{H}_{\textit{\textbf N}}^{1/2-s}(\Gamma)}]'$ given by: for any $\boldsymbol{\mu} \in   \boldsymbol{H}_{\textit{\textbf N}}^{1/2-s}(\Gamma)$
\begin{equation*} \label{esb3b}
\langle \widetilde{v \textit{\textbf n}},\, \boldsymbol{\mu} \rangle_{[{\boldsymbol{H}_{\textit{\textbf N}}^{1/2-s}(\Gamma)}]'\times {\boldsymbol{H}_{\textit{\textbf N}}^{1/2-s}(\Gamma)}} : =
\int_{\Omega} v  \Delta \varphi\, dx -  \langle \Delta v ,\,  \varphi \rangle_{[H^{3/2}_{0}(\Omega)]'\times H^{3/2}_{0}(\Omega)},
\end{equation*}
where $ \varphi \in H^{2}(\Omega)\cap H^{1}_0(\Omega)$ is such that $\partial_{\textit{\textbf n}}\varphi =  \boldsymbol{\mu}\cdot \textit{\textbf n}$. Recall that a such extension exists thanks to Remark \ref{remrelev} Point i). Clearly, these two notions of traces coincide, so we have $\widetilde{v \textit{\textbf n}} = v \textit{\textbf n}$.  
\smallskip

\noindent iv) Clearly, from Theorem \ref{lsb3b}, if $v\in H^{1/2}(\Omega)$, with  $\Delta v\in [H^{3/2}_{0}(\Omega)]'$, then for any $0 < s < 1/2$, $v \textit{\textbf n} \in [{\boldsymbol{H}_{\textit{\textbf N}}^{s}(\Gamma)}]'.$ And  for any $ \varphi \in H^{2-s}(\Omega)\cap H^{1}_0(\Omega)$,
\begin{equation*}
\langle  v , \,  \Delta \varphi  \rangle_{H^{s}(\Omega)\times H^{-s}(\Omega)} -  \langle \Delta v ,\,  \varphi \rangle_{[H^{3/2}_{0}(\Omega)]'\times H^{3/2}_{0}(\Omega)} = \langle v\textit{\textbf n},\, \nabla \varphi \rangle_{\Gamma},
\end{equation*}
Now, if in addition we suppose that there exists $g\in L^2(\Gamma)$ such that
\begin{equation*}\label{relg}
\langle  v , \,  \Delta \varphi  \rangle_{H^{s}(\Omega)\times H^{-s}(\Omega)} -  \langle \Delta v ,\,  \varphi \rangle_{[H^{3/2}_{0}(\Omega)]'\times H^{3/2}_{0}(\Omega)} = \int_\Gamma g \, \partial_{\textit{\textbf n}}\varphi \, d\sigma ,
\end{equation*}
then $v \textit{\textbf n} = g \textit{\textbf n}$  in the sense of  $ [{\boldsymbol{H}_{\textit{\textbf N}}^{s}(\Gamma)}]'$. So $v \textit{\textbf n} \in \textit{\textbf L}^2_N(\Gamma)$ and taking the inner product with $\textit{\textbf n}$ we deduce that  $v \textit{\textbf n} \cdot \textit{\textbf n}  = g $ on $\Gamma$.  Moreover the uniqueness on the trace $v \textit{\textbf n} $ implies that of $g$. This allows us to consider that g is the trace of $v$ as in the case where the domain $\Omega$ is of class $\mathscr{C}^{1, 1}$ and leads us to the following definition.
\end{remark}

\begin{e-definition} \label{deftr}Let $v\in H^{1/2}(\Omega)$ with  $\Delta v\in [H^{3/2}_{0}(\Omega)]'$.  If there exists a function $g\in L^2(\Gamma)$ such that for some $0 \leq s <  1/2$  
\begin{equation*}\label{deftrace}
\forall \varphi \in H^{2-s}(\Omega)\cap H^{1}_0(\Omega),\quad 
\langle  v , \,  \Delta \varphi  \rangle_{H^{s}(\Omega)\times H^{-s}(\Omega)} -  \langle \Delta v ,\,  \varphi \rangle_{[H^{3/2}_{0}(\Omega)]'\times H^{3/2}_{0}(\Omega)} =  \int_\Gamma g\,  \partial_{\textit{\textbf n}}\varphi \, d\sigma ,
\end{equation*}
we will say that $g$ is the trace of $v$.
\end{e-definition}

\begin{remark}\upshape i) Let $v\in H^{1/2}(\Omega)$ with $\nabla v \in  [\boldsymbol{H}^{1/2}(\Omega)]'$. Then $\Delta v\in [H^{3/2}_{0}(\Omega)]'$. We know that the trace of $v$ exists and belongs to $L^2(\Gamma)$. Moreover we have the following Green formula: for any $0 \leq  s < 1/2$ and
\begin{equation*}\label{GrH1demietH1demip}
\forall \varphi \in H^{2-s}(\Omega)\cap H^{1}_0(\Omega),\quad 
\langle  v , \,  \Delta \varphi  \rangle_{H^{1/2}(\Omega)\times [H^{1/2}(\Omega)]'} -  \langle \Delta v ,\,  \varphi \rangle_{[H^{3/2}_{0}(\Omega)]'\times H^{3/2}_{0}(\Omega)} =  \int_\Gamma v \, \partial_{\textit{\textbf n}}\varphi \, d\sigma .
\end{equation*}
This Green formula holds also for $\varphi$ in $H^{3/2}_{0}(\Omega)$ with $\Delta \varphi$ in $ [H^{1/2}(\Omega)]'$ (see Theorem \ref{CoruH10DeltaudualH1/2}).

\noindent ii) According to Point i) of Remark \ref{remrelevbis}, the trace defined in the above defintion when $\Omega$ is Lipschitz is indeed the same as that where the domain is $\mathscr{C}^{1, 1}$.
\end{remark}

\section{Preliminary results for polygonal and polyedral domains}\label{poly}

Before studying in the following sections the regularity issues for the inhomogeneous Dirichlet problem stated in a bounded Lipschitz open set, it is important to recall and analyse the case where the domain is a Lipschitz polygonal domain in $\mathbb{R}^2$, and particularly when it is non convex.

\subsection{Characterization of harmonic functions kernels}\label{ssCharkernels} 

The characterization of the kernels of harmonic functions belonging to Sobolev spaces $H^s(\Omega)$ and satisfying a homogeneous Dirichlet condition is simple in the case where the open set $\Omega$ is sufficiently regular, even in the case where $s$ is a negative real number.  Let us define the following kernel for any $ - 1/2 \leq s <  \infty$: 
$$
\mathscr{N}_s(\Omega) = \{\varphi \in H^{-s} (\Omega); \; \Delta \varphi = 0 \; \mathrm{in}\; \Omega \; \mathrm{and}\; \varphi = 0\; \mathrm{on}\;  \Gamma \},
$$
with the same notation used by Grisvard \cite{Gri1}. The meaning of the trace for non-regular but harmonic functions is explained in \cite{Gri1} and also in Section 4.

In the case where $\Omega$ is a bounded domain, convex or of class $\mathscr{C}^{1, 1}$ in $\mathbb{R}^N$ with $N \geq 2$,  this kernel is trivial for all $s < 1/2$. The situation is different if the bounded domain $\Omega$ is only Lipschitz. \medskip

As an example, let us consider again the following Lipschitz and non convex domain: 
\begin{equation*}
\mathrm{for}  \; 1/2 < \alpha < 1, \quad \Omega = \{(r, \theta);\; 0 < r < 1,\quad 0 < \theta < \frac {\pi} {\alpha}\}.
\end{equation*}
We can easily verify that the following function
\begin{equation*}
z(r, \theta) = (r^{-\alpha} - r^{\alpha})\mathrm{sin}(\alpha\theta)
\end{equation*}
is harmonic in $\Omega$ with $z = 0$ on $\Gamma$ and $z\in H^{t}(\Omega)$ for any $t < 1 -\alpha $. Note that $0 <  1 -\alpha < 1/2$.\medskip

Let us now consider the case where $\Omega$ is a bounded Lipschitz polygonal domain, we will simply write polygon, whose angles are denoted by $\omega_1, \ldots, \omega_n$. Even if we have to reorder this family, we can assume that it is increasing: $\omega_1 \leq \ldots \leq \omega_n$. We set $\alpha_j = \frac{\pi}{\omega_j} \in\,  ]1/2, \infty[$.\medskip

Recall that when $ s\in \mathbb{N}$, we have (see \cite{Gri1})
\begin{equation}\label{eqdim}
\mathrm{dim}\, \mathscr{N}_s(\Omega)  = \sum_{j = 1}^{j = n}\nu_s(\omega_j), \quad \mathrm{where}\quad \nu_s(\omega_j) = \mathrm{the\, largest \ integer} < \frac{1 + s}{\alpha_j}, \end{equation}
provided that for $s \geq 1$,
\begin{equation*}\label{conddim}
\forall j = 1\ldots, n, \quad \omega_j \notin \{\frac{1}{s+1}\pi, \ldots, \frac{s}{s+1}\pi\}. 
\end{equation*}

The case where the parameter $s$ is real is also interesting. Here, we will limit ourselves to cases where $-1/2 \leq s \leq 0$. We can easily verify that the equality \eqref{eqdim} holds also for such $s$. A brief explanation of this result can be given. The singularities involved in the structure of the kernel $\mathscr{N}_s(\Omega)$ are, as in the above example, of the type $(r^{-\alpha_j} - r^{\alpha_j})\mathrm{sin}(\alpha_j\theta)$, where here $1/2 < \alpha_j = \frac{\pi}{\omega_j}< 1$. However the function $\vert x \vert^{-\alpha_j}$ belongs to $H^t(B)$ for any $ t$ strictly less than $1 - \alpha_j $, where $B$ is the unit disk centered at origin.  Note that the function  $\vert x \vert^{-\alpha_j}$ belongs also to the Besov space $[H^1(B), L^2(B)]_{\alpha_j, \infty}$. Note that $\frac{1 + s}{\alpha_j} \in  ]1/2, 2[$ when $-1/2 \leq s \leq 0$. So we have

\begin{equation}\label{dimang}
\forall -1/2 \leq s \leq 0, \quad dim\,  \mathscr{N}_{s}(\Omega) =  Card\{j \in \{1, \ldots, n\}; \;  \omega_j > \pi(1+s)^{-1}\}.
\end{equation}
Since the family $(\alpha_j)_{1\leq j \leq n}$ is decreasing, then 
\begin{equation}\label{Ns0}
 \mathscr{N}_s(\Omega) = \{0\}\quad \mathrm{ when}\quad   s \leq \alpha_1 - 1.
 \end{equation}
 Observe that the relation \eqref{eqdim} remains true if the parameter $s \geq -1/2$ is real or if $\Omega$ is a curvilinear polygonal open set. \medskip

\begin{remark}\label{remNoyau}\upshape  i) Note that $
\mathscr{N}_0(\Omega) = \{0\}$ if and only if the polygon $\Omega$ is convex.\smallskip

\noindent ii) We have  $\mathscr{N}_{-1/2}(\Omega) = \{0\} $ for any polygon (and also for any bounded Lipschitz domain of $\mathbb{R}^N,$ with $N \geq 2$, see Theorem \ref{UniciteH1/2} below).\smallskip

\noindent  iii) For any polygon $\Omega$, there exists $s_0(\Omega) \in \, ]0, 1/2[$ such that
\begin{equation}\label{s_0}
\mathrm{for \, any}\; s \geq s_0(\Omega),\; \;   \mathscr{N}_{-s}(\Omega) = \{0\} \quad \quad \mathrm{and}\quad \quad \mathrm{for \, any}\; s < s_0(\Omega),  \; \mathscr{N}_{-s}(\Omega) \neq \{0\}.
 \end{equation}
It is easy to check that $ s_0 = 1 - \alpha_n $.\smallskip

\noindent   iv) For any bounded Lipschitz domain $\Omega$ of $\mathbb{R}^2$, there exists $s_0(\Omega) \in \, ]0, 1/2[$ such that  \eqref{s_0} holds.

 \medskip
 Let now $\Omega$ be a bounded Lipschitz polyhedral domain of $\mathbb{R}^3$,  we will simply write polyhedron. The situation is little bit different. We denote by $\Gamma_k$, $k = 1, \ldots,  K$ the faces of $\Omega$ and by $E_{jk}$ the edge between $\Gamma_j$ and $\Gamma_k$ when $\overline{\Gamma}_j$ and $\overline{\Gamma}_k$ intersect. The measure of the interior angle of the edge $E_{jk}$ is denoted by $\omega_{jk}$ and as in 2D we have the following property: 
$$
\mathrm{if}\;  \alpha_{jk} \geq 1 - s, \quad \mathrm{for \, any }\;  1 \leq j, k \leq K, \quad \mathrm{then }\quad \mathscr{N}_s(\Omega) = \{0\},
$$
where $\alpha_{jk} = \frac{\pi}{\omega_{jk}}$. However, if one of the numbers $ \alpha_{jk}$ is strictly less than $ 1 - s$, then $\mathrm{dim}\, \mathscr{N}_s(\Omega) = + \infty$. 
\end{remark}

\subsection{Interpolation of subspaces}

We recall in this subsection some interpolation results of subspaces. The first one is due to Ivanov and Kalton \cite{IK}. Let $(X_0, X_1)$ be a Banach couple with $X_0\cap X_1$ dense in $X_0$ and in $X_1$. Let $Y_0$ be a closed subspace of $X_0$ with codimension one. Setting for any $0 < \theta < 1$
$$
X_\theta = [X_0, X_1]_\theta \quad \mathrm{and}\quad Y_\theta = [Y_0, X_1]_\theta,
$$
we have the following result concerning the interpolation of subspaces:
 
\begin{theorem} [{\bf Ivanov-Kalton}]\label{thmIK} There exist two indices $0 \leq \sigma_0 \leq \sigma_1 \leq 1$ such that\\
i) If $0 < \theta < \sigma_0$, then the space $Y_\theta $ is a closed subspace of codimension one in the space $ X_\theta $.\\
ii) If $\sigma_0 \leq \theta \leq \sigma_1$, then the norm of $Y_\theta $ is not  equivalent to the norm of $X_\theta $.\\
iii) If $\sigma_1 < \theta < 1$, then $Y_\theta = X_\theta$ with equivalence of norms.
\end{theorem}

As specified in  \cite{IK}, the special case of a Hilbert space of Sobolev type connected with elliptical boundary value problem was studied in \cite{Lions}, with the well known case $X_0 = H^1(\Omega),\,  X_1 = L^2(\Omega)$ and   $Y_0 = H^1_0(\Omega)$, but where  $Y_0$ is here a closed subspace of codimension infinite in $X_0$. We recall that the corresponding critical values are $\sigma_0 = \sigma_1 = 1/2$ and $Y_{1/2} = H^{1/2}_{00}(\Omega)$.\medskip

The above theorem is generalized  by Asekritova, Cobos and Kruglyak \cite{ACK} when $Y_0$ is a closed subset of finite codimension $n$ in $X_0$:

\begin{theorem} [{\bf Asekritova-Cobos-Kruglyak}]\label{thmACK} There exist $2m$ indices satisfying:  
$$
\mathrm{for}\;  j = 1, \ldots, m, \quad 0 \leq \sigma_{0j} \leq \sigma_{1j} \leq 1
$$
and such that 
\begin{equation*}\label{CNSclose}
Y_\theta \; \; \mathrm{is\,\, closed\,\, in} \; \; X_\theta \quad \Longleftrightarrow \quad \theta \notin \displaystyle \bigcup_{j= 1}^{j= n} [\sigma_{0j} , \sigma_{1j} ].
\end{equation*}
Moreover, in that case if the cardinal 
\begin{equation*}\label{Cardinter}
\vert\{j \in \{1, \ldots, m \}; \; \theta < \sigma_{0j}\}\vert \quad \mathrm{is\,\, equal\,\, to}\;  k,
\end{equation*}
then the space $Y_\theta$ is a closed subspace of codimension $k$ in $X_\theta$.
\end{theorem}

\subsection{ $H^s(\Omega)$-regularity.}

Recall now the following result due to Grisvard (see Theorem 2.4 in \cite{Gri1} and  Theorem 2.1 in \cite{Gri2}). 

\begin{theorem} [{\bf Grisvard, $H^2$-Regularity}]\label{theoGrisvardPol} Let $\Omega$ be a polygonal domain of $\mathbb{R}^2$  or a polyhedral domain of $\mathbb{R}^3$. \\
i) The following inequality holds: there exists a constant $C(\Omega)$ such that
\begin{equation*}\label{inegH2}
\forall v \in H^2(\Omega)\cap H^{1}_0(\Omega),  \quad \Vert v \Vert_{H^2(\Omega)} \leq C(\Omega) \Vert \Delta v \Vert_{L^2(\Omega)},
\end{equation*}
where the constant $C(\Omega)$ depends only on the Poincar\'e constant of $\Omega$ (see \eqref{inegisoH2L2} in next remark point iii)).

\noindent ii) The following  operator is an isomorphism
\begin{equation}\label{DeltaIsoH3/2H1/2ortx}
\Delta : H^2(\Omega)\cap H^{1}_0(\Omega) \longrightarrow   [\mathscr{N}_0(\Omega)]^\bot ,
\end{equation}
where 
$$
  [\mathscr{N}_0(\Omega)]^\bot  = \{f \in L^2(\Omega); \; \forall \varphi\in \mathscr{N}_0(\Omega), \;  \int_\Omega f\varphi \, dx= 0\}.
  $$
iii) Codim$\mathrm{(}$Im $\Delta \mathrm{)}$  is finite in $2$D and infinite in $3D$.
\end{theorem}

\begin{remark}\label{remtrace}\upshape i) For the meaning of traces for the functions $v \in L^2(\Omega)$ satisfying $\Delta v \in L^2(\Omega)$  when  $\Omega$ is a polygonal domain of $\mathbb{R}^2$  or a polyhedral domain of $\mathbb{R}^3$, see Remark \ref{remrelevbis}.\medskip

\noindent ii) Unlike the case where the domain $\Omega$ is convex or is regular, of class $\mathscr{C}^{1,1}$ for example, the kernel $\mathscr{N}_0(\Omega)$ is not trivial. In the polygonal case,  as mentioned above (see also Subsection \ref{ssCharkernels}), it is of finite dimension and its dimension is equal to the number of vertices of the polygon whose corresponding interior angle is strictly greater than  $\pi$.\medskip

\noindent iii) According to  the Fredholm alternative, since the range of the operator $\Delta :  H^2(\Omega)\cap H^{1}_0(\Omega) \longrightarrow   L^2(\Omega)$  is closed in  $L^2(\Omega)$, its orthogonal is equal to the kernel of the adjoint operator. The result of Point ii) means that this kernel is in fact the space  $\mathscr{N}_0(\Omega)$.
\end{remark}  

The solvability of problem 
$$
(\mathscr{L}_D^0)\ \ \ \  -\Delta u = f\quad \ \mbox{in}\ \Omega \quad
\mbox{and } \quad u = 0 \ \ \mbox{on }\Gamma,
$$ 
in a framework of fractional Sobolev spaces when $\Omega$ is a polygon or a polyhedron, and naturally also when   $\Omega$ is a general Lipschitz domain, has been studied by many authors.  Recall that when  
 $\Omega$ is a general Lipschitz domain, for any $1/2 < s < 3/2$, the operator
\begin{equation*}\label{isoDeltaLip}
 \Delta : H^{2- s}_0(\Omega) \longrightarrow H^{-s}(\Omega)
\end{equation*}
is an isomorphism. So, the question is: what happens in the case where $\Omega$ is a polygonal domain of $\mathbb{R}^2$  or a polyhedral domain of $\mathbb{R}^3$ and where $0 \leq s \leq 1/2$? A first answer can be given by the following corollary.

\begin{corollary}\label{firstcor} Let $\Omega$ be a polygonal domain of $\mathbb{R}^2$  or a polyhedral domain of $\mathbb{R}^3$. Then for any $0 < \theta < 1$, the operator
\begin{equation}\label{DeltaIsoH2-thetaH-thetabbis}
 \Delta : H^{2-\theta}(\Omega) \cap H^1_0(\Omega)\longrightarrow M_\theta(\Omega)
\end{equation}
is an isomorphism, where 
$$
M_\theta(\Omega) : = [(\mathscr{N}_0(\Omega))^\bot,  H^{-1}(\Omega)]_{\theta}
$$
 is continuously embedded in $  H^{-\theta}(\Omega)$ if $\theta \neq 1/2$ and in $[H^{1/2}_{00}(\Omega)]'$ if $\theta = 1/2$.
\end{corollary}

\begin{proof} i) First, let us recall that the operator
\begin{equation}\label{DeltaIsoH10H-1a}
\Delta : \; H^{1}_{0}(\Omega)\longrightarrow  H^{-1}(\Omega)
\end{equation}
is an isomorphism.  By a simple interpolation argument, thanks to \eqref{DeltaIsoH3/2H1/2ortx} and \eqref{DeltaIsoH10H-1a}, we deduce that  for any $0 < \theta < 1$, the operator
\begin{equation}\label{DeltaIsoH2-thetaH-thetab}
 \Delta : H^{2-\theta}(\Omega) \cap H^1_0(\Omega)\longrightarrow [(\mathscr{N}_0(\Omega))^\bot,  H^{-1}(\Omega)]_{\theta}
\end{equation}
is an isomorphism. Moreover, using Fredholm alternative, the Laplace operator considered as operating from $H^{2-\theta}(\Omega) \cap H^1_0(\Omega)$  into $H^{-\theta}(\Omega)$  if $\theta \neq 1/2$, resp. into $[H^{1/2}_{00}(\Omega)]'$ if $\theta = 1/2$, verifies the relation:
\begin{equation}\label{MKer}
\overline{M_\theta(\Omega)} = [Ker (\Delta^\star)]^\bot = [\mathscr{N}_{-\theta}(\Omega)]^\bot,
\end{equation}
where 
$$
[\mathscr{N}_{-\theta}(\Omega)]^\bot = \{f\in H^{-\theta}(\Omega); \, \langle f, \varphi\rangle = 0, \, \forall \varphi \in \mathscr{N}_{-\theta}(\Omega)\}
$$
 if $\theta \neq 1/2$ (for $\theta = 1/2$, replace $H^{-\theta}(\Omega)$ by $[H^{1/2}_{00}(\Omega)]'$). Note that $\mathscr{N}_{-\theta}(\Omega) = \{0\}$ for $\theta \geq 1/2$.\medskip
 
\noindent ii) Let us prove that the operator \eqref{DeltaIsoH2-thetaH-thetab} is an isomorphism.  Indeed,  setting $S= \Delta^{-1}$ we know from \eqref{DeltaIsoH3/2H1/2ortx} and \eqref{DeltaIsoH10H-1a} that  the operators $S : \; H^{-1}(\Omega))\longrightarrow  H^{1}_{0}(\Omega)$ and $S : \; [\mathscr{N}_0(\Omega)]^\bot\longrightarrow  H^2(\Omega)\cap H^{1}_0(\Omega)$ are continuous. So by interpolation, we  also  have the continuity of the operator 
\begin{equation}\label{opS}
 S : \; M_\theta(\Omega) \longrightarrow  H^{2-\theta}(\Omega) \cap H^1_0(\Omega).
 \end{equation}
 Clearly this last operator is injective. Let us prove its surjectivity. Before that recall that the density of a Banach $X$ in a Banach $Y$ implies the density of $X$ in the interpolate space $[X, Y]_\theta$ (see \cite{Lions}).  Since  $ H^{2}(\Omega) \cap H^1_0(\Omega)$ is dense in  $H^1_0(\Omega)$, we deduce that $[\mathscr{N}_0(\Omega)]^\bot$ is also dense in $H^{-1}(\Omega)$. Applying the above property with $X = [\mathscr{N}_0(\Omega)]^\bot$  and $Y =  H^{-1}(\Omega)$, we deduce the density of $[\mathscr{N}_0(\Omega)]^\bot$ in $ M_\theta(\Omega)$. Given now $f\in M_\theta(\Omega)$, there exists a sequence $f_j \in [\mathscr{N}_0(\Omega)]^\bot$  such that $f_j \rightarrow f$ in $M_\theta(\Omega)$. Setting $u_j = Sf_j \in  H^2(\Omega)\cap H^{1}_0(\Omega)$,  we know that 
$$
\Vert Sf_j \Vert_{H^{2-\theta}(\Omega)}\leq  C\Vert f_j \Vert_{M_\theta(\Omega)}.
$$ 
Consequently $u_j \rightarrow u$ in $ H^{2-\theta}(\Omega)$ with $\Delta u = f$ in $\Omega$ and $u = 0$ on $\Gamma$. So the operator \eqref{opS} is surjective and then is an isomorphism.
\end{proof}

\begin{remark}\upshape The reader's attention is drawn here to the interpolation argument used above. The fact that  the operators  \eqref{DeltaIsoH10H-1a} and \eqref{DeltaIsoH3/2H1/2ortx} are isomorphisms does not necessarily mean that the interpolated operator is as well. However, the invertibility holds in our case thanks to the density of $H^2(\Omega)\cap H^{1}_{0}(\Omega) $ in $H^{1}_{0}(\Omega)$.
We can find some counterexamples in \cite{FJ} with invertible operator on $ L^p(\mathbb{R})$ and $ L^q(\mathbb{R})$ but not on $ L^r(\mathbb{R})$ for some $r$ between $p$ and $q$. 
\end{remark}

The difficulty now lies in determining the interpolated space $M_\theta(\Omega)$. However the optimal regularity  in the case of polygonal  domains is in fact better (see Remark \ref{rem4}, Point i)  below) and can be expressed as follows:

\begin{theorem} [{\bf $H^s$- Regularity I}]\label{theoHsregul}  Let $\Omega$ be a polygonal domain of $\mathbb{R}^2$ or a polyerdral domain of $\mathbb{R}^2$  that we assume non convex. We denote by $\omega^\star$ the measure of the largest interior angle of $\Omega$ and we set $\alpha^\star = \pi/\omega^\star$. Then\\  
i) for  any $\theta \in \, ] 1 - \alpha^\star, 1[$ with $ \theta \neq 1/2$, the operator
\begin{equation}\label{DeltaIsoH2-thetaH-thetaterz}
 \Delta : H^{2-\theta}(\Omega) \cap H^1_0(\Omega)\longrightarrow {H}^{-\theta}(\Omega)
\end{equation}
is an isomorphism.\smallskip

\noindent ii) For  $ \theta = 1/2$, the operator
\begin{equation}\label{DeltaIsoH2-thetaH-1demibisz}
 \Delta : H^{3/2}_0(\Omega)\longrightarrow  [{H}^{1/2}_{00}(\Omega)]'
\end{equation}
is also an isomorphism (see Proposition \ref{corGrisvardPol} for more information). \smallskip

\noindent iii) We have the following characterizations:
\begin{equation}\label{carainterb}
  M_{\theta}(\Omega) =  \begin{cases} H^{-\theta}(\Omega) \quad \textrm{ if }\; \theta \in \, ] 1 - \alpha^\star,\;  1[\quad \mathrm{with}\;  \; \; \theta \neq 1/2\\
 [H^{1/2}_{00}(\Omega)]' \quad\textrm{ if }\; \theta = 1/2,
 \end{cases}
 \end{equation}
 with equivalent norms. More precisely, there exists a positive constant $C(\Omega)$ which depends on the Lipschitz constant of $\Omega$ and on $\omega^\star$.
\end{theorem}

\begin{proof}  For Point i) and Point ii), see for instance \cite{BBX} and also \cite{D} Theorem 18.13.\smallskip

Using \eqref{DeltaIsoH2-thetaH-thetaterz} and \eqref{DeltaIsoH2-thetaH-1demibisz}, by identification we deduce from \eqref{DeltaIsoH2-thetaH-thetabbis} the  characterization \eqref{carainterb}.
\end{proof}

\begin{remark} \label{rem4}\upshape i) Note that since $\Omega$ is assumed non convex, then $0 < 1 -\alpha^\star < 1/2$.  
When $\alpha^\star$ is close to 1, the domain $\Omega$ is close to be convex and we are near the $H^2$-regularity. Conversely, when $\alpha^\star$ is near 1/2, the  domain $\Omega$ is close to a cracked domain and the expected regularity in this case is better than $ H^{3/2}$. That means that for any nonconvex polygonal domain, there exists $\varepsilon = \varepsilon(\omega^\star)\in \, ]0, 1/2[$ depending on  $\omega^\star$ (in fact, $\varepsilon(\omega^\star) = \alpha^\star - 1/2$) such that for any $ 0 < s < \varepsilon$ and any $f \in H^{-\frac{1}{2} + s}(\Omega)$ the  $H^1_0(\Omega)$ solution of Problem $(\mathscr{L}_D^0)$ belongs to $ H^{3/2 + s}(\Omega)$.\smallskip

\noindent ii) A natural question to ask concerns the limit case $ \theta = 1 - \alpha^\star$, where the regularity $H^{1 + \alpha^\star}(\Omega)$ is in general not achieved for 
RHS $f$ in $H^{-1 + \alpha^\star}(\Omega)$. However it is attained if we suppose $f\in \bigcap_{s > \alpha^\star}H^{-1 + s}(\Omega)$  (see \cite{BBX}). \smallskip

\noindent iii) What happens if $0< \theta < 1 - \alpha^\star$? This question will be examined  a little further on (see Theorem \ref{regpolthetasm}).
\end{remark}

The following proposition is important because it specifies the dependence of one of the constants involved in the equivalence given in Theorem \ref{theoHsregul} Point iii).

\begin{e-proposition}\label{corGrisvardPol}   Let $\Omega$ be a polygonal domain of $\mathbb{R}^2$  or a polyhedral domain of $\mathbb{R}^3$. Then 
\begin{equation}\label{inegPolyH3demi}
\forall v \in H^{2-\theta}(\Omega)\cap H^1_0(\Omega),  \quad \Vert \Delta v \Vert_{H^{-\theta}(\Omega)} \leq \Vert \Delta v \Vert_{M_{\theta}(\Omega)} \leq\Vert v \Vert_{H^{2-\theta}(\Omega)} \leq C(\Omega) \Vert \Delta v \Vert_{M_{\theta}(\Omega)},
\end{equation}
where the contant $C(\Omega)$ above depends only on $\theta$,  on the Poincar\'e constant and the Lipschitz constant  of  $\Omega$.
\end{e-proposition}

\begin{proof}  i) We firstly observe that the first inequality above is immediate. We begin by showing the second inequality. Recall that the following linear operators
\begin{equation*}
\Delta : \; H^{2}(\Omega)\cap H^{1}_{0}(\Omega)\longrightarrow (\mathscr{N}_0(\Omega))^\bot\quad \mathrm{and}\quad \Delta : \; H^{1}_{0}(\Omega)\longrightarrow  H^{-1}(\Omega),
\end{equation*}
where $(\mathscr{N}_0(\Omega))^\bot$ is equipped with the $L^2(\Omega)$ norm, are continuous. Moreover
\begin{equation}\label{dblineg}
\forall v \in H^{2}(\Omega)\cap H^{1}_{0}(\Omega), \quad \Vert \Delta v \Vert_{L^2(\Omega)} \leq \Vert  v \Vert_{H^2(\Omega)}\quad\;  \mathrm{and}\quad \;  
\forall v \in H^{1}_{0}(\Omega), \quad  \Vert \Delta v \Vert_{H^{-1}(\Omega)} \leq \Vert  v \Vert_{H^1(\Omega)}.
\end{equation}
Recall that if $T$ is a linear and continuous operator from $E_0$ into $F_0$ and  from $E_1$ into $F_1$, where $E_j$ and $F_j$ are Banach spaces, then for any $0 < \theta < 1$ the linear operator $T$ is  also continuous  from $E_\theta = [E_0, E_1]_\theta $ into  $F_\theta = [F_0, F_1]_\theta $  and we have the following interpolation inequality: for any $v\in E_\theta$, 
\begin{equation}\label{inegabstinter}
\Vert Tv \Vert_{F_{\theta}} \leq \Vert T \Vert_{\mathscr{L}(E_0; F_0)}^{1 - \theta}  \Vert  T \Vert_{\mathscr{L}(E_1; F_1)}^{\theta}\Vert v \Vert_{E_{\theta}},
\end{equation} 
(see Adams \cite{Ada} page 222, Berg-Lofstr$\mathrm{\ddot{o}}$m \cite{BL} Theorem 4.1.2 and Triebel \cite{Tri} Remark 3 page 63). The first inequality in \eqref{inegPolyH3demi} is then a consequence of \eqref{dblineg} and \eqref{inegabstinter}.\medskip

\noindent ii) One of the key points in establishing the third inequality in \eqref{inegPolyH3demi} is the following equality:
\begin{equation}\label{DeltnorSecond}
\forall v\in  H^2(\Omega)\cap H^{1}_0(\Omega), \quad  \vert v \vert_{H^{2}(\Omega)} = \Vert \nabla^2 v \Vert_{L^{2}(\Omega)} =  \Vert \Delta v \Vert_{L^{2}(\Omega)},
\end{equation}
where $\vert \cdot \vert_{H^{2}(\Omega)} $ denotes the semi-norm $H^{2}(\Omega)$ and  $\nabla^2$ the Hessian matrix.

Equality \eqref{DeltnorSecond} is a direct consequence of the following property due to Grisvard \cite{Gri1} and \cite{Gri2}:
\begin{equation*}
\forall v\in  H^2(\Omega)\cap H^{1}_0(\Omega), \quad \int_\Omega \frac{\partial ^2 v}{\partial x_i^{2}}\, \frac{\partial ^2 v}{\partial x_j^{2}} \, dx=  \int_\Omega \big\vert \frac{\partial ^2 v}{\partial x_i \partial x_j}\big\vert^2 \, dx\qquad \mathrm{for }\; i, j = 1, \ldots, N. 
\end{equation*}
Therefore
\begin{equation*}
\Vert \Delta v \Vert^2_{L^{2}(\Omega)} = \Vert \frac{\partial ^2 v}{\partial x^{2}} \Vert^2_{L^2(\Omega)} + \Vert \frac{\partial ^2 v}{\partial y^{2}} \Vert^2_{L^2(\Omega)} + 2  \int_\Omega  \frac{\partial ^2 v}{\partial x^{2}}  \frac{\partial ^2 v}{\partial y^{2}} \, dx. 
\end{equation*}
Besides,
\begin{equation*}
\int_\Omega \vert \nabla v\vert^2 = - \int_\Omega v \Delta v \leq C_P(\Omega) \Vert \nabla v \Vert_{L^{2}(\Omega)}\Vert \Delta v \Vert_{L^{2}(\Omega)},
\end{equation*}
where $C_P(\Omega)$ is the Poincar\'e constant.
And then
\begin{equation*}
\Vert \nabla v \Vert_{L^{2}(\Omega)} \leq C_P(\Omega)\Vert \Delta v \Vert_{L^{2}(\Omega)}.
\end{equation*}
Finally, we get the following estimate
\begin{equation*}\label{inegisoH2L2}
\Vert v \Vert_{H^{2}(\Omega)}  \leq (1 +  C^2_P(\Omega))^{1/2} \Vert \Delta v \Vert_{L^{2}(\Omega)}.
\end{equation*}

\noindent iii) Let us  introduce the following operators: for any $i, j = 1,\ldots, N$, with $N = 2$ or $N = 3$, we set 
$$
K_{ij} = \frac{\partial^2}{\partial  x_i\partial x_j}\quad \mathrm{and}\quad   L_{ij} = K_{ij} \Delta^{-1},
$$ 
where $\Delta^{-1}$ denotes the inverse of Laplacian as in the proof of Corollary \ref{firstcor}. 

Using  \eqref{DeltnorSecond}  we get, for any $i, j = 1,\ldots, N$, the following
\begin{equation}\label{inegLijgrdade2L2}
\forall f\in (\mathscr{N}_0(\Omega))^\bot,\quad  \Vert  L_{ij}f \Vert_{L^2(\Omega)}\leq \Vert  f \Vert_{L^2(\Omega)}.
\end{equation} 
We know that
\begin{equation}\label{ineggrad2H-1}
\forall v \in {H^{1}_0(\Omega)}, \quad  \Vert \nabla^2 v \Vert_{H^{-1}((\Omega)} \leq \Vert \nabla v \Vert_{L^{2}((\Omega)} \leq  C_P(\Omega)\Vert \Delta v \Vert_{H^{-1}(\Omega)}.
\end{equation} 
From inequality \eqref{ineggrad2H-1} and  isomorphism \eqref{DeltaIsoH10H-1a} we have also
\begin{equation}\label{inegLijgrdade2H-1}
\forall f\in H^{-1}(\Omega),\quad  \Vert  L_{ij}f \Vert_{H^{-1}(\Omega)}\leq  C_P(\Omega)\Vert  f \Vert_{H^{-1}(\Omega)}.
\end{equation} 

We deduce from \eqref{inegabstinter}, \eqref{inegLijgrdade2H-1}, \eqref{inegLijgrdade2L2}, \eqref{carainterb} and \eqref{DeltaIsoH2-thetaH-thetabbis} the following inequalities: for any $f\in M_\theta(\Omega)$ 
\begin{equation*}\label{inegLijgrdade2H-1int}
 \Vert  L_{ij}f \Vert_{H^{-\theta}(\Omega)}\leq  C_P(\Omega)\Vert  f \Vert_{M_\theta(\Omega)}, \; \mathrm{if}\; \theta \neq 1/2\quad \mathrm{and}\quad    \Vert  L_{ij}f \Vert_{(H^{1/2}_{00}(\Omega))'}\leq  C_P(\Omega)\Vert  f \Vert_{M_{1/2}(\Omega)}\;   \mathrm{if}\; \theta = 1/2.
\end{equation*}
Therefore,  we have  the following estimate: for any $v \in H^{2-\theta}(\Omega)\cap H^1_0(\Omega),$
\begin{equation}\label{inegPolyH3demiba}
\Vert \frac{\partial^2 v}{\partial  x_i\partial x_j} \Vert_{H^{-\theta}(\Omega)} \leq C_P(\Omega) \Vert \Delta v \Vert_{M_{\theta}(\Omega)},  \; \mathrm{if}\; \theta \neq 1/2,\; \Vert \frac{\partial^2 v}{\partial  x_i\partial x_j} \Vert_{(H^{1/2}_{00}(\Omega))'} \leq C_P(\Omega) \Vert \Delta v \Vert_{M_{1/2}(\Omega)}\;   \mathrm{if}\; \theta = 1/2.
\end{equation}
Moreover, using  Corollary 3.4 and \eqref{inegPolyH3demiba} we get for any $v \in H^{2-\theta}(\Omega)\cap H^1_0(\Omega),$
\begin{equation*}
 \Vert  v \Vert_{H^{2-\theta}(\Omega)} \leq C_{P, L} (\Omega)\Vert \nabla^2 v \Vert_{H^{-\theta}(\Omega)} \leq C_{P, L} (\Omega)\Vert \Delta v \Vert_{M_{\theta}(\Omega)} \quad \mathrm{if}\; \theta \neq 1/2
\end{equation*}
and 
\begin{equation*}
 \Vert  v \Vert_{H^{3/2}(\Omega)} \leq C_{P, L} (\Omega)\Vert \nabla^2 v \Vert_{[H^{1/2}_{00}(\Omega)]'} \leq C_{P, L} (\Omega)\Vert \Delta v \Vert_{M_{1/2}(\Omega)} \quad \mathrm{if}\; \theta = 1/2.
\end{equation*}
That concludes the proof.\end{proof}

We are now in a position to extend Theorem \ref{theoHsregul}  to the case where $0< \theta \leq 1 - \alpha^\star$. 
We begin by considering the case more simple  where the domain $\Omega$ is a polygon with only one angle having a measure $\omega^\star$ larger than $\pi$ and of vertex $A$. We know from \eqref{dimang} that the kernel $ \mathscr{N}_{-\theta}(\Omega) $ is of dimension $1$, say $\mathscr{N}_{-\theta}(\Omega) = \langle z\rangle $. 

\begin{theorem} [{\bf $H^s$- Regularity II}] \label{regpolthetasm} Let $\Omega$ be a polygon with only one angle having a measure $\omega^\star$ larger than $\pi$. Then,\\
i) for  any $0 \leq \theta < 1 - \alpha^\star, $ the operator
\begin{equation}\label{DeltaIsoH2-thetaH-thetabis}
 \Delta : H^{2- \theta}(\Omega) \cap H^1_0(\Omega)\longrightarrow \langle z \rangle^\bot
\end{equation}
is an isomorphism, where $
 \langle z \rangle^\bot = \{\varphi \in H^{-\theta}(\Omega); \; \langle \varphi, z\rangle = 0\}.$\smallskip

\noindent ii) For $\theta = 1 - \alpha^\star, $ the operator
\begin{equation*}\label{CC}
 \Delta : H^{1 + \alpha^\star}(\Omega) \cap H^1_0(\Omega)\longrightarrow M_{1 - \alpha^\star}(\Omega)
\end{equation*}
is an isomorphism. Moreover  
$$
\bigcap_{r < 1 - \alpha^\star}H^{-r}(\Omega) \hookrightarrow M_{1 - \alpha^\star}(\Omega)\hookrightarrow H^{-1 + \alpha^\star}(\Omega),
$$
where the topology of $ M_{1 - \alpha^\star}(\Omega)$ is finer than that of $ H^{-1 + \alpha^\star}(\Omega)$.
\end{theorem}

\begin{proof} First, let's remember that for any $0 < \theta < 1$ we have the  isomorphism \eqref{DeltaIsoH2-thetaH-thetabbis}.

Our goal is to prove that  $\sigma_0 = \sigma_1= 1 - \alpha^\star$, with the same notations as in Theorem \ref{thmIK}. The above results will then be a consequence of Theorem \ref{theoHsregul}, Corollary \ref{firstcor} and the characterization of Kernel $\mathscr{N}_{-\theta}(\Omega)$, which is of dimension 1 for all $0\leq \theta <  1 - \alpha^\star$ and reduced to $\{0\}$ when  $ \theta >  1 - \alpha^\star$.\medskip

First, thanks to Theorem \ref{theoHsregul},  we have $\sigma_1 \leq  1 - \alpha^\star$. Second, using Point ii) of Theorem \ref{theoGrisvardPol} and Point ii) of Theorem  \ref{thmIK} we deduce that $\sigma_0 > 0$. In fact, the value of $\sigma_0 > 0$ is directly related to the function $z$ of the kernel $\mathscr{N}_{-\theta}(\Omega)$, when $0\leq \theta <  1 - \alpha^\star$, which belongs to the Besov space $[H^1(\Omega), L^2(\Omega)]_{\alpha^{\star}, \infty}$, but not to $H^{1 - \alpha^\star}(\Omega)$. So the only possible value of  $\sigma_0$ is $ 1 - \alpha^\star$ and then $\sigma_0 = \sigma_1= 1 - \alpha^\star$.  Using then the isomorphism \eqref {DeltaIsoH2-thetaH-thetab}, the identity \eqref{MKer} and  Theorem \ref{thmIK}, we conclude that:
$$
M_\theta(\Omega) = \langle z \rangle^\bot\quad \mathrm{for\, any}\quad 0 \leq \theta < 1 - \alpha^\star.
$$
Moreover the norm of $M_\theta(\Omega)$ is equivalent to the norm of $H^{-\theta}(\Omega)$, where the corresponding constants involved in this equivalence depend on $ \alpha^\star$.\medskip

Now, concerning the critical value $1 - \alpha^\star$ and according to Theorem \ref{thmIK}, Point ii), the norm of the space $M_{1 - \alpha^\star}(\Omega)$ is not equivalent to the norm of  $ H^{-1 + \alpha^\star}(\Omega)$. So all the properties stated in the theorem above take place.\end{proof}

\begin{remark}\upshape \noindent  i) In \cite{BBX}, Theorem 4.1,  the authors proved that for any $f$ belonging to some subspace of the Besov space $[H^1(\Omega), L^2(\Omega)]_{\alpha^{\star}, \infty}$, where $\Omega$ satisfies the assumptions of the above theorem, the $ H^1_0(\Omega)$ function $u$ satisfying  $\Delta u = f$ in $\Omega$ is in fact in the Besov space $[H^2(\Omega), H^1(\Omega)]_{\alpha^{\star}, \infty}$ which contains the space $H^{1 + \alpha^\star}(\Omega) \cap H^1_0(\Omega)$. That means that our result in Point ii) above is little bit better.
\smallskip

\noindent ii) For $f$ given in ${H}^{-s}(\Omega)$ with $0 \leq s < 1 - \alpha^\star$, let $u$ be the solution $H^1_0(\Omega)$ of the problem $(\mathscr{L}_D^0)$.  In fact $u\in  H^{3/2}_0(\Omega)$ and we know that $u$ is more regular outside a neighborhood $V$ of the vertex $A$: precizely $u \in H^{2- s}(\Omega\setminus V) $. It is convenient to introduce polar coordinates $(r, \theta)$ centered at $A$ such that the two sides of the angle correspond to $\theta = 0$ and $\theta = \omega^\star$. Let us introduce the following function:
$$
S(r, \theta) = \eta(r) r^{\alpha^\star} sin(\alpha^\star \theta) 
$$
where $\alpha^\star = \pi/\omega^\star $ and $\eta$ is a regular cut-off function defined on $\mathbb{R}^+$, equal to 1 near zero, equal to zero in some interval $ [a, \infty[$, with some small $a > 0$. Observe that $S \in H^t(\Omega)$ for any $t < 1+ \alpha^\star$ and  $\Delta S \in H^t(\Omega)$ for any $t <  \alpha^\star$. The function $w = u - \lambda S$, with $\lambda$ constant to be determined later, satisfies $\Delta w  \in {H}^{-s}(\Omega)$.  If $\langle f, z \rangle = 0$, we deduce from \eqref{DeltaIsoH2-thetaH-thetabis} that $u\in H^{2-s}(\Omega)$. If  $\langle f, z \rangle \neq 0$, we choose $\lambda =  \langle \Delta S, z \rangle/\langle f, z \rangle$. So $\langle \Delta w, z \rangle = 0$ and then  $w\in H^{2-s}(\Omega)$. That means that $u - \lambda S$ belongs to  $H^{2-s}(\Omega)$.\smallskip

\noindent iii) In \cite{J-K}, the authors recall that for any $s > 3/2$ there is a Lipschitz domain $\Omega$ and $f\in \mathscr{C}^\infty(\overline{\Omega})$ such that the solution $u$ to the inhomogeneous Dirichlet problem $(\mathscr{L}_D)$  does not belong to $H^{s}(\Omega)$. Point i) above shows that the compatibility condition is the hypothesis on $f$ that allows us to obtain the expected regularity $H^{s}(\Omega)$.
\end{remark}

We will  now consider  the case  where the domain $\Omega$ is a polygon having  $n$ angles  $\omega_1, \ldots, \omega_n$, with $n\geq 2$, larger than $\pi$. We suppose that $\omega_1 \leq \ldots \leq \omega_n$. We know from \eqref{eqdim} that the kernel $ \mathscr{N}_{-\theta}(\Omega) $ is of dimension $n$, say $\mathscr{N}_{-\theta}(\Omega) = \langle z_1, \ldots, z_n\rangle $.

\begin{theorem} [{\bf $H^s$- Regularity III}] \label{regpolthetasm+}  Let $\Omega$ be a polygon.  We denote by  $\omega_1, \ldots, \omega_n$, with $n\geq 2$, all angles larger than $\pi$ and suppose $\omega_1 \leq \ldots \leq \omega_n$.  Setting  $\alpha_k = \pi/\omega_k$ for $k = 1, \ldots, n$ and by convention $\alpha_0 = 1$, then\\
i) for  any $\theta \in \, ] 1 - \alpha_n, 1[$ with $ \theta \neq 1/2$, the operator 
$$
 \Delta : H^{2-\theta}(\Omega) \cap H^1_0(\Omega)\longrightarrow {H}^{-\theta}(\Omega)
 $$ 
 is an isomorphism.\\
ii) For  $ \theta = 1/2$, the operator 
$$
\Delta : H^{3/2}_0(\Omega)\longrightarrow  [{H}^{1/2}_{00}(\Omega)]'
$$ 
is an isomorphism.\\
iii) For any fixed $k = 0, \ldots, n-1$, $\theta \in \, ] 1 - \alpha_k, 1 - \alpha_{k+1}[$ and $f\in H^{-\theta}(\Omega)$ satisfying the following compatibility condition
\begin{equation*}
\forall \varphi \in \langle z_{k+1}, \ldots, z_{n}\rangle, \quad \langle f, \varphi\rangle = 0, 
\end{equation*}
Problem  $(\mathscr{L}_D^0)$ has a unique solution $u \in H^{2-\theta}(\Omega)$.\\
\noindent iv) We have the following characterizations:
\begin{equation}\label{carainterb}
  M_{\theta}(\Omega) =  \begin{cases} H^{-\theta}(\Omega) \quad \textrm{ if }\; \theta \in \, ] 1 - \alpha_n,\;  1[\quad if \;  \; \; \theta  \in \, ] 1 - \alpha_n, 1[,\\
 \{f\in H^{-\theta}(\Omega); \;  \langle f, \varphi\rangle = 0, \; \forall \varphi \in \langle z_{k+1}, \ldots, z_{n} \} \quad if \;  \; \; \theta  \in \, ] 1 - \alpha_k, 1 - \alpha_{k+1}[
 \end{cases}
 \end{equation}
 with equivalent norms. More precisely, there exists a positive constant $C(\Omega)$ which depends on the Lipschitz constant of $\Omega$ and on $\omega_1, \ldots, \omega_n$.

\end{theorem}

We do not provide a proof of this theorem. It is similar to that of Theorem \ref{regpolthetasm}  and makes use of Theorem \ref{thmACK}.

\begin{remark}\upshape i) For critical values $\theta = 1 - \alpha_k $, with $k = 1, \ldots, n$, the operator $
 \Delta : H^{1 + \alpha_k}(\Omega) \cap H^1_0(\Omega)\longrightarrow M_{1 - \alpha_k}(\Omega)$ 
is an isomorphism. Moreover  
$$
\bigcap_{r < 1 - \alpha_k}H^{-r}(\Omega) \hookrightarrow M_{1 - \alpha_k}(\Omega)\hookrightarrow H^{-1 + \alpha_k}(\Omega),$$
where the topology of $ M_{1 - \alpha_k}(\Omega)$ is finer than that of $ H^{-1 + \alpha_k}(\Omega)$.\smallskip

\noindent{ii)} For $\theta \in \, ] 1 - \alpha_n, 1[$ we have the following estimate: there exists a constant $C$ depending on $\alpha_n$ such that
\begin{equation}\label{equiNormPolyn}
\forall  f\in M_{\theta}(\Omega), \quad \Vert f \Vert_{M_{\theta}(\Omega)} \leq C \Vert f \Vert_{H^{-\theta}(\Omega)}.
\end{equation}
And for  $\theta \in \, ] 1 - \alpha_k, 1 - \alpha_{k+1}[$, with  $k = 0, \ldots, n-1$ we have the following estimate: there exists a constant $C$ depending on $\alpha_k$ and $\alpha_{k+1}$ such that inequality \eqref{equiNormPolyn} holds. So, as in the proof of Proposition \ref{corGrisvardPol}, we deduce the following inequalities: if $\theta \in \,  ] 1 - \alpha_n, 1[$, respectively  $\theta \in \, ] 1 - \alpha_k, 1 - \alpha_{k+1}[$, for some  $k = 0, \ldots, n-1$, then
\begin{equation}\label{const}
\forall v\in H^{2-\theta}(\Omega) \cap H^1_0(\Omega), \quad \Vert v \Vert_{H^{2-\theta}(\Omega)} \leq C \Vert \Delta v \Vert_{H^{-\theta}(\Omega)}, 
\end{equation}
where the constant $C$ depends on the Poincar\'e constant of $\Omega$, on  the Lipschitz constant of $\Omega$  and on $ \alpha_n$, respectively on  $\alpha_k$ and $  \alpha_{k+1}$.\smallskip

\noindent{iii)} For any $s >0$, there exists a polygon and $f\in H^{s - 1/2}(\Omega)$ such that the $H^1_0(\Omega)$ solution $u$ to Problem $(\mathscr{L}^0_D)$ does not belong to  $H^{s + 3/2}(\Omega)$. This result is valid for any $N \geq 3$.

\end{remark} \medskip

We conclude this subsection with the following remark concerning the constants involved in the equivalence of the norms of $M_{\theta}(\Omega)$ and $[\widetilde{H}^{\theta}(\Omega)]'$.

\begin{remark}\upshape Let $\Omega$ be a polygon and $(\Omega_k)_k$ an increasing sequence of polygons included in $\Omega$ and converging to $\Omega$. We can choose this sequence such that the sequence $(L_k)_k$ of Lipschitz constants of $\Omega_k$ is less or equal to the Lipschitz constant of $\Omega$.

Let  $\varphi \in \mathscr{D}(\Omega)$ and $k_0$ such that supp $\varphi \subset \Omega_{k_0}$.  We know that there exists a positive constant $C(\Omega)$ such that 
$$
\Vert \Delta \varphi \Vert_{[H^{1/2}_{00}(\Omega)]'}  \leq \Vert \Delta \varphi \Vert_{M(\Omega)} = (1 +  C(\Omega)\Vert \Delta \varphi \Vert_{[H^{1/2}_{00}(\Omega)]'},
$$
where $C(\Omega)$ depends on $z_1, \ldots, z_n$. The question we are now asking is how the norm $ \Vert \Delta \varphi \Vert_{M(\Omega_k)} $ evolves when k tends to infinity?

As above, we have also the following estimate: for any $k \geq k_0$, 
$$
\Vert \Delta \varphi \Vert_{[H^{1/2}_{00}(\Omega_k)]'}  \leq \Vert \Delta \varphi \Vert_{M(\Omega_k)} = (1 +  C(\Omega_k)\Vert \Delta \varphi \Vert_{[H^{1/2}_{00}(\Omega_k)]'},
$$
where $C(\Omega_k)$ depends on the harmonic functions of the kernel $\mathscr{N}_0(\Omega_k)$. However
$$
\Vert \Delta \varphi \Vert_{[H^{1/2}_{00}(\Omega_k)]'}  \rightarrow \Vert \Delta \varphi \Vert_{[H^{1/2}_{00}(\Omega)]'} \quad \mathrm{and} \quad  \Vert \Delta \varphi \Vert_{M(\Omega_k)}  \rightarrow \Vert \Delta \varphi \Vert_{M(\Omega)}\quad \mathrm{ as}\;  k \rightarrow \infty.
$$ 
 So we deduce that $ C(\Omega_k)$ tends to $C(\Omega)$ when $k$ tends to $\infty$. 
\end{remark}

\section{Uniqueness and regularity in $L^{p}$-theory}\label{secuniqueness}

The section is devoted to revisit some uniqueness and regularity results in $L^p$-theory. Before a more in-depth study, we will focus on three results that are widely used and cited in the literature, concerning the uniqueness and existence of solutions $u \in L^{p}_s(\Omega)$ when $f \in  L_{s-2}^p(\Omega)$ for the problem $(\mathscr{L}_D^0)$. These results are presented in the article \cite{J-K}. Unfortunately, as stated, we will see that they are partially or completely incorrect. Before introduicing them, let's recall the notations and definitions of the spaces used.

For any $- \infty < s < \infty$ and $1 \leq p \leq \infty$, 
$$
L_s^p(\mathbb{R}^N) = \{(I - \Delta)^{-s/2}f; \; f \in L^p(\mathbb{R}^N) \} \quad \mathrm{and}\quad  \mathrm{for}\; s \geq 0, \;  L_s^p(\Omega) = \{v_{\vert\Omega}; \; v \in L_s^ p(\mathbb{R}^N) \}.
$$
$$
 \mathrm{for}\; s >  0, \; 1 < p < \infty,   \quad L_{s,0}^p(\Omega) = \{v\in  L_s^p(\mathbb{R}^N) ; \;  \mathrm{supp}\, v \subset \overline{\Omega}  \} \quad \mathrm{and}\quad  L_{-s}^p(\Omega) = [L_{s,0}^{p'}(\Omega)]'.
$$
We also introduce the following Sobolev spaces: for $m \in \mathbb{N},\;  1 \leq p \leq \infty$,
$$
 W^{m, p}(\Omega) = \{ v \in \mathscr{D}'(\Omega);  D^\lambda v \in L^p(\Omega), \, \forall \vert \lambda \vert \leq m\}
$$
and for $s = m + \sigma$, with $0 < \sigma < 1, \; 1  \leq p < \infty$, 
$$
W^{s, p}(\Omega) = \{ v \in  W^{m, p}(\Omega);   \sum_{\vert\lambda\vert = m}\int_\Omega\int_\Omega\frac{\vert D^\lambda u(x) - D^\lambda u(y)\vert^p}{\vert x - y \vert^{N + p\sigma}}dx\,dy < \infty \},
$$
which is a Banach space with the following norm:
$$
\Vert u\Vert_{W^{s, p}(\Omega)} = (\Vert u\Vert_{W^{m, p}(\Omega)}^p+ \sum_{\vert\lambda\vert = m}\int_\Omega\int_\Omega\frac{\vert D^\lambda u(x) - D^\lambda u(y)\vert^p}{\vert x - y \vert^{N + p\sigma}}dx\,dy)^{1/p}.
$$
We recall that for any $s \geq 0$ and $1 < p < \infty$, 
$$
 L_s^p(\Omega) = W^{s, p}(\Omega)  \quad\mathrm{when}\;  s \in \mathbb{N}\quad \mathrm{or \;\;  when }\;  p = 2
$$
and
$$
W^{s, p}(\Omega) \hookrightarrow L_s^p(\Omega) \quad\mathrm{when}\; p \leq 2 \quad \mathrm{and}\quad   L_s^p(\Omega)\hookrightarrow W^{s, p}(\Omega) \quad\mathrm{when}\; p \geq 2.
$$

\subsection{Two explicit examples for harmonic kernels}

Unlike the case where the domain $\Omega$ is regular, the kernel space of harmonic functions $W^{s, p}(\Omega)$ that are zero on the boundary of $\Omega$, we denote by $W^{s, p}_0(\Omega)\cap \mathscr{H}$,  is generally non-trivial for a large collection of $s$ and $p$, as shown by the explicit examples below.

\begin{example} [\bf{ 2D-Case}]\label{exexpl2D}\upshape
Let us  consider the following domain:
\begin{equation}\label{domOmeg}
\Omega = \{(r, \theta);\; 0 < r < 1,\quad 0 < \theta < \pi /a\}, \quad \mathrm{with}\; a > 1/2\quad  \mathrm{and}\quad  \mathrm{close \, \, to } \, \, 1/2.
\end{equation}
We can easily verify that the following function
\begin{equation}\label{sing}
z(r, \theta) = (r^{-a} - r^{a})\mathrm{sin}(a\theta)
\end{equation}
is harmonic in $\Omega$ with $z = 0$ on $\Gamma$. Moreover, given $s$ in the interval $[0, 3/2[$, it is easy to verify that the function $z$ belongs to $ L^p_s(\Omega) $ for any $p$ such that $1< p < 2/ (a + s)$. In particular for $s = 1$,  we have $z \in L^p_1(\Omega) =  W^{1, p}(\Omega)$ for any $p$ such that $1< p < 2/ (a + 1)$,  where $2/ (a + 1)$ is strictly less than  $4/3 $ but close  to $4/3$. We can also see  that $z$ belongs to $W^{3/2 - \varepsilon, 1}(\Omega)$ for any $\varepsilon > a - 1/2$ and $z \in L^q(\Omega)$ for any $q  < 2/a$.  
\end{example}

\begin{example} [\bf{3D-Case}]\label{exexpl3D}\upshape
Let us  consider the following domain: 
\begin{equation}\label{domOmeg3}
\Omega = \{(r, \theta, \varphi);\; 0 < r < 1,\quad 0 < \theta < \theta_0, \; 0 < \varphi < 2\pi\}, \quad \mathrm{with}\; \theta_0  < \pi\quad  \mathrm{and}\quad  \mathrm{close \, \, to } \, \, \pi.
\end{equation}
Our goal in this example is to find  as above harmonic functions in some $W^{1, p}(\Omega)$, equal to zero on the boundary of $\Omega$, under the form $v(r, \theta, \varphi) = c(r) w(\theta)$. So we need to impose the conditions $c(1) = 0$ and $w(\theta_0) = 0$. For that we proceed as in \cite{Gri2}. The harmonicity of $v$ implies that there exists $\lambda = \lambda(\theta_0) > 0$ such that
\begin{equation}\label{eqclambda}
\partial_r(r^2 \partial_r c) = \lambda c, \quad 0 < r < 1, \quad \mathrm{and} \quad Lw = - \lambda w, \quad 0 < \theta < \theta_0 .
\end{equation}
where $L$ is the Beltrami Laplacian on the sphere, which is here given by $Lw =  (\mathrm{sin}\, \theta)^{-1} \partial_\theta(\mathrm{sin} \, \theta\, \partial_\theta w)$. Setting $I =\,  ]0, \theta_0[$ and 
$$
H(I) = \{ \psi\in  \mathscr{D}'(I); \;   (\mathrm{sin} \, \theta)^{1/2} w \in L^2(I), \; (\mathrm{sin} \, \theta)^{1/2} w' \in L^2(I), \; w(\theta_0) = 0\},
$$
we know that there exist a sequence $(w_k)_{k\geq 1}$ of $H(I)$ and an increasing sequence $(\lambda_k)_{k\geq 1}$ of reals with $\lambda_k =  \lambda_k(\theta_0) > 0$ for any $ k \geq 1$ and $\lambda_k \rightarrow \infty$ such that
$$
\forall \psi \in H(I), \quad \int_I ( \mathrm{sin} \, \theta\, \,  w_k' \psi' - \lambda_k \,  \mathrm{sin} \, \theta\, \,  w_k \psi) d\theta = 0.
$$
So $w_k$ is an eigenvector of $- L$ associated to the eigenvalue $\lambda_k$. The smallest eigenvalue $\lambda_1(\theta_0)$ satisfies:
$$
\lambda_1(\theta_0) = \min_{w \in H(I), \,  w \neq 0}\frac{\displaystyle \int_I  \mathrm{sin} \, \theta\, \, \vert w'\vert^2}{ \displaystyle\int_I  \mathrm{sin} \, \theta \, \,  \vert w\vert^2}.
$$
Moreover $\lambda_1(\theta_0) \searrow 0$ when $\theta_0 \nearrow \pi.$ 

Setting now $v_k(r, \theta, \varphi) = c_k(r) w_k(\theta)$ and using  \eqref{eqclambda} with $\lambda = \lambda_k$ and the condition $c_k(1) = 0$, we deduce that 
$$
r^2 c_k'' + 2r c_k' - \lambda_k c_k = 0 \quad \mathrm{in}\; ]0, 1[ \quad \mathrm{and}\quad c_k(r) = a_k(r^{\alpha_k} - r^{\beta_k})
$$
where
\begin{equation}\label{alphak}
\alpha_k = \frac{-1 - \sqrt{1 + 4 \lambda_k}}{2}, \qquad \beta_k = \frac{-1 + \sqrt{1 + 4 \lambda_k}}{2}\qquad \mathrm{for}\quad k \geq 1
\end{equation}
and $a_k$ is an arbitrary constant.

Now, we can  verify that the following harmonic function in $\Omega$ and equal to zero on $\Gamma$
\begin{equation}\label{noyau3D}
v_1(r, \theta, \varphi) = (r^{\alpha_1} - r^{\beta_1}) w_1(\theta) 
\end{equation} 
belongs to $ W^{1, p}(\Omega)$ for any $1\leq p <  p_0(\theta_0)$, where 
\begin{equation*}\label{p0}
p_0(\theta_0) = 3/(1 - \alpha_1)  < 3/2 \quad \mathrm{and}\quad  \lim_{\theta_0\rightarrow \pi}p_0(\theta_0) = 3/2.
\end{equation*}
 The function $v_1$ satisfies \eqref{UniqJK}, with $\alpha = 1$ and  $1\leq p < p_0(\theta_0)$, but is not identically zero. 
\end{example}

The fact that these kernels are generally nontrivial leads to several consequences. 
 One of them concerns the inverse inequality of the trace operator, as we will see below.

Let us start by recalling the following result established by Jerison and Kenig in \cite{J-K}, Proposition 5.17 and which concerns a uniqueness criterion in space $L^{p}_s(\Omega)$ for the problem $(\mathscr{L}_D^0)$.

\begin{e-proposition}[Jerison-Kenig,  \cite{J-K}] Let $\Omega$ be a bounded Lipschitz domain and suppose that $v$ satisfies:
\begin{equation}\label{UniqJK}
\Delta v = 0\; \mathrm{ in} \; \Omega, \quad Tr\,  v = 0\; \mathrm{ on }\;  \Gamma\quad \mathrm{ and}\quad v \in L^p_\alpha(\Omega)\; \; \mathrm{ for\,  some}\;  \alpha > 1/p.
\end{equation}
Then $v$ is identically zero. 
\end{e-proposition}

 This result is in fact true for $\mathscr{C}^1$ domains but not for Lipschitz domains in general, as shown in the Example \ref{exexpl2D}  with  $\alpha = 1$ and $1 < p <   2/ (a + 1)$ and in the Example \ref{exexpl3D}  with  $\alpha = 1$ and $1 < p <   p_0(\theta_0))$.

To prove this uniqueness result, Jerison and Kenig first write 
\begin{equation}\label{estim1JKPro517}
\Vert v_{\vert \Gamma_j}\Vert_{B^p_{\alpha - 1/p}(\Gamma_j)} = \Vert (\varphi_j - v)_{\vert \Gamma_j}\Vert_{B^p_{\alpha - 1/p}(\Gamma_j)} \leq C \Vert  (\varphi_j - v)\Vert_{L^p_\alpha(\Omega_j)} \rightarrow 0 \quad \mathrm{as}\; j \rightarrow \infty,
\end{equation}
where $\varphi_j \in \mathscr{D}(\Omega)$ is such that $\varphi_j$ tends to $v$ in $L^p_\alpha(\Omega)$
and then used the estimate
\begin{equation}\label{estim2JKPro517}
\Vert v \Vert_{L^p_\alpha(\Omega_j)}\leq C \Vert v_{\vert \Gamma_j}\Vert_{B^p_{\alpha - 1/p}(\Gamma_j)},
\end{equation}
where $\Omega_j$ is a smooth domain such that $\overline{\Omega_j} \subset \Omega$ and its boundary $\Gamma_j$ is uniformly Lipschitz and tends to $\Gamma$. This estimate  \eqref{estim2JKPro517} is true, but since Example \ref{exexpl2D} and Example \ref{exexpl3D} the constant $C$ must depend on $ j$ and tend to  $\infty$ as $ j \rightarrow \infty$.

If we take $\alpha = 1$, for example, recall that for any bounded domain $\Omega$ of class $\mathscr{C}^1$, there exists a positive constant $C(\Omega)$ such that for all $1 < p < \infty$,
\begin{equation}\label{inegC1trace1}
\forall v \in W^{1, p}(\Omega)\cap \mathscr{H}, \quad \Vert v \Vert_{W^{1, p}(\Omega)} \leq C(\Omega) \Vert v \Vert_{W^{1-1/p, p}(\Gamma)}. 
\end{equation}
Due to this non trivial kernel $ W^{1, p}_0(\Omega)\cap \mathscr{H}$, for $p$ strictly less than some value $p_0(\Omega) < 2N/(N+ 1)$, the above estimate is clearly not true in general when $\Omega$ is only Lipschitz and $1 < p < p_0(\Omega)$. However, for  bounded Lipschitz domains, we have
\begin{equation}\label{inegC1trace2}
\forall v \in W^{1, p}(\Omega)\cap \mathscr{H}, \quad \inf_{z \in W^{1, p}_0(\Omega)\cap \mathscr{H}} \Vert v + z  \Vert_{W^{1, p}(\Omega)} \leq C(\Omega) \Vert v \Vert_{W^{1-1/p, p}(\Gamma)}. 
\end{equation} 
Similar inequalities take place for $v \in W^{s, p}(\Omega)\cap \mathscr{H}$ with a large collection of $s$ and $p$.

This shows that this issue is poorly understood. Unfortunately, this lack of understanding has led to existence and regularity results that are only partially valid for the study of the problem $(\mathscr{L}_D)$.\medskip

\subsection{Uniqueness Criteria}

The following theorem provides us with the first uniqueness criterion. This result improves the classical uniqueness result when $u\in H^s(\Omega)$, with $s > 1/2$. 

\begin{theorem}\label{UniciteH1/2} If $u\in H^{1/2}(\Omega)$ is harmonic and $u=0$ on $\Gamma$, then $u = 0$ in $\Omega$. 
\end{theorem}

\begin{proof} Without loss of generality, we can assume that N = 2. Let $\Omega_k$ be an increasing sequence of polygons 
which converges to the domain $\Omega$.  From Theorem \ref{regpolthetasm+} Point i), we know that for any $k \geq 1$, there exists $\psi_k \in H^{s_k + 3/2}(\Omega_k)\cap H^1_0(\Omega_k)$, for some $s _k> 0$ depending on $\Omega_k$, such that $\Delta \psi_k = u $ in $\Omega_k$, with the following estimate:
\begin{equation*}
\Vert \psi_k\Vert_{H^{s_{k } + 3/2}(\Omega_k)} \leq C(\Omega_k) \Vert u \Vert_{L^2(\Omega_k)}\leq C(\Omega_k) \Vert u \Vert_{L^2(\Omega)}
\end{equation*}
where  the constant  $C(\Omega_{k})$ is bounded by a constant $C(\Omega)$ which depends only on the Lipschitz constant and on the Poincar\'e constant of $\Omega$, see \eqref{const}. Furthermore, by the classical Ne\v{c}as Property (see Theorem \ref{NecProi}), we know that
\begin{equation*}
 \Vert\partial_{\textit{\textbf{n}}_{k}} \psi_k\Vert_{\textit{\textbf L}^2(\Gamma_{k})}\, \leq\ C(\Omega_k)(\Vert \psi_k\Vert_{H^1(\Omega_{k)}} + \Vert \Delta \psi_k \Vert_{L^2(\Omega_{k})}) 
\end{equation*}
and then
\begin{equation*}
 \Vert\partial_{\textit{\textbf{n}}_{k}} \psi_k\Vert_{\textit{\textbf L}^2(\Gamma_{k})}\, \leq\ C(\Omega) \Vert u \Vert_{L^2(\Omega)}. 
\end{equation*}
So as in the proof of Theorem \ref{CoruH10DeltaudualH1/2} below, Step 3, we can deduce that there exists $\Lambda \in L^2(\Gamma)$ such that
\begin{equation}\label{limintG}
\forall g \in L^2(\Gamma), \quad \lim_{k\rightarrow \infty}\int_{\Gamma_{k}} g\partial_{\textit{\textbf{n}}_{k}} \psi_k\, d\sigma_k = \int_\Gamma g\Lambda\, d\sigma.
\end{equation}
Note that
\begin{equation}\label{seqpsi}
\int_{\Omega_{k}} \vert u \vert^2\, dx = \int_{\Omega_{k}}  u \Delta \psi_k\, dx =  \int_{\Gamma_{k}} u \partial_{\textit{\textbf{n}}_{k}} \psi_k\, d\sigma_k .
\end{equation}

We will prove that  $u = 0$ a.e in $\Omega$, which will allow us to deduce the required uniqueness. For that let us consider the following subsets of $\Omega$: 
\begin{equation*}
\omega = \{ x \in \Omega; \; u(x) \neq 0\}\quad \mathrm{and}\quad \omega_k= \omega \cap \Omega_k 
\end{equation*}
and suppose that the measure $\vert \omega\vert $ of $\omega$ is different to zero. As the sequence $(\vert \omega_k\vert)_k$ converges to $\vert \omega\vert $, we clearly have
\begin{equation}\label{pos}
\lim_{k\rightarrow \infty} \int_{\Omega_{k}} \vert u \vert^2 \, dx= \int_{\Omega} \vert u \vert^2 \, dx> 0.
\end{equation}
Now passing to the limit in \eqref{seqpsi},  we deduce from \eqref{limintG} and the fact that $u = 0$ on $\Gamma$ that  
\begin{equation*}
\int_{\Omega} \vert u \vert^2 \, dx=  0,
\end{equation*}  
which contradicts \eqref{pos}. So $\vert \omega\vert = 0 $ and then $u = 0$ a.e in $\Omega$.
\end{proof} 

Recall now that if $\Omega $ is of class $\mathscr{C}^{1}$ or convex and $1 < p < \infty$, then for any  
$$
f\in W^{-1,\, p}(\Omega), \quad \mathrm{and} \quad g\in W^{1 - 1/p,\, p}(\Gamma),
$$
the following problem 
$$
(\mathscr{L}_D)\ \ \ \  -\Delta u = f\quad \ \mbox{in}\ \Omega \quad
\mbox{and } \quad u = g \ \ \mbox{on }\Gamma
$$
has a unique solution $u\in W^{1,\, p}(\Omega).$ On the other hand, as we saw above, this is not the case in general when the domain $\Omega$ is only Lipschitz. 

We will use in the sequel the following notation
$$
X \cap \mathscr{H}= \{\varphi \in X; \; \Delta \varphi = 0 \; \mathrm{in}\; \Omega \}.
$$

\begin{theorem}  i)  Let $\Omega$ be a bounded Lipschitz  domain of $\mathbb{R}^N$ with $N \geq 2$. Then 
\begin{equation}\label{carkerdim1}
W^{1, 2N/(N+1)}_0(\Omega) \cap \mathscr{H}= \{0\}
\end{equation}
More generally, 
\begin{equation}\label{carkerdim2}
\forall \, 1/2 < \alpha < (N+1)/2, \quad 
W^{\alpha, 2N/(N+2\alpha-1)}_0(\Omega) \cap \mathscr{H}= \{0\},
\end{equation}
or equivalently
\begin{equation}\label{carkerdim3}
\forall \, 1 < p <  2, \quad 
W^{1/2 + N(1/p - 1/2), \, p}_0(\Omega) \cap \mathscr{H}= \{0\}.
\end{equation}
\noindent ii) For any $ p <  2N/(N+1)$, there exists a bounded Lipschitz  domain $\Omega$ such that
\begin{equation}\label{carkerdim4}
W^{1, p}_0(\Omega) \cap \mathscr{H}\neq  \{0\}.
\end{equation}
More generally, let $1 < p < 2$ be fixed. For any $1 < r < p$, there exists a bounded Lipschitz  domain $\Omega$ such that
\begin{equation}\label{carkerdim5}
W^{1/2 + N(1/p - 1/2), \, r}_0(\Omega) \cap \mathscr{H}\neq \{0\}.
\end{equation}
\noindent iii) For any bounded Lipschitz  domain $\Omega$ of $\mathbb{R}^N$, with $N \geq 2$, there exists $p_0(\Omega) < 2N/(N+1)$ such that
\begin{equation}\label{carkerLp1b}
W^{1, p}_0(\Omega) \cap \mathscr{H} =  \{0\} \, \; \mathrm{if} \; \; p \geq p_0(\Omega)\quad  and  \quad W^{1, p}_0(\Omega) \cap \mathscr{H}\neq  \{0\} \, \; \mathrm{if} \; \; 1< p < p_0(\Omega).
\end{equation}
When $\Omega$ is $\mathscr{C}^1$ or convex,  the exponent $p_0(\Omega)$ may be taken to be equal $1$. 

\noindent iv) Moreover if $\Omega$ is a polygonal (resp. 3D polyedral), we have
\begin{equation}\label{defp02D}
p_0(\Omega) = 2\omega^\star/(\pi + \omega^\star) \quad \mathrm{with} \quad \omega^\star < 2\pi \; \mathrm{is \, \, the\, \,  larger\, \,  angle\, \,   of} \; \Omega, \quad \mathrm{for}\; N = 2. 
\end{equation}
and
\begin{equation}\label{defp03D}
p_0(\Omega) = 3\omega^\star/(\pi + \omega^\star) \quad \mathrm{with} \quad \omega^\star < \pi \; \mathrm{is \, \, the\, \,  larger\, \, polar\, \, angle\, \,   of} \; \Omega, \quad \mathrm{for}\; N = 3. 
\end{equation}
\end{theorem}

\begin{proof} i) The first two identities \eqref{carkerdim1} and \eqref{carkerdim2} above are a direct consequence of Theorem \ref{UniciteH1/2} and the following embeddings:
$$
W^{1,\, 2N/(N+1)}(\Omega)\hookrightarrow H^{1/2}(\Omega) 
$$
and
$$
 W^{\alpha, 2N/(N+2\alpha-1)}(\Omega) \hookrightarrow H^{1/2}(\Omega)\quad  \mathrm{for \, any } \quad 1/2 < \alpha < (N + 1)/2.
$$
We get \eqref{carkerdim3} by setting $p = 2N/(N+2\alpha-1)$, which gives $\alpha =  \frac{1}{2} + N(\frac{1}{p} - \frac{1}{2})$.\smallskip

\noindent ii) We will give the proof in the 2D and 3D cases, the proof being similar for the other dimensions. 

For $N = 2$,  let $ p <  4/3$  and $\alpha$ such that $1/2 < \alpha < (2/p) - 1$.  Let us now consider the following domain:
\begin{equation*}\label{domOmeg}
\Omega = \{(r, \theta);\; 0 < r < 1,\quad 0 < \theta < \frac {\pi} {\alpha}\}.
\end{equation*}
We can easily verify that the following function?
\begin{equation}\label{sing}
z(r, \theta) = (r^{-\alpha} - r^{\alpha})\mathrm{sin}(\alpha\theta)
\end{equation}
is harmonic in $\Omega$ with $z = 0$ on $\Gamma$ and $z\in W^{1, q}(\Omega)$ for any $q < \frac{2}{\alpha + 1}$. So we get the required result since  $p < \frac{2}{\alpha + 1} < \frac{4}{3}$. \medskip

For $N = 3$,  let $ p <  3/2$  and let us consider the domain $\Omega$ defined in \eqref{domOmeg3}. We choose  $\theta_0$ sufficiently close to $\pi$ to have  $\alpha_1 = \alpha_1(\theta_0)$, given by $\eqref{alphak}$, satisfying $1 < - \alpha_1 < (3/p) - 1$. The function $v_1$ defined in \eqref{noyau3D} is harmonic in $\Omega$ with $v_1 = 0$ on $\Gamma$ and $v_1\in W^{1, q}(\Omega)$ for any $q < \frac{3}{1 - \alpha_1}$.  So we get \eqref{carkerdim4} since  $p < \frac{3}{1 - \alpha_1} < \frac{3}{2}$. \medskip

To prove \eqref{carkerdim5}, we just consider the 2D case and we take again the same domain $\Omega$, with  $\frac{1}{2} < \alpha <  2(\frac{1}{r} - \frac{1}{p}) + \frac{1}{2}$.
\medskip

\noindent iii)  The singularities involved in the structure of the kernel $W^{1, p}_0(\Omega) \cap \mathscr{H}$ are, in the 2D case (resp. in the 3D case), of the type $r^{-\alpha}$, with $ \alpha > 1/2$, (resp. $ \alpha > 1 $). However the function $\vert x \vert^{-\alpha}$ belongs to $W^{1, q}(B)$ for any $ q$ strictly less than $2/(1 + \alpha)$,  (resp. $3/(1 + \alpha)$), where $B$ is the unit disk centered at origin.  More generally in any dimension, we have $ \alpha > (N-1)/2$.\medskip 

\noindent iv) We start by the 2D case. We suppose that $\Omega$ is a non convex polygonal domain with $\omega^\star$ the largest angle with the origin as its vertex $A$.  Introducing polar coordinates $(r, \theta)$ centered at $A$, the two sides of the angle correspond to $\theta = 0$ and $\theta = \omega^\star$.

We can construct a harmonic function $v$ in $\Omega$, equal to $0$ on $\Gamma$, equal  in a neighborhood $V$ of the origin to the function $z$ given by the relation \eqref{sing}, with $\alpha = \alpha^\star = \pi/\omega^\star$   and such that $v\in W^{1, p}(\Omega)$ for any $p < \frac{2}{\alpha^\star + 1} =: p_0(\Omega)$. We then deduce the second part of \eqref{carkerLp1b}. 

Let us now show how to find this harmonic function $v$. Let $\eta$  a regular cut-off function defined on $\mathbb{R}^+$, equal to 1 near zero, equal to zero in some interval $ [a, \infty[$, with some small $a > 0$. The function $w = \eta z$ has the same regularity $W^{1, p}(\Omega)$ than $v$, belongs also to $H^t(\Omega)$ for any $t < 1- \alpha^\star$ and
$$
\Delta w = 2 \nabla \eta \cdot \nabla z + z \Delta \eta \in H^{-s}(\Omega) \quad \mathrm{for \, \, any }\quad s >  \alpha^\star.
$$
We know that there exists a unique function $\chi\in H^{2-s}(\Omega) \cap H^1_0(\Omega)$ for any $ s >  \alpha^\star$ and satisfying $\Delta \chi = \Delta w$ in $\Omega$. The function $v =: \chi - w$ is the function we are looking for.

Let us now prove the first  part of \eqref{carkerLp1b}.  Using Sobolev embeddings, we have
$$
W^{1, p}_0(\Omega) \cap \mathscr{H} \subset \{\varphi \in H^{s}(\Omega) \cap \mathscr{H}, \, \varphi = 0\, \, \mathrm{on}\, \, \Gamma\}\quad \mathrm{with}\quad s = 2/p'.
$$
So if $ p \geq p_0(\Omega)$, we have $s > 2/(p_0(\Omega))' = 1 - \alpha^\star$. From \eqref{Ns0} the second space above is trivial and then we deduce our result.\medskip

For the 3D case, the proof is similar. We just need to replace $\omega^\star$ by $\theta_0$, $ \alpha^\star$ by $\alpha_1$ and $p_0(\Omega) =  \frac{2}{\alpha^\star + 1}$ by  $p_0(\Omega) = 3/(1 - \alpha_1)$. And to use the above function $v_1$ defined in \eqref{noyau3D} and belonging to $ W^{1, p}(\Omega)$ for any $1\leq p <  \frac{3}{1 - \alpha_1}$.\end{proof}

\begin{remark}\upshape \label{remJKp}  We will give in a forthcoming paper a new proof of the following result:  for any bounded Lipschitz domain $\Omega \subset \mathbb{R^N}$,  with $N \geq 2$, there exists $p_0(\Omega)$ such that $1 < p_0(\Omega) < 2N/(N+1)$ such that for any $p_0 \leq p \leq p_0'$ and $f \in W^{-1, p}(\Omega) $, Problem $(\mathscr{L}_D^0)$, ${\it i.e}\;  g  = 0$,  has a unique solution $u\in W^{1,\, p}(\Omega)$ satisfying the estimate
$$
\Vert u \Vert_{W^{1,\, p}(\Omega)}\leq C \Vert f \Vert_{W^{-1,\, p}(\Omega)}.
$$
\end{remark}

In \cite{J-K}, the authors state results concerning existence and regularity in $W^{1, p}$ (see Theorem 0.5) and in fractional Sobolev spaces (see Theorem 1.1 for $N \geq 3$ and Theorem 1.3 for $N = 2$). Let us first recall the results given in Theorem 0.5, where we replace $p_1$ with its conjugate: 

\begin{theorem} [{\bf Jerison-Kenig 95}, $W^{1, p}$ case]\label{JK95W1} Let $\Omega$ be a bounded Lipschitz domain in $\mathbb{R}^N$, with $N\geq 2$. There is an exponent $p_1$, with $p_1 < 4/3$, when $N= 2$ and $p_1 < 3$, when $N\geq 3$ such that if $p_1 < p < p_1'$, then the inhomogeneous Dirichlet problem $(\mathscr{L}_D^0)$ has a unique solution $u \in W^{1, p}(\Omega)$ for any $f\in W^{-1, p}(\Omega)$.
\end{theorem} 

This result was known when $\Omega$ is a polygon with  $p_1 = 2/(1+ (\pi/\omega^\star))$ where $\omega^\star$ is the larger angle of $\Omega$. In particular the following operator
$$
\mathrm{for \, \, any}\; 4/3 \leq p \leq 4, \quad \Delta : W^{1, p}_0(\Omega) \rightarrow W^{-1, p}_0(\Omega)
$$ 
is an isomorphism.\medskip

Let us then recall the results in 2D given in Theorem 1.3: 

\begin{theorem} [{\bf Jerison-Kenig 95}, Fractional case in $2D$]\label{JK95Frac} Let $\Omega$ be a bounded Lipschitz domain in $\mathbb{R}^2$. There exists $\varepsilon$, $0 < \varepsilon \leq 1/2$, depending only on the Lipschitz constant of $\Omega$ such that for every $f\in L^p_{s - 2}(\Omega)$, there is a unique 
$u \in L^p_s(\Omega)$ to the inhomogeneous Dirichlet problem provided one of the following holds:\\
\begin{equation*}
\begin{array}{rl}
& \mathrm{(a)} \qquad p_0 < p < p_0' \quad \mathrm{and}\quad 1/p < s < 1 + 1/p\\
& \mathrm{(b)} \qquad 1< p \leq p_0 \qquad \; \mathrm{and}\quad 2/p -1/2 - \varepsilon  < s < 1 + 1/p\\
& \mathrm{(c)} \qquad  p_0' \leq  p < \infty \quad \mathrm{and}\quad 1/p < s <   2/p + 1/2 + \varepsilon,
\end{array}
\end{equation*}
where $1/p_0 = \varepsilon + 1/2$ and $1/p_0' = -\varepsilon + 1/2$.
\end{theorem}

For any bounded Lipschitz domain $\Omega$,  the exponents $p_1$ and $p_0$ given by Theorem \ref{JK95W1} and Theorem \ref{JK95Frac} are clearly identical.  Moreover, when $\Omega$ is a polygon, we have 
$$
p_0 = p_1 = p_0(\Omega) =: 2/(1+ \alpha^\star),\quad \mathrm{with}\quad  \alpha^\star = \pi / \omega^\star \quad \mathrm{and}\quad \varepsilon =  \alpha^\star/2 \in \, ]1/4, 1/2[,
$$
where $\omega^\star$ is the larger angle of $\Omega$.   It is important to note that for any fixed domain $\Omega$ corresponds a unique $p_0(\Omega)$, which is strictly less than $4/3$ in accordance with Theorem \ref{JK95W1} and a unique $\varepsilon(\Omega) = : 1/p_0(\Omega) - 1/2$. 

We shall now see that Theorem \ref{JK95Frac}, as stated, is not suitable. Indeed, conditions (b) and (c) are inappropriate for the reasons that follow. For clarity, we will only consider the case $s = 1$ corresponding to $W^{1,\, p}(\Omega)$ solutions. However, our arguments below remain valid for other values of $s$.  More precisely, let us choose a polygonal domain $\Omega$ whose largest angle $\omega^\star$ is close to $2 \pi$. So the corresponding $p_0(\Omega)$, given above, is such that  $p_0(\Omega) < 4/3$ and very close to 4/3. Therefore $\varepsilon > 1/4$, with $\varepsilon$ very close to 1/4.  We have $\alpha^\star  = 1/2 + a$, with $a > 0$ sufficiently close to $0$.  Now, it is easy to verify that for any real $ p$ belonging to the interval  $]8/7 , (4/3)(1+a)^{-1}[$,  the pair $(p, 1)$ satisfies the condition (b). But in this case, since the kernel $W^{1, p}_0(\Omega) \cap \mathscr{H}$ is non-trivial, the solution cannot be unique, contrary to what is stated in condition (b).  Another way to see that condition (c) can be satisfied is to take $p$ in the interval $]2p_0/(p_0 + 1), p_0[.$ 

Similarly, we can see that for any real $ p$ belonging to the interval $I = \, ]4(1-3a)^{-1}, 8[$,  the pair $(p, 1)$ satisfies the condition (c). Another way to see that condition (b) can be satisfied is to take $p$ in the interval $]p'_0, 2 p'_0[.$ However the following operators
$$
\Delta: \; W^{1, p}_0(\Omega) \rightarrow W^{-1, p}_0(\Omega) \quad \mathrm{and} \quad  W^{1, p'}_0(\Omega) \rightarrow W^{-1, p'}_0(\Omega)
$$
are adjoint from one another. And the range of the space $W^{1,\, p}_0(\Omega)$ by the Laplacian is a subspace of the space
\begin{equation*}
\{f \in W^{-1,\, p}(\Omega); \; \langle f, \varphi\rangle = 0, \; \forall \varphi \in W^{1,\, p'}_0(\Omega)\cap \mathscr{H}\}.
\end{equation*}
So, if $f $ belongs to $W^{-1,\, p}(\Omega)$, with $p \in \, ]p'_0, 2 p'_0[$ and without satisfying the condition
$$
\langle f, \varphi\rangle = 0, \; \forall \varphi \in W^{1,\, p'}_0(\Omega)\cap \mathscr{H},
$$
the $H^1_0(\Omega)$ solution  to the inhomogeneous Dirichlet problem does not belong to $W^{1, p}(\Omega)$. Contrary to what is stated in condition (c) of Theorem \ref{JK95Frac}. Also the condition $(c)$ is also not available. \medskip 

Let us finally recall the results  given in Theorem 1.1 for $N \geq 3$: 

\begin{theorem} [{\bf Jerison-Kenig 95}, Fractional case when $N \geq 3$]\label{JK95FracN3} Let $\Omega$ be a bounded Lipschitz domain in $\mathbb{R}^N$, with $N \geq 3$. There exists $\varepsilon$, $0 < \varepsilon \leq 1$, depending only on the Lipschitz constant of $\Omega$ such that for every $f\in L^p_{s - 2}(\Omega)$, there is a unique 
$u \in L^p_s(\Omega)$ to the inhomogeneous Dirichlet problem provided one of the following holds:\\
\begin{equation*}
\begin{array}{rl}
& \mathrm{(a)} \qquad p_0 < p < p_0' \quad \mathrm{and}\quad 1/p < s < 1 + 1/p\\
& \mathrm{(b)} \qquad 1< p \leq p_0 \qquad \; \mathrm{and}\quad 3/p -1 - \varepsilon  < s < 1 + 1/p\\
& \mathrm{(c)} \qquad  p_0' \leq  p < \infty \quad \mathrm{and}\quad 1/p < s <   3/p + \varepsilon,
\end{array}
\end{equation*}
where $1/p_0 = \varepsilon/2 + 1/2$ and $1/p_0' =  1/2 -\varepsilon/2$.
\end{theorem}

As for the 2D case above,  the exponents $p_1$ and $p_0$ given by Theorem \ref{JK95W1} and Theorem \ref{JK95FracN3} are identical. When $\Omega$ is a polyhedron, we have
$$
p_0 = p_1 = p_0(\Omega) =: 3 /(1- \alpha_1^\star),\; \;  \mathrm{with}\; \;   \alpha_1^\star  = -(1 + \sqrt{1 + 4 \lambda_1^\star})/2 \quad \mathrm{and}\quad \varepsilon =  -(1+ 2\alpha_1^\star)/3\in \, ]1/3, 2/3[,
$$
where $\lambda_1^\star$ is the smallest eigenvalue of the Beltrami Laplacian $- \Delta_\Gamma$.  Note that $p_0(\Omega) < 3/2$ and recall that $\lambda_1^\star > 3/4$, when the polyedron $\Omega$ is convex (see Lemma 2.5 in \cite{Gri2}). The bounds $1/3$ and $2/3$ of the above interval  are determined from the limit values of $\lambda_1^\star(\theta_0)$ when $\theta_0$ tends to $\pi$ and  $\pi/2$ respectively.

Choosing a polyedral domain $\Omega$ whose largest polar angle $\omega^\star$ is close to $\pi$, we show as in the 2D case that Theorem \ref{JK95FracN3}, as stated, is not suitable. Indeed the condition (b) is satisfied with $p \in \,  ]3 p_0/(p_0 + 2), p_0[$, while condition (c) is satisfied for $p \in\, ]p_0', 3p_0'/2[$.
\medskip

For other remarks, see the Appendix below.

 \section{Area integral estimate. Counter-example}\label{Counter Example}

\subsection{Area integral estimate. Reminders and comments}
 In \cite{Necas}, Ne\v{c}as proved the following property (see Theorem 2.2 Section 6): if $\varrho^{\alpha/p} u\in L^p(\Omega)$ and $\varrho^{\alpha/p}\nabla u \in L^p(\Omega)$, for some $0 \leq \alpha < p - 1$, then $u_{\vert\Gamma} \in L^p(\Gamma)$ and 
\begin{equation*}\label{inegNec}
\int_\Gamma \vert u \vert^p\, d\sigma \leq C(\Omega) \big(\int_\Omega \varrho^\alpha\vert  u \vert^p  \, dx + \int_\Omega \varrho^\alpha\vert \nabla u \vert^p \, dx\big).
\end{equation*}
However, if $\alpha = p - 1$, he showed, using a counterexample with $\Omega= \,  ]0, 1/2[\,  \times\,  ]0, 1/2[ $, that the above inequality does not hold in general. In particular for $\alpha = 1$ and $p = 2$, if $ \sqrt \varrho \, u$ and $\sqrt \varrho \, \nabla u \in L^2(\Omega)$ (which implies that $u\in H^{1/2}(\Omega)$),  the function $u$ may have no trace in $L^2(\Gamma)$. \medskip

What about if in addition the function $u$ is harmonic? In \cite{Dahl}, see Corollary, Section 6,  the author proved the following property: let $\textit{\textbf x}_0$ a fixed point in $\Omega$, then there exists a constant $C > 0$ such that if $u$ is a harmonic function in $\Omega$ and vanishes at $\textit{\textbf x}_0$, then 

\begin{equation}\label{inegaltraceL2Gamma1}
C^{-1} \int_\Gamma \vert u \vert^2 d\sigma \leq \int_\Omega \varrho \vert\nabla u\vert^2 dx \leq C \int_\Gamma \vert u \vert^2 d\sigma.
\end{equation}

Seventeen years later, in \cite{Dahlberg}, the authors have demonstrated, in the same context but using a different proof, a similar result  to \eqref{inegaltraceL2Gamma1}:  there exists $C > 0$ such that
\begin{equation}\label{ineg1}
\int_\Gamma \vert u \vert^2 d\sigma \leq C \int_\Gamma \vert S(u)\vert^2 d\sigma = C \int_\Omega \delta \vert\nabla u\vert^2 dx,
\end{equation}
without any specification of spaces of such function. Here, $\Omega$ is a bounded Lipschitz domain of $\mathbb{R}^N$, $N \geq 2$,  and $S(u)$ is the area integral of $u$ and $\delta$ is an adaptative distance to the boundary, equivalent  to the distance $\varrho$ to the boundary $\Gamma$. The proof of \eqref{ineg1} is given on pages 1428 and 1429 in  \cite{Dahlberg} and contains unjustified formal calculations. The authors repeatedly use Green's formulas, the validity of which is not guaranteed. In one of them, they claim that the integral
$$
\int_\Omega D_n(\frac{\delta}{D_n \delta}uD_nu)dx 
$$
is zero since $\delta = 0$ on $\Gamma$. At the end of their proof, the authors should have verified their inequalities using a density result of the following form:
$$
\mathscr{D}(\overline{\Omega}) \cap  \mathscr{H} \quad \mathrm{is\; dense\; in }\quad \{v\in L^2(\Omega); \;  \int_\Omega \varrho \vert\nabla v\vert^2 \, dx< \infty\} \cap  \mathscr{H} ,
$$
where $ \mathscr{H} $ is the space of harmonic functions in $\Omega$.
\smallskip

Recall that if $u\in L^2(\Omega)$, then we have the following implication: 
$$
\sqrt \varrho\,  \nabla u \in L^2(\Omega) \Longrightarrow u \in H^{1/2}(\Omega) \quad \mathrm{and}\quad \sqrt \varrho\,  \nabla^2 u \in L^2(\Omega) \Longrightarrow u \in H^{3/2}(\Omega)
$$
and the reverse implications hold if moreover $u$ is a harmonic function, see Theorem 3.2 and Theorem 3.8 in \cite{AM} or Theorem 4.2 in \cite {J-K}. Here $\nabla^2 u$ denotes the Hessian matrix of $u$. In addition,  we have  the following equivalence norms for harmonic functions:
\begin{equation*}
\begin{array}{rl}
\Vert u \Vert_{H^{1/2}(\Omega)} \approx & \Vert u \Vert_{L^2(\Omega)} + \Vert \sqrt  \varrho\,  \nabla u  \Vert_{L^{2}(\Omega)} \quad \mathrm{for}\; u \in H^{1/2}(\Omega)\cap \mathscr{H}, \\
\Vert u \Vert_{H^{3/2}(\Omega)} \approx & \Vert u \Vert_{L^2(\Omega)} + \Vert \sqrt  \varrho\,  \nabla^2 u  \Vert_{L^{2}(\Omega)} \quad \mathrm{for}\; u \in H^{3/2}(\Omega)\cap \mathscr{H}. 
\end{array}
\end{equation*}

Inequalities \eqref{inegaltraceL2Gamma1} would then result in the equivalence of the following norms:
\begin{equation*}
\Vert u \Vert_{H^{1/2}(\Omega)} \approx  \Vert u \Vert_{L^2(\Gamma)} \quad \mathrm{for}\; u \in H^{1/2}(\Omega)\cap \mathscr{H} 
\end{equation*}
and would imply that any harmonic function $H^{1/2}(\Omega)$ has a trace  in $L^2(\Gamma)$.  Consequently, the gradient of any harmonic function $u$ in $H^{3/2}(\Omega)$ would have a trace in $\textit{\textbf L}^2(\Gamma)$, {\it i.e} $u\in H^1(\Gamma)$ and $\partial_\textit{\textbf n} u \in L^2(\Gamma)$.\medskip

In the same spirit, inequalities \eqref{inegaltraceL2Gamma1} would imply the following property: let $u$ be a harmonic function in $\Omega$ satisfying $u(\textit{\textbf x}_0) = 0$ and $\nabla u(\textit{\textbf x}_0) = {\bf 0}$ at some point $\textit{\textbf x}_0\in \Omega$, then
\begin{equation}\label{inegaltraceL2Gammapent}
\Vert u \Vert_{H^1(\Gamma)} \leq C(\Omega)\Big(\int_\Omega \varrho\vert \nabla^2 u \vert^2 \, dx\Big)^{1/2}.
\end{equation}
And as above,  we would have the equivalence of the following norms:
\begin{equation*}
\Vert u \Vert_{H^{3/2}(\Omega)} \approx  \Vert u \Vert_{H^1(\Gamma)} \quad \mathrm{for}\; u \in H^{3/2}(\Omega)\cap \mathscr{H}.
\end{equation*}

 \subsection{Counter-example}\label{Counter Example-S}
 The purpose of this subsection is to show that the inequalitiy \eqref{inegaltraceL2Gammapent} cannot be valid for Lipschitz domains and therefore the same applies to the inequalities \eqref{inegaltraceL2Gamma1} and \eqref{ineg1}. However, we will propose in the next section an alternative for the functions $H^{1/2}(\Omega)$, resp. $H^{3/2}(\Omega)$, to have a trace in $L^2(\Gamma)$, resp. in $H^1(\Gamma)$.\medskip

\begin{e-proposition}   [{\bf Counter example for Inequality \eqref{inegaltraceL2Gammapent}}]  \label{ConterexampleH1demitrace1b} For any  $\varepsilon > 0$, there exist a Lipschitz domain $\Omega_\varepsilon \subset \mathbb{R}^2$ and a harmonic function $w_\varepsilon \in H^{3/2}(\Omega_\varepsilon)$ (with $\sqrt \varrho_\varepsilon\, \nabla^2 w_\varepsilon  \in L^2(\Omega_\varepsilon))$ such that the following family 
$$
(\Vert \varrho \nabla^2 w_\varepsilon \Vert_{L^{2}(\Omega_\varepsilon)} + \Vert w_\varepsilon \Vert_{H^{3/2}(\Omega_\varepsilon)})_\varepsilon,
$$
is bounded with respect $\varepsilon$ and
\begin{equation*}
\Vert w_\varepsilon \Vert_{H^1(\Gamma_\varepsilon)} \rightarrow \infty \quad as \; \varepsilon\rightarrow  0.
\end{equation*}
\end{e-proposition} 

\begin{proof} Recall first that if $u\in L^2(\Omega) $ is  harmonic, then 
$$
u\in H^{1/2}(\Omega) \Longleftrightarrow \sqrt \varrho \, \nabla u \in L^2(\Omega)\; \mathrm{and}\; u\in H^{3/2}(\Omega) \Longleftrightarrow \sqrt \varrho \, \nabla^2 u \in L^2(\Omega),
$$
(see Theorem \ref{c06-l1} and Theorem \ref{inegharm} above or Theorem 4.2 in \cite{J-K}).
\medskip

\noindent{\bf Step 1.}  We suppose now that  $\Omega= \,  ]0, 1/2[\,  \times\,  ]0, 1/2[ $ and for any $\varepsilon > 0$ and close to 0, we define, as in the counter example to G. David given in the paper of D. Jerison and C. Kenig (\cite{J-K}), the following open set:
$$
\Omega_\varepsilon = \{(x, y) \in \mathbb{R}^2; \;  0 < x < 1/2\;\; \mathrm{and}\; \;  \varepsilon \lambda (x/\varepsilon)< y < 1/2\},
$$
where 
$$
\lambda (x) = x\; \mathrm{ for} \;  0\leq x \leq 1,  \quad  \lambda (x) = 2- x \; \mathrm{ for } \;  1\leq x \leq 2 \quad \mathrm{ and } \; \; \lambda(x + 2) =\lambda (x).
$$
We set
$$
\Gamma_\varepsilon = \{(x, y) \in \mathbb{R}^2; \;  y = \varepsilon \lambda (x/\varepsilon), \; \mathrm{with}\;  0 < x < 1/2\},
$$
which is just a part of the boundary of $\Omega$.\medskip

\noindent{\bf Step 2.} Let us consider the following function (see also Example 2.1 given in Chapter 6 of the book of Ne\v{c}as) which depends only on the variable $y$:
$$
v(x, y) = \int_0^{y} ds\int_{1/2}^s \frac{dt}{t\, \ell n\, t}.
$$
It is easy to verify that $\sqrt{\varrho}\, \nabla^2 v \in \textbf{\textit L}^2(\Omega)$ where $\varrho$  is the distance to the boundary of $\Omega$. So from Theorem \ref{c06-l1}, we have also $v\in H^{3/2}(\Omega)$ and $\Delta v\in [H^{1/2}_{00}(\Omega)]'$. Moreover the $L^2$ norm of the tangential derivative $\partial_\tau v$ of $v$ on $\Gamma_\varepsilon $ tends to infinity as $\varepsilon$ tends to zero. Indeed, we can see that 
$$
\partial_\tau v (x, x) = \frac{\sqrt 2}{2} \int_x^{1/2} \frac{dt}{t\, \vert  \ell n\, t\vert} \quad \mathrm{ for } \; \; 0 <  x < \varepsilon,
$$
and 
$$
\partial_\tau v (x, 2\varepsilon - x) = - \frac{\sqrt 2}{2}  \int_{2\varepsilon - x}^{1/2} \frac{dt}{t\, \vert  \ell n\, t\vert} \quad \mathrm{ for } \; \; \varepsilon <  x < 2\varepsilon.
$$
Observe that 
$$
 \int_x^{1/2} \frac{dt}{t\, \vert  \ell n\, t\vert} = \ell n(- \ell n\, x) -  \ell n( \ell n\, 2).
 $$
So we have
\begin{equation*}
\begin{array}{rl}
\displaystyle\int_{\varepsilon}^{2\varepsilon}\left(\int_{2\varepsilon - x}^{1/2}\frac{dt}{t\, \vert  \ell n\, t\vert}\right)^2 dx = & \displaystyle\int_{\varepsilon}^{2\varepsilon} [\ell n(- \ell n\, (2\varepsilon - x)) -  \ell n( \ell n\, 2)]^2 dx\\
= &\displaystyle \int_{0}^{\varepsilon} [\ell n(- \ell n\, x) -  \ell n( \ell n\, 2)]^2 dx.
\end{array}
\end{equation*}
In addition as the function $\partial_\tau v$ is periodic with the period $ 2\varepsilon$ we get
$$
\Vert\partial_\tau v\Vert_{L^2(\Gamma_\varepsilon)}^2 =  (1/4\varepsilon)\int_0^{\varepsilon}\left(\int_x^{1/2}\frac{dt}{t\, \vert  \ell n\, t\vert}\right)^2 dx, 
$$
where we have chosen $\varepsilon = 1/4k$, with $k\in \mathbb{N}^\star$. 
\medskip

\noindent{\bf Step 3.}  We will give an estimate of the integral:
$$
I_\varepsilon = \int_0^{\varepsilon}\left(\int_x^{1/2}\frac{dt}{t\, \vert  \ell n\, t\vert}\right)^2 dx. 
$$
Setting $s =  \ell n \, t$, we get 
$$
\int_x^{1/2}\frac{dt}{t\, \vert  \ell n\, t\vert} = -  \int_{ \ell n\, x}^{-\ell n\, 2}\frac{ds}{s} = \ell n(- \ell n\, x) -  \ell n( \ell n\, 2).
$$
So $I_\varepsilon  = I_{1\varepsilon} - I_{2\varepsilon} + I_{3\varepsilon}$, where
$$
 I_{1\varepsilon} =  \int_0^{\varepsilon}  [\ell n(- \ell n\, x)]^2dx \quad I_{2\varepsilon} = 2 \ell n( \ell n\, 2)\int_0^{\varepsilon} \ell n(- \ell n\, x)dx
 $$
and
$$
 I_{3\varepsilon} =  \int_0^{\varepsilon}  [\ell n( \ell n\, 2)]^2dx.
$$
Since the function $\ell n(- \ell n\, x)$ is nondecreasing in the interval $]0, \varepsilon]$ and  $I_\varepsilon  \geq I_{1\varepsilon} - I_{2\varepsilon} $, we deduce easily the estimate for $\varepsilon$ close to 0:
\begin{equation*}
 I_{\varepsilon} \geq  \frac{\varepsilon}{2} \ell n(-\ell n\, \varepsilon)]^2
\end{equation*}
So, we get
$$
 \frac{1}{2\sqrt 2} [\ell n(-\ell n\, \varepsilon)]  \leq \Vert\partial_\tau v\Vert_{L^2(\Gamma_\varepsilon)} \longrightarrow \infty \quad \mathrm{as} \quad \varepsilon \rightarrow 0.
$$
\noindent{\bf Step 4.} Now, using Theorem \ref{theoHsregul}  and Proposition \ref{corGrisvardPol} in the polygon $\Omega_\varepsilon$, there exits a unique solution $ H^{3/2}_0(\Omega_\varepsilon)$ satisfying $\Delta u_\varepsilon = \Delta v$ in $\Omega_\varepsilon$ with the estimate 
\begin{equation}\label{ineguvareps}
\Vert   u_\varepsilon  \Vert_{H^{3/2}(\Omega_\varepsilon)}\leq C(\Omega_\varepsilon) \Vert \Delta v \Vert_{M_{1/2}(\Omega_\varepsilon)},
\end{equation}
where $ C(\Omega_\varepsilon)$ depends only on the Poincar\'e constant and on the Lipschitz constant of $\Omega_\varepsilon$, which are both  bounded with respect $\varepsilon$: so we have  $  C(\Omega_\varepsilon) \leq  C(\Omega)$. Since the norm $ \Vert \Delta v \Vert_{M_{1/2}(\Omega_\varepsilon)}$ is equivalent to the norm $ \Vert \Delta v \Vert_{[H^{1/2}_{00}(\Omega_\varepsilon)]'}$, we have the following:
\begin{equation}\label{relHM}
 \exists C_ \varepsilon \geq  0  \quad \mathrm{such \, \, that}\quad \Vert \Delta v \Vert_{[H^{1/2}_{00}(\Omega_\varepsilon)]'} \leq \Vert \Delta v \Vert_{M_{1/2}(\Omega_\varepsilon)} = (1 + C_ \varepsilon)\Vert \Delta v \Vert_{[H^{1/2}_{00}(\Omega_\varepsilon)]'}. 
\end{equation}
To simplify the notations, we will write  $M(\Omega)$ instead of $M_{1/2}(\Omega)$, resp. $M(\Omega_\varepsilon)$ instead of $M_{1/2}(\Omega_\varepsilon)$. 
Since $\Omega$ is convex, the kernel space $\mathscr{N}_0(\Omega)$ is reduced to zero. So
\begin{equation}\label{egalMH}
\! \! \! \! M(\Omega) =  [L^2(\Omega), H^{-1}(\Omega)]_{1/2} = (H^{1/2}_{00}(\Omega))' \quad \mathrm{with}\quad \Vert f \Vert_{M(\Omega)} = \Vert f \Vert_{ (H^{1/2}_{00}(\Omega))'}\quad \mathrm{for}\; f \in (H^{1/2}_{00}(\Omega))'
\end{equation}

We claim that 
\begin{equation*}\label{Ceps}
\lim_{\varepsilon \rightarrow 0}C_\varepsilon = 0  .
\end{equation*} 

Let us begin by observing that for any $f \in (H^{1/2}_{00}(\Omega))'$,
\begin{equation}\label{inegDelta}
\Vert f \Vert_{[H^{1/2}_{00}(\Omega_\varepsilon)]'} = \displaystyle\sup_{\varphi\in H^{\,1/2}_{00}(\Omega_\varepsilon), \, \varphi \not= 0}\dfrac{ \vert \langle f,\,  \varphi\rangle_{\Omega_\varepsilon}\vert}{ \Vert \varphi \Vert_{H^{1/2}_{00}(\Omega_\varepsilon)}} = \displaystyle\sup_{\varphi\in H^{\,1/2}_{00}(\Omega_\varepsilon), \, \varphi \not= 0}\dfrac{ \vert \langle f,\, \widetilde{ \varphi} \rangle_{\Omega}\vert}{ \Vert \widetilde{ \varphi} \Vert_{H^{1/2}_{00}(\Omega)}},
\end{equation}
where the extension by zero of $\varphi$ in $\Omega$ satisfies
\begin{equation}\label{egnor}
 \widetilde{ \varphi} \in H^{1/2}_{00}(\Omega) \quad \mathrm{and}\quad  \Vert \widetilde{ \varphi} \Vert_{H^{1/2}_{00}(\Omega)} =  \Vert  \varphi\Vert_{H^{1/2}_{00}(\Omega_\varepsilon)}.
 \end{equation}
As the sequence $(\Omega_\varepsilon)_{\varepsilon}$ is increasing  in the sense of inclusion, when  $\varepsilon \searrow 0$, we deduce from \eqref{inegDelta} and \eqref{egnor} that the sequence $ (\Vert f \Vert_{[H^{1/2}_{00}(\Omega_\varepsilon)]'})_\varepsilon$ is increasing when  $\varepsilon \searrow 0$
and
$$
\lim_{\varepsilon \rightarrow 0}\Vert f \Vert_{[H^{1/2}_{00}(\Omega_\varepsilon)]'} = \Vert f \Vert_{(H^{1/2}_{00}(\Omega))'}.
$$
Concerning the corresponding norms for the domain $\Omega_\varepsilon$,  we first recall that
$$
M(\Omega_\varepsilon)  =: [(\mathscr{N}_0(\Omega_\varepsilon))^\bot,  H^{-1}(\Omega_\varepsilon)]_{1/2} =  (H^{1/2}_{00}(\Omega_\varepsilon))',
$$
where $(\mathscr{N}_0(\Omega_\varepsilon))^\bot$ is equipped with the $L^2(\Omega_\varepsilon)$ norm.  Similarly,  the sequence $(\Vert \Delta v \Vert_{M(\Omega_\varepsilon)})_\varepsilon$ is increasing when  $\varepsilon \searrow 0$.  From \eqref{egalMH}, we know that
\begin{equation*}\label{egaMHb}
\Vert \Delta v \Vert_{M(\Omega)} = \Vert \Delta v \Vert_{ (H^{1/2}_{00}(\Omega))'}.
\end{equation*}
So the only possibility for the above equality to hold is for the constant $C_\varepsilon$ in \eqref{relHM} to approach 0 as $\varepsilon$ tends to $ 0$. Thus 
$$
  \lim_{\varepsilon \rightarrow 0}\Vert \Delta v \Vert_{M(\Omega_\varepsilon)}= \Vert \Delta v \Vert_{(H^{1/2}_{00}(\Omega))'} = \Vert \Delta v \Vert_{M(\Omega)}  .
$$

\noindent{\bf Step 5.} Using the estimate \eqref{ineguvareps}, we finally get the following:
$$
\Vert   u_\varepsilon  \Vert_{H^{3/2}(\Omega_\varepsilon)}\leq C(\Omega) \Vert \Delta v \Vert_{(H^{1/2}_{00}(\Omega))'}\leq C(\Omega) \Vert  v \Vert_{H^{3/2}(\Omega)},
$$
where $ C(\Omega)$ does not depend on $\varepsilon$.  

So the harmonic function in $\Omega_\varepsilon$ defined by $w_\varepsilon = v - u_\varepsilon $ belongs to $H^{3/2}(\Omega_\varepsilon)$ and satisfies
\begin{equation*}\label{estimCE}
 \Vert   w_\varepsilon  \Vert_{H^{3/2}(\Omega_\varepsilon)}\leq C(\Omega) \Vert  v \Vert_{H^{3/2}(\Omega)}.
\end{equation*}

Then, as $u_\varepsilon  = 0$ on $\partial\Omega_\varepsilon$, we have $\partial_\tau w_\varepsilon = \partial_\tau v$ on $\Gamma_\varepsilon$ and $w_\varepsilon$ would verify the following equality: $
\Vert\partial_\tau w_\varepsilon\Vert_{L^2(\Gamma_\varepsilon)} = \Vert\partial_\tau v\Vert_{L^2(\Gamma_\varepsilon)},$ which tends to infinity as $\varepsilon$ tends to zero.\end{proof} \medskip

Thus, there can be no estimate for harmonic function $u$, on the norm $\Vert u \Vert_{H^1(\Gamma)}$ by the norm  $\Vert u \Vert_{ H^{3/2}(\Omega)}$  dependent only on the Lipschitz constant of the considered  domain $\Omega$. More precisely, the following inequality
\begin{equation*}
\Vert u \Vert_{H^1(\Gamma)} \leq  C(\Omega) \Vert u \Vert_{ H^{3/2}(\Omega)}
\end{equation*}
cannot be satisfied in general, unlike when the domain is of class $\mathscr{C}^{1, 1}$, as we will see in the next section. As explained in the introduction, this implies that the following  inequality  
\begin{equation*}
\Vert u \Vert_{L^2(\Gamma)} \leq  C(\Omega) \Vert u \Vert_{H^{1/2}(\Omega)}
\end{equation*}
cannot be satisfied in general for harmonic functions $H^{1/2}(\Omega)$ when the domain is only Lipschitz and  the same applies to inequalities  \eqref{inegaltraceL2Gamma1} and \eqref{ineg1}.

\section{Traces in the limit cases $H^{1/2}(\Omega)$ and $H^{3/2}(\Omega)$ (Continuation)} \label{traces2}

Recall that when $\Omega$ is of class $\mathscr{C}^{k-1, 1}$, with positive integer $k$, then the  the linear mapping 
\begin{equation*}\label{trlim}
\gamma:  u \longrightarrow u_{\vert \Gamma}\\
\end{equation*}
is continuous for any  $1/2< s < 1/2 + k$ (see \cite{Wend}). In particular if $\Omega$ is of class $\mathscr{C}^{1, 1}$ and $u\in H^{3/2}(\Omega)$, then its trace belongs to $H^1(\Gamma)$ (see also \cite{McL}, Theorem 3.7). \medskip

Now, concerning the normal derivative $\partial_{\textit{\textbf n}}u$ of $u$, we recall that if $\Omega$ is a bounded Lipschitz domain and $u\in H^{3/2}(\Omega)$, the normal derivative is in general not defined (even if $\Omega$ is more regular). However, when $u\in H^{s}(\Omega)$, with $s >  3/2$, then $\partial_{\textit{\textbf n}}u \in L^2(\Gamma)$ only. In the case where $\Omega$ is $\mathscr{C}^{1, 1}$  and $u \in H^s(\Omega)$, with $3/2 < s < 5/2$, then $\partial_{\textit{\textbf n}}u\in H^{s-3/2}(\Gamma)$.\medskip

For the extremal value $s = 1/2$, we have seen in Section \ref{traces} that the range of $E(\nabla;\, \Omega)$ (see \eqref{defEnablaOmega}  for the definition of this space) by $\gamma$  is included in $L^2(\Gamma)$.   

\begin{remark}\label{Surj}\upshape What about the characterization of the range of $E(\nabla;\, \Omega)$ by the linear mapping $\gamma: u \mapsto u_{\vert \Gamma}$? Is this range equal or strictly included in $L^2(\Gamma)$? We will answer this question a little later on, see Theorem \ref{protrace} below.
\end{remark} 

We will now focus on questions related to the normal derivative of $H^1(\Omega)$ functions with sufficiently regular Laplacian.  Recall that if $v \in H^1(\Omega)$ with $\Delta v \in L^2(\Omega)$ (or even $\Delta v \in [H^{1/2}(\Omega)]'$), then the normal derivative  $\partial_\textit{\textbf n} v = \nabla v\cdot\textit{\textbf n}  \in H^{-1/2}(\Gamma)$ and we have the following Green formula:
\begin{equation*}
	\forall\varphi\in H^{1}(\Omega),\quad  \int_{\Omega}(\nabla v\cdot \nabla \varphi + \varphi \Delta v) \, dx = \langle \partial_\textit{\textbf n} v, \varphi \rangle_{H^{-1/2}(\Gamma) \times H^{1/2}(\Gamma)}.
\end{equation*}
To prove the existence of this normal derivative and this Green formula, we can use the surjectivity of  the trace operator from $H^1(\Omega)$ into $H^{1/2}(\Gamma)$,  the same Green formula for smooth functions $v$ in $\mathscr{D}(\overline{\Omega})$ and finally the density of this last space into $E(\Delta, \Omega) = \{v \in H^1(\Omega); \;  \Delta v \in L^2(\Omega)\}$. Or more simply, as in \cite{Fabes3},  by defining the normal derivative as the  continuous linear form on $H^{1/2}(\Gamma)$ 
\begin{equation*}
T: g \mapsto \int_{\Omega}(\nabla v\cdot \nabla \varphi + \varphi \Delta v)\, dx, 
\end{equation*}
 where $\varphi \in H^1(\Omega)$ is any extension of $g$ satisfying $\Vert \varphi \Vert_{H^1(\Omega)} \leq C(\Omega) \Vert g \Vert_{H^{1/2}(\Gamma)}$.

\begin{remark} \upshape \label{remDN} What happens now for the regularity of $ \partial_\textit{\textbf n} v$ if in addition  $v\in H^{3/2}(\Omega)$ and $\Delta v = 0$ in $\Omega$?  Clearly, we know that $ \partial_\textit{\textbf n} v \in H^{-s}(\Gamma)$ for any $0 < s \leq 1/2$. So the precise question is now: can we take $s = 0$ and then to get $ \partial_\textit{\textbf n} v \in L^2(\Gamma)$?  For that  we must show that the above linear form $T$ can be extended (and in this case the extension will be unique since  $H^{1/2}(\Gamma)$ is dense in $L^2(\Gamma)$) to a continuous linear form on $L^2(\Gamma)$ still denoted by $T$:
\begin{equation*}
T:  g \mapsto \langle \nabla v,  \nabla u_g\rangle_{H^{1/2}(\Omega)\times [H^{1/2}(\Omega)]'} , 
\end{equation*}
where $u_g \in H^{1/2}(\Omega)$ is the unique harmonique function satisfying $u_g = g$ on $\Gamma$. Recalling that the linear operator $S : g \mapsto u_g$ is continuous from $L^2(\Gamma)$ into $H^{1/2}(\Omega)$, the answer to the previous question will be positive if and only if the following linear mapping
\begin{equation}\label{nablaS}
\nabla_\circ S : L^2(\Gamma) \longrightarrow [H^{1/2}(\Omega)]' 
\end{equation}
is  continuous. We will answer this question in the next theorem.\medskip
\end{remark}

\begin{theorem} [{\bf Traces of harmonic functions  $H^{1/2}$}]\label{protrace}
Assume that $\Omega$ is of class $\mathscr{C}^{1, 1}$. We have the following properties:\\
i) Let $u \in H^{1/2}(\Omega)$. Then we have the following implication:
\begin{equation}\label{impharC11}
u \quad \mathrm{harmonic}\quad   \Longrightarrow \quad u_{\vert\Gamma} \in L^2(\Gamma) \quad \mathrm{and} \quad \nabla u \in  [\textit{\textbf  H}^{\, 1/2}(\Omega)]'
\end{equation} 
In particular, the trace operator $\gamma :  H^{1/2}(\Omega)\cap \mathscr{H} \rightarrow L^2(\Gamma)$ is an isomorphism, where $\mathscr{H}$ is the space of harmonic functions in $\Omega$. Moreover, the range $\gamma(E(\nabla, \Omega))$ is equal to $ L^2(\Gamma)$. \\
ii) We have the algebraical identity $H^{1/2}(\Omega)\cap \mathscr{H}  = E(\nabla, \Omega)\cap \mathscr{H}$ and topological:
$$
\Vert \nabla v \Vert_{[\textit{\textbf H}^{\,1/2}(\Omega)]'} \approx \Vert \nabla v \Vert_{[\textit{\textbf H}^{\, 1/2}_{00}(\Omega)]'} \quad \mathrm{for}\; v \in H^{1/2}(\Omega)\cap \mathscr{H}.
$$
iii) The operator \eqref{nablaS} is  continuous.
\end{theorem}

\begin{proof} {\bf i) Step 1.}  We know that $S \in \mathscr{L}(H^{1/2}(\Gamma); H^1(\Omega)\cap \mathscr{H})$ even if $\Omega$ is only Lipschitz. Since $\Omega$ is of class $\mathscr{C}^{1, 1}$ we have also $S \in \mathscr{L}(H^{-1/2}(\Gamma); L^2(\Omega)\cap \mathscr{H})$ (see Theorem 7 in  \cite{ARB}). However, this last continuity property does not hold  when $\Omega$ is only Lipschitz since the normal trace operator from $ H^{2}(\Omega)\cap H^{1}_0(\Omega)$ into $ H^{1/2}(\Gamma)$ is not well defined and of course non surjective. Now, as
$$
[L^2(\Omega)\cap \mathscr{H}, H^1(\Omega)\cap \mathscr{H}]_{1/2} = H^{1/2}(\Omega)\cap \mathscr{H},
$$
(see \cite{J-K} page 183) we deduce  by interpolation that $S \in \mathscr{L}(L^{2}(\Gamma); H^{1/2}(\Omega)\cap \mathscr{H})$.  

Besides, the trace operator $\gamma$ satisfies:
$$
\forall v\in H^1(\Omega)\cap \mathscr{H}, \quad \Vert \gamma v \Vert_{H^{1/2}(\Gamma)} \leq C(\Omega) \Vert v \Vert_{H^1(\Omega)}\quad 
$$
and
$$
\forall v\in L^2(\Omega)\cap \mathscr{H}, \quad \Vert \gamma v \Vert_{H^{-1/2}(\Gamma)} \leq C(\Omega) \Vert v \Vert_{L^2(\Omega)},
$$
(see Lemma 2 in  \cite{ARB} for this last inequality which holds since $  \Omega$ is $ \mathscr{C}^{1, 1}$).
By interpolation again, we deduce that
\begin{equation*}
\forall v\in H^{1/2}(\Omega)\cap \mathscr{H}, \quad \Vert \gamma v \Vert_{L^{2}(\Gamma)} \leq C(\Omega) \Vert v \Vert_{H^{1/2}(\Omega)}.
\end{equation*} 
That means that the trace operator $\gamma :  H^{1/2}(\Omega)\cap \mathscr{H} \rightarrow L^2(\Gamma)$ is continuous and is bijective.   Moreover $S = \gamma^{-1}$ and 
$$
 \Vert v \Vert_{H^{1/2}(\Omega)} \approx \Vert \gamma v \ \Vert_{L^{2}(\Gamma)} \quad \mathrm{for}\quad  v\in H^{1/2}(\Omega)\cap \mathscr{H}.
 $$
 This proves the first part of implication \eqref{impharC11}. \medskip
  
\noindent{\bf Step 2.}  Let us now prove the second part of implication \eqref{impharC11}. Let  $u \in H^{1/2}(\Omega)$ a harmonic function. Since $\Omega$ is of class $\mathscr{C}^{1, 1}$,  for any $\textit{\textbf F} \in \mathscr{D}(\Omega)^N$ there exists a unique 
 $\chi \in H^2(\Omega)\cap L^2_0(\Omega)$ such that 
\begin{equation}\label{estimNeu}
\Delta \chi = \mathrm{div}\, \textit{\textbf F} \quad \mathrm{in}\; \Omega \quad \mathrm{and}\quad  \partial_\textit{\textbf n}\chi = 0\quad \mathrm{on}\; \Gamma, \quad \mathrm{with}\quad \Vert \chi \Vert_{H^{3/2}(\Omega)} \leq C \Vert \textit{\textbf F}\,  \Vert_{\textit{\textbf  H}^{\, 1/2}(\Omega)}. 
\end{equation}
Using the harmonicity of $u$, we have
$$
\langle \nabla u, \textit{\textbf F}\, \rangle_{[\mathscr{D}'(\Omega)]^N \times \mathscr{D}(\Omega)^N} = - \int_\Omega u\,  \mathrm{div}\, \textit{\textbf F} \, dx = - \int_\Omega u\,  \Delta \chi \, dx.
$$
Since the regularity of the domain $\Omega$, we know that for any $ v\in L^2(\Omega)$ with $ \Delta v \in L^2(\Omega)$ and $\varphi \in H^2(\Omega)$, 
\begin{equation*}
 \int_\Omega v\Delta \varphi \, dx -  \int_\Omega \varphi \Delta v \, dx = \langle v, \partial_\textit{\textbf n} \varphi\rangle_{H^{-1/2}(\Gamma)\times H^{1/2}(\Gamma)} -  \langle \partial_\textit{\textbf n} v , \varphi \rangle_{H^{-3/2}(\Gamma)\times H^{3/2}(\Gamma)}.
\end{equation*}
Recall also that the Steklov operator (also called the Dirichlet-to-Neumann operator) denoted by $S_P$ satisfies the estimate (which holds even if $\Omega$ is only Lipschitz):
$$
\Vert S_P u \Vert_{H^{-1}(\Gamma)} \leq C(\Omega) \Vert  u \Vert_{L^{2}(\Gamma)}.
$$

So we deduce that
$$
\langle \nabla u, \textit{\textbf F}\, \rangle_{[\mathscr{D}'(\Omega)]^N \times \mathscr{D}(\Omega)^N}  = - \langle \partial_\textit{\textbf n} u, \chi\rangle_{H^{-1}(\Gamma)\times H^1(\Gamma)}
 $$
and then
$$
\vert \langle \nabla u, \textit{\textbf F}\, \rangle_{[\mathscr{D}'(\Omega)]^N \times \mathscr{D}(\Omega)^N} \vert \leq C(\Omega) \Vert u \Vert_{L^{2}(\Gamma)} \Vert \chi \Vert_{H^{1}(\Gamma)}.  
$$
As $\Omega$ is $\mathscr{C}^{1, 1}$, we know that
$$
\Vert \chi \Vert_{H^{1}(\Gamma)} \leq C(\Omega)\Vert \chi \Vert_{H^{3/2}(\Omega)}.
$$
Note that this last inequality does not occur when $\Omega$ is only Lipchitz. From the estimate in \eqref{estimNeu}, we finally deduce that
$$
\vert \langle \nabla u, \textit{\textbf F}\, \rangle_{[\mathscr{D}'(\Omega)]^N \times \mathscr{D}(\Omega)^N} \vert \leq C(\Omega)\Vert u \Vert_{L^2(\Gamma)} \Vert \textit{\textbf F}\,  \Vert_{\textit{\textbf  H}^{\, 1/2}(\Omega)}
$$
and then the estimate
$$
\Vert \nabla u \Vert_{[\textit{\textbf  H}^{\, 1/2}(\Omega)]'} \leq C(\Omega)\Vert u \Vert_{L^2(\Gamma)} \leq C(\Omega)\Vert u \Vert_{H^{1/2}(\Omega)} 
$$
by using the density of $\mathscr{D}(\Omega)^N$ in $\textit{\textbf  H}^{\, 1/2}(\Omega)$. This proves the second part of implication \eqref{impharC11}. \medskip

The isomorphism $\gamma :  H^{1/2}(\Omega)\cap \mathscr{H} \rightarrow L^2(\Gamma)$ is a simple consequence of the fact that for every $g\in L^2(\Gamma)$, there exists a unique harmonic function $w$ in $\Omega$ satisfying $w = g$ on $\Gamma$  (see Theorem 5.3 in \cite{J-K} or Theorem 8.4 in \cite{AM}). And we clearly have the identity $\gamma(E(\nabla, \Omega)) = L^2(\Gamma)$.\medskip

\noindent{\bf ii)} We have always the following  inclusion  
$$
E(\nabla, (\Omega))\cap \mathscr{H}\subset H^{1/2}(\Omega)\cap \mathscr{H}  
$$
even if $\Omega$ is only Lipschitz. The reverse inclusion is an immediate consequence of \eqref{impharC11}.\medskip

\noindent{\bf iii)}  The last property concerning the operator \eqref{nablaS} is an immediate consequence of Point ii) and the continuity of $S$.
\end{proof}
 
 \begin{remark}\upshape  Unlike in the case where $\Omega$ is of class $\mathscr{C}^{1, 1}$, the mapping  $\gamma :  H^{1/2}(\Omega)\cap \mathscr{H} \rightarrow L^2(\Gamma) $ is not well defined when $\Omega$ is only Lipschitz, see Section \ref{Counter Example}.
  \end{remark}

Recall the following Ne$\mathrm{\check{c}}$as property (see \cite{Necas}, Chapter 5):

\begin{theorem}[{\bf Ne$\mathrm{\check{c}}$as Property}]\label{NecProt} Let $u\in H^1(\Omega)$ with $\Delta u  \in\ L^2(\Omega)$,  then we have the following equivalences 

\begin{equation}\label{NecProi}
u_{\vert \Gamma}\in H^1(\Gamma) \Longleftrightarrow \partial_\textit{\textbf n} u\in L^2(\Gamma) \Longleftrightarrow  \nabla u \in \textit{\textbf L}^2(\Gamma).
\end{equation}
Moreover, in this case, we have the following estimates
\begin{equation*}
\Vert\partial_\textit{\textbf n} u\Vert_{ \textit{\textbf L}^2(\Gamma)}\ \leq\ C(\Omega)\Big(\inf_{k\in  \mathbb{R}}\left|\left |u + k \right |\right |_{H^1(\Gamma)} + \Vert \Delta u\Vert_{L^2(\Omega)}\Big)
\end{equation*}
and
\begin{equation*}
 \inf_{k\in  \mathbb{R}}\Vert u + k\Vert_{H^1(\Gamma)}  \ \leq\ C(\Omega)\Big( \Vert\partial_\textit{\textbf n} u\Vert_{ \textit{\textbf L}^2(\Gamma)} +  \Vert \Delta u\Vert_{L^2(\Omega)}\Big),
\end{equation*}
where $ C(\Omega)$ depends only on the Lipschitz constant of $\Omega$.
\end{theorem}

We will see below (see Corollary \ref{d02-110118-th1a}) that the above equivalences \eqref{NecProi} are valid if we replace the condition  $\Delta u  \in\ L^2(\Omega)$ by the weaker condition  $\Delta u  \in [H^{1/2}(\Omega)]'$.

\begin{theorem}  [{\bf Traces of harmonic functions  $H^{3/2}$}]\label{protrace32}
Assume that $\Omega$ is of class $\mathscr{C}^{1, 1}$. We have the following properties: let $u \in H^{3/2}(\Omega)$, then we have the following implication:
\begin{equation}\label{impharC1132}
u \quad \mathrm{harmonic}\quad   \Longrightarrow \quad \partial_\textit{\textbf n} u \in L^2(\Gamma) \quad \mathrm{and} \quad \nabla^2 u \in  [\textit{\textbf  H}^{\, 1/2}(\Omega)]'
\end{equation} 
In particular, the trace operator $\gamma_{\textit{\textbf n}} :  v \rightarrow \partial_\textit{\textbf n} v$ is an isomorphism from $ H^{3/2}(\Omega)\cap L^2_0(\Omega) \cap \mathscr{H} $ onto $ L^2_0(\Gamma)$. 
\end{theorem}

\begin{proof} We start by observe that $u_{\vert \Gamma} \in H^1(\Gamma) $ since $\Omega$ is of class $\mathscr{C}^{1, 1}$. So the first part of implication \eqref{impharC1132} is an immediate consequence of \eqref{NecProi}. Another way to obtain this result is to proceed as in the proof of Theorem \ref{protrace}.\\

 Let us now prove the second part of implication \eqref{impharC1132}. Let  $u \in H^{3/2}(\Omega)$ a harmonic function. Since $\Omega$ is of class $\mathscr{C}^{1, 1}$,  for any $f \in \mathscr{D}(\Omega)$ and any $j = 1, \ldots, N$, there exists a unique 
 $\chi_j \in H^2(\Omega)\cap L^2_0(\Omega)$ such that 
\begin{equation}\label{estimNeuj}\
\Delta \chi_j = \frac{\partial f}{\partial x_{j}} \quad \mathrm{in}\; \Omega \quad \mathrm{and}\quad  \partial_\textit{\textbf n}\chi_j = 0\quad \mathrm{on}\; \Gamma, \quad \mathrm{with}\quad \Vert \chi_j \Vert_{H^{3/2}(\Omega)} \leq C_j \Vert f  \Vert_{H^{1/2}(\Omega)}. 
\end{equation}
Using the harmonicity of $u$, we have for any $i = 1, \ldots, N$, 
$$
\langle \frac{\partial^2  u}{\partial x_{i}\partial x_{j}}, f\, \rangle_{\mathscr{D}'(\Omega) \times \mathscr{D}(\Omega)} = - \int_\Omega \frac{\partial u}{\partial x_{i}} \, \frac{\partial f}{\partial x_{j}} \, dx = - \int_\Omega \frac{\partial u}{\partial x_{i}}  \Delta \chi_j \, dx.
$$
Since the regularity of the domain $\Omega$, we know that for any $ v\in L^2(\Omega)$ with $ \Delta v \in L^2(\Omega)$ and $\varphi \in H^2(\Omega)$, 
\begin{equation*}
 \int_\Omega v\Delta \varphi \, dx -  \int_\Omega \varphi \Delta v \, dx = \langle v, \partial_\textit{\textbf n} \varphi\rangle_{H^{-1/2}(\Gamma)\times H^{1/2}(\Gamma)} -  \langle \partial_\textit{\textbf n} v , \varphi \rangle_{H^{-3/2}(\Gamma)\times H^{3/2}(\Gamma)}.
\end{equation*}
So we deduce from this Green formula that
$$
\langle \frac{\partial^2  u}{\partial x_{i}\partial x_{j}}, f \rangle_{\mathscr{D}'(\Omega) \times \mathscr{D}(\Omega)}   = - \langle \partial_\textit{\textbf n} (\frac{\partial u}{\partial x_{i}}), \chi_j\rangle_{H^{-1}(\Gamma)\times H^1(\Gamma)}.
 $$
Note that $\frac{\partial u}{\partial x_{i}} \in H^{1/2}(\Omega)$ and is harmonic. Hence $\frac{\partial u}{\partial x_{i}}_{\vert\Gamma} \in L^2(\Gamma)$ and $\partial_\textit{\textbf n} (\frac{\partial u}{\partial x_{i}}) \in H^{-1}(\Gamma)$. As in the proof of Theorem \ref{protrace}, we deduce that
$$
\vert \langle \frac{\partial^2  u}{\partial x_{i}\partial x_{j}}, f \rangle_{\mathscr{D}'(\Omega) \times \mathscr{D}(\Omega)} \vert \leq C_{ij}(\Omega) \Vert \frac{\partial u}{\partial x_{i}} \Vert_{L^{2}(\Gamma)} \Vert \chi_j \Vert_{H^{1}(\Gamma)} \leq C_{ij}(\Omega) \Vert \frac{\partial u}{\partial x_{i}} \Vert_{L^{2}(\Gamma)} \Vert \chi_j \Vert_{H^{3/2}(\Omega)}.  
$$
From the estimate in \eqref{estimNeuj}, we finally deduce that
$$
\vert \langle \frac{\partial^2  u}{\partial x_{i}\partial x_{j}}, f \rangle_{\mathscr{D}'(\Omega) \times \mathscr{D}(\Omega)} \vert\leq C_{ij}(\Omega) \Vert u \Vert_{H^{3/2}(\Omega)} \Vert f  \Vert_{H^{1/2}(\Omega)}.
$$
and then the estimate
$$
\Vert \frac{\partial^2  u}{\partial x_{i}\partial x_{j}}\Vert_{[H^{\, 1/2}(\Omega)]'} \leq C_{ij}(\Omega) \Vert u \Vert_{H^{3/2}(\Omega)} 
$$
by using the density of $\mathscr{D}(\Omega)$ in $H^{\, 1/2}(\Omega)$. This proves the second part of implication \eqref{impharC1132}. \medskip

The isomorphism $\gamma_{\textit{\textbf n}}  :  H^{3/2}(\Omega)\cap L^2_0(\Omega) \cap \mathscr{H} \rightarrow L^2_0(\Gamma)$ is a simple consequence of the fact that for every $h\in L^2_0(\Gamma)$, there exists a unique harmonic function $w\in H^{3/2}(\Omega)\cap L^2_0(\Omega) $ in $\Omega$ satisfying $ \partial_\textit{\textbf n}  w = h$ on $\Gamma$  (see \cite{Jer1}). And we have the following equivalence:
$$
 \Vert v \Vert_{H^{3/2}(\Omega)} \approx \Vert \gamma_{\textit{\textbf n}} v \ \Vert_{L^{2}(\Gamma)} \quad \mathrm{for}\quad  v\in H^{3/2}(\Omega)\cap L^2_0(\Omega) \cap \mathscr{H}.
 $$
 And we clearly have the identity  $\gamma_\textit{\textbf n}(E(\nabla^2, \Omega)) = L^2_0(\Gamma)$.
\end{proof}

\section{The inhomogeneous problem. Solvability in  $H^{s}(\Omega)$ for $1/2 \leq s \leq 3/2$}\label{Inhomogeneous Problem}
 
Here we will consider the following problem:
$$
(\mathscr{L}_D^0)\ \ \ \  -\Delta u = f\quad \ \mbox{in}\ \Omega \quad
\mbox{and } \quad u = 0  \ \ \mbox{on }\Gamma,
$$
for given $f$ in fractional Sobolev spaces.  It is well known that for any $f\in H^{s-2}(\Omega) $, with $1/2 < s < 3/2$, there exists a unique solution $u\in H^s(\Omega)$ to Problem $(\mathscr{L}_D^0)$. The limit cases $s = 3/2$ or $s = 1/2$ are particularly delicate. For $f\in L^2(\Omega)$, it is well known that there exists a unique solution $u\in H^{3/2}(\Omega)$ to problem $(\mathscr{L}_D^0)$ (see Theorem B in  \cite{J-K}) (in fact $f\in H^{s-2}(\Omega) $, with arbitrary $s > 3/2$, is suffisant). But these assumptions on $f$ are too strong as we can see below. It would be interesting to characterize the range of the Laplacian operator  from $H^{3/2}(\Omega)\cap H^1_0(\Omega)$ into $[H^{1/2}_{00}(\Omega)]'.$ In  \cite{J-K} (see Theorem 0.4) the authors show that it is not possible for the operator 
\begin{equation}\label{DeltaIsoH3/2}
\Delta : H^{3/2}(\Omega)\cap H^1_0(\Omega) \longrightarrow [H^{1/2}_{00}(\Omega)]'.
\end{equation}
to be an isomorphism, even if $\Omega$ is of class $\mathscr{C}^1$. Their proof is based on the argument, which consists to say that if a harmonic function $v$ belongs to $H^{3/2}(\Omega)$, then its trace satisfies $v_{\vert \Gamma} \in H^1(\Gamma)$. However, Proposition \ref{ConterexampleH1demitrace1b} asserts that this is not always the case when the domain $\Omega$ is only Lipschitz. In the following theorem we prove that the operator \eqref{DeltaIsoH3/2} is really an isomorphism.

\begin{theorem}\label{IsoDeltaH3/2H1/2} The following operators 
\begin{equation}\label{DeltaIsoH3/2H1/2}
\Delta : H^{3/2}_0(\Omega) \longrightarrow [{H}^{1/2}_{00}(\Omega)]'\quad \mathrm{and} \quad \Delta : H^{1/2}_{00}(\Omega) \longrightarrow  [H^{3/2}_0(\Omega)]'
\end{equation}
are isomorphisms.
\end{theorem}

\begin{proof} {\bf Step 1.}  It suffices to consider the case $N = 3$, the proof being similar for the other dimensions. Let $(\Omega_k)_k$  be a sequence of polyedral open sets included in $\Omega$ and converging to $\Omega$. We can choose this sequence such that the sequence $(L_k)_k$ of Lipschitz constants of $\Omega_k$ is bounded. Let  $\varphi \in \mathscr{D}(\Omega)$ and $k_0$ such that supp $\varphi \subset \Omega_{k_0}$.  Then using Proposition \ref{corGrisvardPol}, we have 
$$
\Vert \varphi \Vert_{H^{3/2}(\Omega) } = \Vert \varphi \Vert_{H^{3/2}(\Omega_{k_0}) } \leq C(\Omega_{k_0})\Vert \Delta \varphi  \Vert_{M(\Omega_{k_0})},
$$
where the constant $C(\Omega_{k_0})$ above  depends only on the Poincar\'e constant  of  $\Omega_{k_0}$  and on $L_{k_0}$. So  the constant  $C(\Omega_{k_0})$ is bounded by a constant $C(\Omega)$ which depends only on the Lipschitz constant $L$ and on the Poincar\'e constant of $\Omega$. 

From Theorem \ref{theoHsregul}, Point iii) we know that
$$
\Vert \Delta \varphi \Vert_{M(\Omega_{k_0})} \leq  C'(\Omega_{k_0}) \Vert \Delta \varphi \Vert_{[H^{1/2}_{00}(\Omega_{k_0})]'}.
$$
Besides, for any $f\in [H^{1/2}_{00}(\Omega)]'$ we have
\begin{equation*}
\begin{array}{rl}
\Vert f \Vert_{[H^{1/2}_{00}(\Omega_{k_0})]'}  &=  \displaystyle\sup_{\psi \in H^{1/2}_{00}(\Omega_{k_0}), \, \psi \not=  0}\dfrac{\langle f, \psi\rangle}{\Vert \psi\Vert_{H^{1/2}_{00}(\Omega_{k_0})}} =  \displaystyle\sup_{\psi  \in H^{1/2}_{00}(\Omega_{k_0}), \, \psi \not=  0}\dfrac{\langle f, \widetilde{\psi}\rangle}{\Vert \widetilde{\psi}\Vert_{H^{1/2}_{00}(\Omega)}}\leq \Vert f \Vert_{[H^{1/2}_{00}(\Omega)]'}
\end{array}
\end{equation*}
where $\widetilde{\psi}$ is the extension by zero inside $\Omega$ of $\psi$. So we get for any $\varphi \in \mathscr{D}(\Omega)$
$$
\Vert \varphi \Vert_{H^{3/2}(\Omega) }  \leq C(\Omega)\Vert \Delta \varphi \Vert_{[H^{1/2}_{00}(\Omega_{k_0})]'}\leq C'(\Omega)\Vert \Delta \varphi \Vert_{[H^{1/2}_{00}(\Omega)]'}. 
$$
Using the density of $\mathscr{D}(\Omega)$ in $H^{3/2}_0(\Omega)$, we deduce  the same estimate for any  $\varphi \in H^{3/2}_0(\Omega)$.\medskip

\noindent{\bf Step 2.} That means that the range $R (\Omega)$ of the following mapping
$$
\Delta : H^{3/2}_0(\Omega) \longrightarrow [H^{1/2}_{00}(\Omega)]'
$$
is closed in $ [H^{1/2}_{00}(\Omega)]'$. So we have $R(\Omega) = $ (Ker $\Delta)^\bot$, with $\Delta$ is given by the operator 
\begin{equation}\label{DeltaIsoH1/2H3/2Polyb}
\Delta : H^{1/2}_{00}(\Omega) \longrightarrow [{H}^{3/2}_{0}(\Omega)]'
\end{equation}
and where
$$
(\mathrm{Ker}\, \Delta)^\bot = \{f\in [H^{1/2}_{00}(\Omega)]'; \; \langle f, v \rangle_{[H^{1/2}_{00}(\Omega)]'\times H^{1/2}_{00}(\Omega)} = 0, \; \forall v \in \mathrm{Ker}\, \Delta\}.
$$
Observe now that the Kernel of the operator \eqref{DeltaIsoH1/2H3/2Polyb} is trivial. As a consequence, we get the surjectivity of the first operator in \eqref{DeltaIsoH3/2H1/2} and by duality the second isomorphism.
\end{proof}

\begin{theorem}\label{isoHsbTheo} For any $0 < s < 1$, the following operator
\begin{equation}\label{isoHsb}
\Delta : H^{3/2- s}_0(\Omega) \rightarrow H^{-1/2-s}(\Omega)
\end{equation}
is an isomorphism.
\end{theorem}

\begin{proof} i) Recall that the operators
\begin{equation*}
\Delta : H^{3/2}_0(\Omega) \rightarrow (H^{1/2}_{00}(\Omega))'  \quad \mathrm{and}\quad \Delta : H^{1}_0(\Omega) \rightarrow H^{-1}(\Omega)
\end{equation*}
are isomorphisms.
Therefore, as in the proof of Corollary \ref{firstcor},  the same applies to the operator
\begin{equation*}
\Delta : [H^{3/2}_0(\Omega),  H^{1}_0(\Omega)]_\theta  \rightarrow [(H^{1/2}_{00}(\Omega))' ,  H^{-1}(\Omega)]_\theta 
\end{equation*}
for any $0 < \theta < 1$, since the density of $H^{3/2}_0(\Omega)$ in $H^  {1}_0(\Omega)$. We know that 
\begin{equation*}
[H^{3/2}_0(\Omega),  H^{1}_0(\Omega)]_\theta = H^{3/2-\theta/2}_0(\Omega) \quad \mathrm{and}\quad [(H^{1/2}_{00}(\Omega))' ,  H^{-1}(\Omega)]_\theta = ([H^{1}_{0}(\Omega) ,  H^{1/2}_{00}(\Omega)]_{1-\theta})'.
\end{equation*}
As
\begin{equation*}
H^{1}_{0}(\Omega) = [H^{2}_0(\Omega),  L^2(\Omega)]_{1/2}  \quad \mathrm{and}\quad  H^{1/2}_{00}(\Omega) = [H^{2}_0(\Omega),  L^2(\Omega)]_{3/4}
\end{equation*}
by reiteration theorem (see \cite{Lions}, Theorem 6.1, Chapter 1), we get 
\begin{equation*}
 [ H^{2}_0(\Omega),  L^2(\Omega)]_{1/2},  [H^{2}_0(\Omega),  L^2(\Omega)]_{3/4}]_{1-\theta} = [H^{2}_0(\Omega),  L^2(\Omega)]_{\theta(1/2)  + (1-\theta) (3/4)}.
\end{equation*}
But
\begin{equation*}
 [H^{2}_0(\Omega),  L^2(\Omega)]_{\theta(1/2)  + (1-\theta) (3/4)} = H^{1/2 + \theta/2}_0(\Omega).
\end{equation*}
Setting $s= \theta/2$, we deduce the isomorphism \eqref{isoHsb} for any $0 < s <  1/2$ and obviously also for $s = 1/2$. \medskip

\noindent ii) As  for any $0 < s \leq  1/2$, we have $(H^{3/2- s}_0(\Omega))' = H^{-3/2+s}(\Omega)$ and $ (H^{-1/2-s}(\Omega))' = H^{1/2+s}_0(\Omega)) $, we deduce by duality  that the operator \eqref{isoHsb} is still an  isomorphism for any $s > 1/2$.
\end{proof}

\section{The homogeneous problem}\label{SectionDPL}

\subsection{Ne\v{c}as Property, first version}\label{NecFV}
 
In the following theorem, we give a variant of the Ne\v{c}as Property (see Theorem \ref{NecProt}), for functions in $H^1_0(\Omega)$ but with Laplacian less regular.

\begin{theorem}\label{CoruH10DeltaudualH1/2}
Let 
\begin{equation*}\label{uH10DeltauH1/2prime}
u\in H^1_0(\Omega)\quad \mathrm{ with} \quad\Delta u  \in\ [H^{1/2}(\Omega)]'.
\end{equation*}
Then $\partial_\textit{\textbf n} u\in L^2(\Gamma)$ and we have the following estimate
\begin{equation*}\label{estdudnL2DeltaudualH1/2}
 \Vert\partial_\textit{\textbf n} u\Vert_{ \textit{\textbf L}^2(\Gamma)}\ \leq\ C(\Omega) \Vert\Delta u\Vert_{ [H^{1/2}(\Omega)]'}, 
 \end{equation*}
 where $C(\Omega)$ depends on the Lipschitz constant of $\Omega$.

\end{theorem}

\begin{proof} {\bf Step 1.} It suffices to deal with the case when $\Omega$ is a Lipschitz hypograph:
 $$
\Omega\ = \ \left\{\,(\textit{\textbf x}',x_N)\in \mathbb{R}^N;\ \ x_N<\xi(\textit{\textbf x}') \right\},
$$
where $\xi\in \mathscr{C}^{0,1}(\mathbb{R}^{N-1})$ with supp $\xi$ compact.   Observe that, with the same proof as in \cite{Gri}, we can show that
$$
\mathscr{D}(\overline{\Omega})  \quad \mathrm{is\,\, dense\,\, in}\quad E(\Delta;\, \Omega) = \{v\in H^1(\Omega); \; \Delta v \in [H^{1/2}(\Omega)]'\}.
$$
So that $\partial_\textit{\textbf n} u\in H^{-1/2}(\Gamma)$ and we have the following Green's formula: for any $\varphi \in H^1(\Omega)$, 
\begin{equation}\label{GreenFormulaH1/2prime}
\int_\Omega \nabla u\cdot \nabla \varphi \, dx + \langle \Delta u,\,  \varphi\rangle_{[H^{1/2}(\Omega)]'\times H^{1/2}(\Omega)} = \langle  \partial_\textit{\textbf n} u,\, \varphi\rangle_{ H^{-1/2}(\Gamma)\times \ H^{1/2}(\Gamma)}.
\end{equation}

Choose a sequence $(\xi_k)_k$ in $\mathscr{C}^\infty(\mathbb{R}^{N-1})$ such that for any $k\geq 1$  we have  $\xi_k\leq \xi  \textrm{ on }  \mathbb{R}^{N-1}$, $\xi_k(\textit{\textbf x}')=\xi(\textit{\textbf x}')$ if $|\textit{\textbf x}'|\geq R_0$ with $ \, \nabla\xi_k\,_{L^{\infty}(\mathbb{R}^{N-1})}\leq C$  and  for any $1\leq p<\infty$
$$
\xi_k\  \longrightarrow \ \ \xi\ \textrm{ in }  L^{\infty}(\mathbb{R}^{N-1})\quad \mathrm{ and}\quad   \nabla\xi_k\  \longrightarrow \ \ \nabla\xi\ \textrm{ in }  L^{p}(\mathbb{R}^{N-1}).
$$
We set 
$$
\Omega_k\, = \ \left\{\,(\textit{\textbf x}',x_N)\in\mathbb{R}^N;\ \ x_N<\xi_k(\textit{\textbf x}') \right\}.
$$

\noindent {\bf Step 2.} By the first isomorphism in Theorem \ref{IsoDeltaH3/2H1/2}, we deduce from \eqref{uH10DeltauH1/2prime} that $u\in H^{3/2}_0(\Omega)$. Setting $f = \Delta u$ and let $(f_k)_k \subset \mathscr{D}(\Omega)$ be such that $f_k \longrightarrow f$ in $ [H^{1/2}(\Omega)]'$as $k\rightarrow \infty$. Let $u_k \in H^{3/2}_0(\Omega_k)\cap H^2(\Omega_k)$ be the unique solution satisfying $\Delta u_k  - u_k = f_k - u$ in $\Omega_k$ as $k\rightarrow \infty$. Setting now
\begin{equation*}
\widetilde{u_k} = \begin{cases}u_k \quad \mathrm{in}\;\; \Omega_k \\
0\quad \; \; \textrm{in}\; \; \Omega\setminus \Omega_k
\end{cases}
\end{equation*}
Clearly, the sequence $(\widetilde{u_k} )_k$ is bounded in $H^1_0(\Omega)$ and for any   $\varphi \in H^1(\Omega)$, 
\begin{equation}\label{GreenFormulakH1/2prime}
\begin{array}{rl}
\displaystyle\int_{\Omega_k} (u_k \varphi\ + \nabla u_k\cdot \nabla \varphi ) \, dx &= - \displaystyle \int_{\Omega_k}  (f_k - u)\,  \varphi \, dx+ \int_{\Gamma_k}  \varphi  \partial_{\textit{\textbf n}_k} u_k \, d\sigma_k\\\\
\displaystyle\int_{\Omega} (u \varphi\ + \nabla u\cdot \nabla \varphi ) \, dx &=  \displaystyle\int_{\Omega}   u\,  \varphi \, dx  -  \langle \Delta u, \varphi \rangle_\Omega + \langle \partial_\textit{\textbf n} u, \, \varphi\rangle_\Gamma .
\end{array}
\end{equation}
Substract  \eqref{GreenFormulaH1/2prime} from  \eqref{GreenFormulakH1/2prime}, we get for any $\varphi\in H^1(\Omega)$
\begin{equation*}\label{FormGreenpart}
\begin{array}{rl}
&\displaystyle\int_{\Omega} [(\widetilde{u_k}-u) \varphi + \nabla(\widetilde{u_k}-u)\cdot\nabla\varphi]\, dx\\\\ 
& \displaystyle =  \left[\int_{\Omega} (\widetilde{u_k} \varphi + \nabla \widetilde{u_k}\cdot\nabla\varphi) \, dx
-\int_{\Omega_k} (u_k\varphi + \nabla u_k\cdot\nabla\varphi)\, dx \right]  + \\\\     
&\displaystyle +   \left[ \int_{\Omega_k} (u_k\varphi  + \nabla u_k\cdot\nabla\varphi) \, dx
  -\int_{\Omega}(u \varphi + \nabla u\cdot\nabla\varphi) \, dx\right]\\\\
&= \displaystyle  \,\int_{\Gamma_k} \varphi  \partial_{\textit{\textbf n}_k} u_k  \, d\sigma_k  - \langle  \partial_\textit{\textbf n} u,\, \varphi\rangle_{\Gamma}\,   +\;  \langle f - f_k, \varphi\rangle_\Omega + \int_{\Omega\setminus\Omega_k}f_k\, \varphi \, dx- \int_{\Omega\setminus\Omega_k}  u \varphi\, dx.
\end{array}
\end{equation*}
So taking $\varphi =  \widetilde{u_k} - u\in H^1_0(\Omega)$, we deduce that
\begin{equation}\label{estimuktile-ubis}
\begin{array}{rl}
&\Vert  \widetilde{u_k} - u\Vert^2_{H^1(\Omega)} \leq  \Vert  \widetilde{u_k} - u\Vert_{H^1(\Omega)}\big( \Vert  u \Vert_{L^2(\Omega\setminus\Omega_k)}  +\;  \Vert f - f_k\Vert_{[H^{1/2}(\Omega)]'}\big) \\\\
&  +\;   \Vert    \partial_{\textit{\textbf n}_k} u_k\Vert_{L^2(\Gamma_k)} \Vert  \widetilde{u_k} - u\Vert_{L^2(\Gamma_k)}  +\;  \Vert  f_k\Vert_{[H^{1/2}(\Omega\setminus \Omega_{k})]'}\Vert  \widetilde{u_k} - u\Vert_{H^1(\Omega\setminus \Omega_{k})]} .
\end{array}
\end{equation}

\noindent {\bf Step 3.} Besides, recall the following Rellich's identity: for any $v\in H^2(\Omega)\cap H^1_0(\Omega)$: 
\begin{equation}\label{RellichH2interH10}
 \int_\Gamma \textbf{\textit h}\cdot\textbf{\textit n}\left| \partial_{\textbf{\textit n}} v \right|^2 d\sigma  = 
\int_\Omega\left[2\textbf{\textit h}\,\Delta v+ 2\frac{\partial \textbf{\textit h}}{\partial x_k}\frac{\partial v}{\partial x_k}-(\mathrm{div}\, \textbf{\textit h})\nabla v\right]\cdot\nabla v\, dx.
\end{equation}
Using then \eqref{RellichH2interH10} with $v = u_k \in H^2(\Omega_k)\cap H^1_0(\Omega_k)$ and the first isomorphism in Theorem \ref{IsoDeltaH3/2H1/2}, we obtain
\begin{equation}\label{RellichH2interH10a}
\begin{array}{rl}
 \Vert   \partial_{\textit{\textbf n}_k} u_k\Vert^2_{L^2(\Gamma_k)}&\leq C  \big(  \Vert \nabla u_k \Vert_{H^{1/2}(\Omega_k)}  \Vert \Delta u_k \Vert_{[H^{1/2}(\Omega_k)]'} +   \Vert \nabla u_k \Vert^2_{L^{2}(\Omega_k)}\big)\\\\
& \leq C  \big(  \Vert u_k \Vert_{H^{3/2}(\Omega_k)}  \Vert \Delta u_k \Vert_{[H^{1/2}(\Omega_k)]'}\big) \leq C  \Vert \Delta u_k \Vert^2_{[H^{1/2}(\Omega_k)]'},
\end{array}
\end{equation}
where $C$ depends only on the Lipschitz  constant of $\Omega$ and does not depend on $k$. But
\begin{equation}\label{estimderivnormL2}
\begin{array}{rl}
 \Vert \Delta u_k \Vert_{[H^{1/2}(\Omega_k)]'}& \leq \Vert u - u_k \Vert_{[H^{1/2}(\Omega_k)]'} + \Vert f_k \Vert_{[H^{1/2}(\Omega_k)]'}  \\ \\
 &\leq C(\Vert f \Vert_{[H^{1/2}(\Omega)]'} + \Vert f_k \Vert_{[H^{1/2}(\Omega_k)]'}  )\leq C\Vert f \Vert_{[H^{1/2}(\Omega)]'}. 
\end{array}
\end{equation}
Setting now
$$
\theta_k(\textit{\textbf x}')=(1+\vert\nabla\xi_k(\textit{\textbf x}')\vert^2)^{1/2},\ \quad\ 
\theta(\textit{\textbf x}')=(1+\vert\nabla\xi(\textit{\textbf x}')\vert^2)^{1/2}
$$
we observe that for any $1\leq p < \infty$ 
$$
1\leq \theta_k\leq C, \quad \theta_k\rightarrow \theta\ \textrm{ in } L^p(\mathbb{R}^{N-1}).
$$
Then we have
\begin{equation}\label{InegL2Gammakbis}
\begin{array}{rl}
\!\!\!\!\Vert \widetilde{u_k} - u\Vert^2_{L^2(\Gamma_k)}\!\!\!\! & = \displaystyle  \int_{\mathbb{R}^{N-1}}\vert u(\textit{\textbf x}',\xi_k(\textit{\textbf x}'))\vert^2 \theta_k(\textit{\textbf x}') d\textit{\textbf x}'\\\\
& = \displaystyle \int_{\mathbb{R}^{N-1}}\vert u(\textit{\textbf x}',\xi_k(\textit{\textbf x}')) - u(\textit{\textbf x}',\xi(\textit{\textbf x}')\vert^2 \theta_k(\textit{\textbf x}') d\textit{\textbf x}'\\\\
\!\!\!\! & \displaystyle \leq C \int_{\mathbb{R}^{N-1}} (\xi(\textit{\textbf x}') - \xi_k(\textit{\textbf x}') )    \int_{\xi_k(\textit{\textbf x}')}^{\xi(\textit{\textbf x}')} \left|\frac{\partial u}{\partial x_N}(\textit{\textbf x}',x_N)\right|^2 dx_N  d\textit{\textbf x}'\\\\
\!\!\!\! & \leq C  \Vert \xi-\xi_k\Vert_{L^\infty(\mathbb{R}^{N-1})}\Vert u \Vert^2_{H^1(\Omega\setminus\Omega_k)}. 
\end{array}
\end{equation}
Using \eqref{estimuktile-ubis}, \eqref{RellichH2interH10a}, \eqref{estimderivnormL2} and \eqref{InegL2Gammakbis}, we deduce the strong convergence in $H^1(\Omega)$ of $\widetilde{u_k} $ to $u$.

Let us introduce
 $$
 \psi_k(\textit{\textbf x}')\ = \ \nabla u_k\cdot\textit{\textbf n}_k(\textit{\textbf x}',\xi_k(\textit{\textbf x}')).
 $$
From \eqref{RellichH2interH10a} and \eqref{estimderivnormL2} $(\psi_k)_k$ is bounded in $L^2(\mathbb{R}^{N-1})$, so we can extract a subsequence, again denoted by $(\psi_k)_k$ such that
$$
\psi_k\ \rightharpoonup\ \psi \ \ \textrm{ in } L^2(\mathbb{R}^{N-1}).
$$
Setting now 
$$
\widetilde\psi(\textit{\textbf x}',\xi(\textit{\textbf x}'))=\psi(\textit{\textbf x}')
$$
and let $\varphi\in H^1(\Omega)$. Then as $\theta_k\ \rightarrow\ \theta$ in $L^2( \mathbb{R}^{N-1})$, we have
\begin{equation*}\label{d04-090318-e4}
\begin{array}{rl}
\displaystyle\int_{\Gamma_k} \varphi \partial_{\textit{\textbf n}_k} u_k \, d\sigma_k\ & = \displaystyle \int_{ \mathbb{R}^{N-1}}\psi_k(\textit{\textbf x}')\varphi(\textit{\textbf x}',\xi_k(\textit{\textbf x}'))\theta_k(\textit{\textbf x}')\, d\textit{\textbf x}'\\\\
& \displaystyle \longrightarrow \int_{ \mathbb{R}^{N-1}}\psi(\textit{\textbf x}')\varphi(\textit{\textbf x}',\xi(\textit{\textbf x}'))\theta(\textit{\textbf x}')\, d\textit{\textbf x}'\ =\ \int_\Gamma\widetilde\psi\varphi \, d\sigma.
\end{array}
\end{equation*}
Here, note that 
\begin{equation*}
\Vert \widetilde\psi \Vert_{L^2(\Gamma)} \leq C \Vert \psi \Vert_{L^2( \mathbb{R}^{N-1})} \leq C \liminf_{k\rightarrow \infty} \Vert \psi_k\Vert_{L^2( \mathbb{R}^{N-1})}
\leq C\Vert f \Vert_{[H^{1/2}(\Omega)]'} .
\end{equation*}
Sending $k \rightarrow \infty$ in \eqref{GreenFormulakH1/2prime} gives

\begin{equation*}\label{d04-090318-e10}
\int_{\Omega}\nabla u\cdot\nabla\varphi\, dx =\ -\langle f,\, \varphi\rangle_{[H^{1/2}(\Omega)]'\times H^{1/2}(\Omega)} + \int_\Gamma\widetilde\psi\varphi\, d\sigma
\end{equation*}
and thanks to \eqref{GreenFormulaH1/2prime}, we get that $\partial_{\textit{\textbf n}} u=\widetilde\psi$ belongs to $L^2(\Gamma)$ with the estimate (\ref{estdudnL2DeltaudualH1/2}).
\end{proof}

\subsection{Solutions in $H^s(\Omega)$ with $1/2\leq s\leq 3/2$}\label{PDH1232}

We are now in position to state our first existence result in the case of a boundary data in $L^2(\Gamma)$.  Our proof is essentially based on the isomorphism given in Theorem \ref{IsoDeltaH3/2H1/2}.

\begin{theorem}\label{ThIsogL2Gamma} i) For any $g\in L^{2}(\Gamma)$, Problem $(\mathscr{L}_D^H)$ has a unique solution $u\in H^{1/2}(\Omega)$. Moreover $\sqrt \varrho\, \nabla u  \in {\textit{\textbf L}}^{2}(\Omega)$ and we have the estimate 
\begin{equation*}\label{d02-240118-e1c1}
 \Vert u\Vert_{H^{1/2}(\Omega)} + \Vert\sqrt \varrho\, \nabla u \Vert_{\textit{\textbf L}^2(\Omega)}  \leq \ C \Vert g \Vert_{L^{2}(\Gamma)}.
\end{equation*}
ii) This solution satisfies the following relation: for any $s \in\,  [0, 1/2[$ and $\varphi \in H^{2-s}(\Omega)\cap H^1_0(\Omega)$, we have
\begin{equation*}\label{d02-240118-e1c1ba}
\langle  u, \,  \Delta \varphi\rangle_{H^s(\Omega)\times H^{-s}(\Omega)} = \int_\Gamma g\partial_{\textit{\textbf n}}\varphi\, d\sigma .
\end{equation*}
iii) Moreover $u$ satisfies also the following property: for any positive integer $k$
$$
\varrho^{k + 1/2}\nabla^{k + 1}u \in \textit{\textbf L}^2(\Omega).
$$
\end{theorem}

\begin{proof} We first observe that the above property in Point iii) and the estimate of $\sqrt \varrho\, \nabla u$ in Point i)  are a direct consequence of Theorem \ref{inegharm}. Since uniqueness is given by Theorem \ref{UniciteH1/2}, it suffices to prove the existence of a solution $u$ in $  H^{\, 1/2}(\Omega)$. \smallskip

 Let $g_k \in H^{1/2}(\Gamma)$ be such that $g_k \longrightarrow g$ in $L^2(\Gamma)$ as $k\rightarrow \infty$ and $u_k \in H^1(\Omega)$ satisfying
$$
\Delta u_k = 0 \quad \mathrm{in}\; \Omega \qquad \mathrm{and}\qquad u_k = g_k\quad \mathrm{on}\; \Gamma.
$$
We know that
\begin{equation*}
\Vert u_k \Vert_{H^{1/2}(\Omega)} = \sup_{f\in [H^{1/2}(\Omega)]', f\not= 0}\frac{\vert \langle f, \, u_k\rangle_\Omega \vert}{\Vert f\Vert_ {[H^{1/2}(\Omega)]'}},
\end{equation*}
where $ \langle f, \, u_k\rangle_\Omega =  \langle f, \, u_k\rangle_{ [H^{1/2}(\Omega)]'\times H^{1/2}(\Omega)}$.  
But for any $f\in [H^{1/2}(\Omega)]'$, using Theorem  \ref{IsoDeltaH3/2H1/2}, there exits a unique $v\in H^{3/2}_0(\Omega)$ satisfying $\Delta v = f$ in $\Omega$. Moreover thanks to Theorem \ref{CoruH10DeltaudualH1/2} we have $\partial_{\textit{\textbf n}} v\in L^2(\Gamma)$ and
$$
\Vert v \Vert_{H^{3/2}(\Omega)} + \Vert  \partial_{\textit{\textbf n}} v \Vert_{L^2(\Gamma)}\leq C \Vert f\Vert_ {[H^{1/2}(\Omega)]'}.
$$
Now
\begin{equation*}
\begin{array}{rl}
\langle f, \, u_k\rangle_{\Omega } &= \displaystyle  - \int_\Omega \nabla u_k \cdot \nabla v \, dx + \int_\Gamma g_k   \partial_{\textit{\textbf n}} v\, d\sigma\\\\
& =  \displaystyle\int_\Omega  v \Delta u_k \, dx -  \langle \partial_{\textit{\textbf n}} u_k, \, v\rangle_{ H^{-1/2}(\Gamma)\times H^{1/2}(\Gamma)} + \int_\Gamma g_k   \partial_{\textit{\textbf n}} v\, d\sigma\\\\
&= \displaystyle \int_\Gamma g_k   \partial_{\textit{\textbf n}} v\, d\sigma.
\end{array}
\end{equation*}
Hence 
\begin{equation*}
\vert\langle f, \, u_k\rangle_\Omega\vert = \vert \int_\Gamma g_k   \partial_{\textit{\textbf n}} v\, d\sigma \vert \leq \Vert g_k \Vert_{L^2(\Gamma)} \Vert   \partial_{\textit{\textbf n}} v\Vert_{L^2(\Gamma)} \leq C \Vert g_k \Vert_{L^2(\Gamma)} \Vert \Vert f\Vert_ {[H^{1/2}(\Omega)]'}.
\end{equation*} 
This shows the estimate
\begin{equation}\label{estimnormH1demigbb}
\Vert u_k \Vert_{H^{1/2}(\Omega)} \leq C \Vert g_k \Vert_{L^2(\Gamma)}
\end{equation}
and consequently  $(u_k)$ is a Cauchy sequence in $H^{1/2}(\Omega)$ which converges to some  function $u\in H^{1/2}(\Omega)$ which is harmonic since for any $k$ the function  $u_k$ is harmonic. Clearly by Green formula we have for any $s \in\,  [0, 1/2[$ and $\varphi \in H^{2-s}(\Omega)\cap H^1_0(\Omega)$,
$$
\langle  u_k, \,  \Delta \varphi\rangle_{H^s(\Omega)\times H^{-s}(\Omega)}= \int_\Gamma g_k\partial_{\textit{\textbf n}}\varphi \, d\sigma.
$$
Passing to the limit above we get the following relation: for any  $\varphi \in H^{2-s}(\Omega)\cap H^1_0(\Omega)$
$$
\langle  u, \,  \Delta \varphi\rangle_{H^s(\Omega)\times H^{-s}(\Omega)} = \int_\Gamma g\partial_{\textit{\textbf n}}\varphi \, d\sigma
$$
with the estimate 
\begin{equation*}
\Vert u \Vert_{H^{1/2}(\Omega)} \leq C \Vert g \Vert_{L^2(\Gamma)}
\end{equation*}
thanks to \eqref{estimnormH1demigbb}. From the above equality and Definition \ref{deftr}, that means that $ u =  g$ on $\Gamma$.
\end{proof}

\begin{theorem} i) Let $f\in [H^{1/2}(\Omega)]' $ satisfying the condition 
\begin{equation}\label{condorthoH1/2b}
 \forall  \varphi \in H^{1}(\Omega)\cap \mathscr{H}, \quad \langle f, \, \varphi  \rangle = 0.
\end{equation} 
Then there exists a unique function $u\in  H^{3/2}_{00}(\Omega)$ satisying $\Delta u = f$ in $\Omega$. \\
ii) If moreover $f\in L^2(\Omega) $ and
\begin{equation*}
 \forall  \varphi \in L^{2}(\Omega)\cap \mathscr{H}, \quad \int_\Omega f \, \varphi  \, dx = 0,
\end{equation*} 
then $u\in H^2_0(\Omega)$.
\end{theorem}

\begin{proof} 
 i) Using Theorem \ref{IsoDeltaH3/2H1/2}, we know the existence of  $u\in  H^{3/2}_{0}(\Omega)$ such that $\Delta u = f$ in $\Omega$. So it suffices to prove that $u\in  H^{3/2}_{00}(\Omega)$. From Theorem \ref{CoruH10DeltaudualH1/2}, we deduce that $\partial_{\textit{\textbf n}}u \in L^2(\Gamma)$. We will show that $\partial_{\textit{\textbf n}}u = 0$. For that, let $\mu\in H^{1/2}(\Gamma)$ and $\varphi \in  H^{1}(\Omega)\cap \mathscr{H}$ such that $\varphi = \mu$ on $\Gamma$.  Now using the Green formula  \eqref{estdudnL2DeltaudualH1/2} we get:
$$
\int_\Gamma \mu \partial_{\textit{\textbf n}}u \, d\sigma = \int_\Gamma \varphi  \partial_{\textit{\textbf n}}u \, d\sigma= \langle \Delta u, \, \varphi\rangle -  \langle  u, \, \Delta\varphi\rangle = 0.
$$
That means that $\partial_{\textit{\textbf n}}u = 0$ in the sense $H^{-1/2}(\Gamma)$ and also $a.e$ on $\Gamma$ since $\partial_{\textit{\textbf n}}u \in L^2(\Gamma)$. 

It remains to show that $\nabla u \in   \textit{\textbf H}^{\,1/2}_{00}(\Omega)$.   This property follows directly from the fact that the operator $\Delta$ is an isomorphism from $ H^{1/2}_{00}(\Omega)$ into $\left[ H^{3/2}_{0}(\Omega)\right]'$ (see \eqref{DeltaIsoH3/2H1/2}). Indeed, let $\textit{\textbf z} \in \textit{\textbf H}^{\,1/2}_{00}(\Omega)$ satisfying $\Delta \textit{\textbf z}  =  \nabla f$ in $\Omega$. So the vector field $\textit{\textbf w} = \nabla u - \textit{\textbf z}$ belongs to $\textit{\textbf H}^{\,1/2}(\Omega)$ and is harmonic.  Using Theorem \ref{TracesH1demigradH1demiprime}  we know that $\textit{\textbf z}  = {\bf 0}$ on $ \Gamma$. As $\nabla u   = {\bf 0}$ on $ \Gamma$,  the trace of $\textit{\textbf w}$ is also zero on $\Gamma$. We deduce then by Theorem \ref{UniciteH1/2}  that  $\textit{\textbf w}$ is identically zero  in $\Omega$. Finally, as $u \in H^{3/2}_0(\Omega)$ and $\nabla u \in   \textit{\textbf H}^{\,1/2}_{00}(\Omega)$, we obtain from Corollary 1.4.4.10 in \cite{Gri} that $u\in H^{3/2}_{00}(\Omega)$.\medskip

\noindent ii) Since 
$$
\forall v\in H^2_0(\Omega), \quad \sum_{i,j = 1}^N \Vert \frac{\partial^2 v}{\partial x_i\partial x_j}\Vert^2_{L^2(\Omega)} = \Vert \Delta v\Vert^2_{L^2(\Omega)},
$$
the range of $H^2_0(\Omega)$ by the Laplacian is closed in $L^2(\Omega)$. The required result is then a consequence  of the Fredholm alternative
\end{proof}  

\begin{remark}\upshape Since $H^1(\Omega)$ is dense in $H^{1/2}(\Omega)$, the condition \eqref{condorthoH1/2b} is equivalent to the following:
\begin{equation*}
 \forall  \varphi \in H^{1/2}(\Omega)\cap \mathscr{H}, \quad \langle f, \, \varphi  \rangle = 0.
\end{equation*}

\end{remark}
The next result is a complement of Theorem \ref{TracesH1demigradH1demiprime} with a different proof.

\begin{corollary}\label{carnoyau} The kernel of the linear mapping $\gamma: u \mapsto u_{\vert \Gamma}$ from the space $E(\nabla;\, \Omega)$ into $L^2(\Gamma)$ is equal to $H^{1/2}_{00}(\Omega)$, where $E(\nabla;\, \Omega)$ is defined  in \eqref{defEnablaOmega} 
\end{corollary}

\begin{proof} The proof of the inclusion  $H^{1/2}_{00}(\Omega)\subset \mathrm{Ker}\, \gamma$ is the same of that  Theorem \ref{TracesH1demigradH1demiprime}. So let $u\in Ker \, \gamma$. Since $\Delta u \in [H^{3/2}_{0}(\Omega)]'$, from Theorem \ref{IsoDeltaH3/2H1/2} there exists a unique function $w \in H^{1/2}_{00}(\Omega)$ satisfying $\Delta w = \Delta u$. The harmonic function $z = u - w$ belonging to $H^{1/2}(\Omega)$ and equal to zero on the boundary is then identically equal to zero. Thus $u = w\in H^{1/2}_{00}(\Omega) $. \end{proof}

The following result for boundary data in $H^1(\Gamma)$ is well known. Using Theorem \ref{CoruH10DeltaudualH1/2}, we give here an other proof that we find in the litterature.

\begin{theorem}\label{IsogdansH1Gamma}   For any $g\in H^{1}(\Gamma)$, the problem $(\mathscr{L}_D^H)$ has a unique solution $u\in H^{3/2}(\Omega)$. Moreover $\sqrt \varrho \, \nabla^2 u \in {\textit{\textbf L}}^2(\Omega)$ and we have the estimate 
\begin{equation*}\label{estimH3demiH1Gamma1}
\Vert u \Vert_{H^{3/2}(\Omega)} + \Vert\sqrt \varrho\, \nabla^2 u \Vert_{\textit{\textbf L}^2(\Omega)} \leq \ C\,\Vert g \Vert_{H^{1}(\Gamma)}.
\end{equation*}
The solution $u$ satisfies also the following property: for any positive integer $k$
$$
\varrho^{k + 1/2}\nabla^{k + 2}u \in \textit{\textbf L}^2(\Omega).
$$
\end{theorem}

\begin{proof} Let $u\in H^1(\Omega)$ the unique solution of Problem $(\mathscr{L}_D^H)$. As above, by using Theorem \ref{inegharm}, it suffices to prove that $u\in H^{3/2}(\Omega)$. 

Recall that from Ne$\mathrm{\check{c}}$as Property such a solution satisfies 
\begin{equation*}\label{nablaudans L^2Gamma}
\nabla u_{\vert \Gamma} \in \textit{\textbf L}^2(\Gamma), \quad \mathrm{with}\quad  \Vert \nabla u \Vert_ {\textit{\textbf L}^2(\Gamma)} \leq C\, \Vert g \Vert_{H^{1}(\Gamma)}.
\end{equation*}
We then reason as in the proof of Theorem \ref{CoruH10DeltaudualH1/2} (see Step 1), without assuming that $\Omega$ is a Lipschitz hypograph.
For $ \textit{\textbf x} = (\textit{\textbf x}', x_N) \in \Omega$, we set
$$
g_k(\textit{\textbf x}) = \chi(x_N) u (\textit{\textbf x}', \xi_k(\textit{\textbf x}')).
$$
Since $u\in H^2_{\mathrm{loc}}(\Omega)$, there exists a unique solution $u_k \in H^2(\Omega_k)$ satisfying $-\Delta u_k = 0$ in $\Omega_k$ and $ u_k = g_k$ on $\Gamma_k$. Setting now
\begin{equation*}
\widetilde{u_k} = \begin{cases} u_k \quad \; \mathrm{in}\;\; \Omega_k \\
g_k\quad \; \; \textrm{in}\; \; \Omega\setminus \Omega_k.
\end{cases}
\end{equation*}
Clearly,  $ (\widetilde{u_k})_k$ is bounded in $H^1(\Omega)$ and $\widetilde{u_k} \rightharpoonup u$ in $H^1(\Omega)$. We will show that $(\nabla \widetilde{u_k})_k$ is bounded in $\textit{\textbf H}^{\, 1/2}(\Omega)$, which will show that $\nabla u\in \textit{\textbf H}^{\, 1/2}(\Omega)$ and then $u\in  H^{\, 3/2}(\Omega)$.

Let $\textit{\textbf F} \in \boldsymbol{\mathscr{D}}(\Omega_k)$ and $\boldsymbol{\varphi} \in \textit{\textbf H}^{\, 3/2}_0(\Omega_k)$ be such that
\begin{equation*}
\Delta \boldsymbol{\varphi} = \textit{\textbf F}\quad \mathrm{in}\; \Omega_k\quad \mathrm{with}\quad  \Vert \boldsymbol{\varphi}\Vert_{\textit{\textbf H}^{\, 3/2}(\Omega_k)}\leq C \Vert \textit{\textbf F}\, \Vert_{[\textit{\textbf H}^{\, 1/2}(\Omega_k)]'},
\end{equation*}
where $C$ is a positive constant which depends only on the Lipschitz  constant and the Poincar\'e constant of $\Omega$.  By Green's formula we get
$$
\int_{\Omega_{k}}\nabla u_k \cdot \Delta \boldsymbol{\varphi}\, dx = \int_{\Gamma_{k}} \nabla u_k\cdot \partial_{\textit{\textbf n}_k} \boldsymbol{\varphi}\, d\sigma_k.
$$
Hence, using Theorem \ref{CoruH10DeltaudualH1/2} we obtain
$$
\displaystyle \vert \int_{\Omega_{k}}\nabla u_k \cdot \textit{\textbf F}\, \vert \leq \Vert \nabla u_k\Vert_{\textit{\textbf L}^2(\Gamma_k)}\Vert  \partial_{\textit{\textbf n}_k} \boldsymbol{\varphi}\Vert_{\textit{\textbf L}^2(\Gamma_k)} \leq C \Vert \nabla u_k\Vert_{\textit{\textbf L}^2(\Gamma_k)}\Vert\textit{\textbf F}\, \Vert_{[\textit{\textbf H}^{\, 1/2}(\Omega_k)]'}.
$$
But, from classical Ne$\mathrm{\check{c}}$as Property given in Theorem  \ref{NecProt}
$$
\Vert \nabla u_k\Vert_{\textit{\textbf L}^2(\Gamma_k)} \leq C \Vert g_k\Vert_{H^1(\Gamma_k)} \leq C \Vert \nabla u \Vert_ {\textit{\textbf L}^2(\Gamma)}  \leq C \Vert g \Vert_ {H^1(\Gamma)}
$$
As $ \boldsymbol{\mathscr{D}}(\Omega_k)$ is dense in $[\textit{\textbf H}^{\, 1/2}(\Omega_k)]'$, we deduce the following inequality
\begin{equation}\label{inegnablaukH1/2}
\Vert  \nabla u_k\Vert_{\textit{\textbf H}^{\, 1/2}(\Omega_k)}  \leq C \Vert g \Vert_ {H^1(\Gamma)}.
\end{equation}
Observe now that
$$
\Vert  \nabla \widetilde{u_k}\Vert_{\textit{\textbf H}^{\, 1/2}(\Omega)}^2 = \Vert  \nabla \widetilde{u_k}\Vert_{\textit{\textbf L}^{\, 2}(\Omega)}^2 + \int_\Omega\int_\Omega\frac{\vert \nabla \widetilde{u_k}(\textit{\textbf x}) -  \nabla \widetilde{u_k}(\textit{\textbf y})\vert^2}{\vert \textit{\textbf x} - \textit{\textbf y} \vert^{N + 1}}dx\,dy.
$$ 
In order to estimate the last double integral, it suffices thanks to \eqref{inegnablaukH1/2} to estimate the following:
$$
\int_{\Omega_{k}}\int_{\Omega\setminus\Omega_{k}}\frac{\vert \nabla u_k(\textit{\textbf x}) -  \nabla {g_k}(\textit{\textbf y})\vert^2}{\vert \textit{\textbf x} - \textit{\textbf y} \vert^{N + 1}}\, dx\quad \mathrm{and} \quad \int_{\Omega\setminus\Omega_{k}}\int_{\Omega\setminus\Omega_{k}}\frac{\vert \nabla g_k(\textit{\textbf x}) -  \nabla g_k(\textit{\textbf y})\vert^2}{\vert \textit{\textbf x} - \textit{\textbf y} \vert^{N + 1}}\, dx.
$$
But when $\textit{\textbf y} \in \Omega\setminus\Omega_{k}$, we have ${g_k}(\textit{\textbf y}) = u_k(P_k(\textit{\textbf y}) )$, where $P_k(\textit{\textbf y}) =     (\textit{\textbf y}', \xi_k(\textit{\textbf y}'))$. As $\vert \textit{\textbf x} - \textit{\textbf y} \vert \geq \vert \textit{\textbf x} - P_k(\textit{\textbf y}) \vert$, for $\textit{\textbf x}\in \Omega_k$,  and since 
$$
\Omega_k  = \{x\in \Omega;\; \varrho^*(x, \Gamma)  > \frac{1}{k}\},
$$
where $\varrho^*$ is the regularized (signed or not signed) distance to $\Gamma$, which satisfies :
$$
\forall x\in \Omega, \quad C_1 \varrho(x, \Gamma) \leq \varrho^* (x, \Gamma) \leq C_2\varrho(x, \Gamma),
$$
we can verify that
$$
\int_{\Omega_{k}}\int_{\Omega\setminus\Omega_{k}}\frac{\vert \nabla u_k(\textit{\textbf x}) -  \nabla {g_k}(\textit{\textbf y})\vert^2}{\vert \textit{\textbf x} - \textit{\textbf y} \vert^{N + 1}}\, dx \leq \frac{C}{k}\int_{\Omega_{k}}\int_{\Omega_{k}}\frac{\vert \nabla u_k(\textit{\textbf x}) -  \nabla {u_k}(\textit{\textbf y})\vert^2}{\vert \textit{\textbf x} - \textit{\textbf y} \vert^{N + 1}}\, dx.
$$
For the second double integral, we observe that when $(\textit{\textbf x}, \textit{\textbf y} )\in (\Omega\setminus\Omega_{k})\times (\Omega\setminus\Omega_{k})$, we have $\vert \textit{\textbf x} - \textit{\textbf y} \vert \geq C \vert P_k(\textit{\textbf x})- P_k(\textit{\textbf y}) \vert $ since $\xi_k$ is Lipschitzian, with the Lipschitz constant not depending on $k$. So we deduce that
$$
\int_{\Omega\setminus\Omega_{k}}\int_{\Omega\setminus\Omega_{k}}\frac{\vert \nabla g_k(\textit{\textbf x}) -  \nabla g_k(\textit{\textbf y})\vert^2}{\vert \textit{\textbf x} - \textit{\textbf y} \vert^{N + 1}}\, dx \leq \frac{C}{k^2}\int_{\Omega_{k}}\int_{\Omega_{k}}\frac{\vert \nabla u_k(\textit{\textbf x}) -  \nabla {u_k}(\textit{\textbf y})\vert^2}{\vert \textit{\textbf x} - \textit{\textbf y} \vert^{N + 1}}\, dx
$$
and we conclude that $(\Vert\nabla \widetilde{u_k}\Vert_{\textit{\textbf H}^{\, 1/2}(\Omega)})_k $ is bounded. Finally, 
$$
\widetilde{u_k} \rightharpoonup u\quad \mathrm{ in }\; H^{3/2}(\Omega).
$$
\end{proof}

To find solutions in $H^s(\Omega)$ with $1/2 < s < 3/2$, it suffices to use Theorem \ref{ThIsogL2Gamma}, Theorem  \ref{IsogdansH1Gamma}  and an interpolation argument, which leads us to the following result:

\begin{theorem}\label{d02-240118-t1ex}
Let $\Omega$ be a Lipschitz bounded open subset  of $\mathbb{R}^N$ and let $1/2 < s < 3/2$. Then for any $g\in H^{s-1/2}(\Gamma)$, the problem $(\mathscr{L}_D^H)$ has a unique solution $u\in H^s(\Omega)$ with the estimate 
\begin{equation*}\label{d02-240118-e1c}
 \Vert u \Vert_{H^s(\Omega)}\ \leq \ C\, \Vert g \Vert_{H^{s-1/2}(\Gamma)}.
\end{equation*}
\end{theorem}

\begin{proof} Let $S: g \mapsto u_g$ the linear operator defined as follows: for $g\in L^2(\Gamma)$, $u_g \in H^{1/2}(\Omega)$ is the unique harmonic function satisfying $u_g = g$ on $\Gamma$. We know that the operators $S: L^2(\Gamma) \rightarrow  H^{1/2}(\Omega) \cap \mathscr{H}$ and  $S: H^{1}(\Gamma) \rightarrow  H^{3/2}(\Omega) \cap \mathscr{H}$ are linear and continuous. Hence, by interpolation, we deduce that for any $0 < \theta < 1$
$$
S: [H^1(\Gamma), L^2(\Gamma)]_\theta \rightarrow  [H^{3/2}(\Omega) \cap \mathscr{H}, H^{1/2}(\Omega) \cap \mathscr{H}]_\theta 
$$
is continuous. Clearly, we have 
$$
 [H^1(\Gamma), L^2(\Gamma)]_\theta = H^{1-\theta}(\Gamma).
 $$
In order to prove our result, we need to characterize the interpolation space $ [H^{3/2}(\Omega) \cap \mathscr{H}, H^{1/2}(\Omega) \cap \mathscr{H}]_\theta $. However, for all $0< \theta < 1$ we have
	$$
		[ H^1(\Omega) \cap \mathscr{H}, L^2(\Omega)\cap \mathscr{H}]_{\theta} = 
		H^{1-\theta}(\Omega)\cap \mathscr{H},
	$$
	(see \cite{J-K}). In fact, this result can be easily extended as follows: for any positive integer $m$,
	$$
		[ H^m(\Omega) \cap \mathscr{H}, L^2(\Omega)\cap \mathscr{H}]_{\theta} = H^{m(1-\theta)}(\Omega)\cap \mathscr{H}.
	$$
Using this last identity with $m = 2$, we get
$$
H^{3/2}(\Omega) \cap \mathscr{H} = [ H^2(\Omega) \cap \mathscr{H}, L^2(\Omega)\cap \mathscr{H}]_{1/4} \quad\mathrm{and}\quad  H^{1/2}(\Omega) \cap \mathscr{H} = [ H^2(\Omega) \cap \mathscr{H}, L^2(\Omega)\cap \mathscr{H}]_{3/4}.
$$ 
By reiteration theorem (see \cite{Lions}, Theorem 6.1, Chapter 1), we then deduce that
\begin{equation*}
 [ H^{3/2}(\Omega) \cap \mathscr{H},  H^{1/2}(\Omega) \cap \mathscr{H} ]_{\theta} = [ H^2(\Omega) \cap \mathscr{H}, L^2(\Omega)\cap \mathscr{H}]_{(1-\theta)(1/4) + \theta (3/4)} =    H^s(\Omega) \cap \mathscr{H},
\end{equation*}
with $s = 2[1- (1-\theta)(1/4) - \theta (3/4)] = 3/2 - \theta$. That means that
$$ 
[H^{3/2}(\Omega) \cap \mathscr{H}, H^{1/2}(\Omega) \cap \mathscr{H}]_\theta  = H^{3/2- \theta}(\Omega) \cap \mathscr{H}, 
$$
which concludes the proof.
\end{proof}

As a consequence of Theorem \ref{UniciteH1/2} , we have:

\begin{corollary} Let $u \in H^{1/2}(\Omega)$ be a harmonic function in $\Omega$. If for some $ s\in \, ]0, 1 ]$,  we have $u_{\vert \Gamma} \in  H^{s}(\Gamma)$, then $u \in H^{s+1/2}(\Omega)$.
\end{corollary}
\begin{proof} Using Theorem \ref{d02-240118-t1ex}, there exists a unique harmonic function $w \in H^{s+1/2}(\Omega)$ such that $w = u$ on $\Gamma$.
Setting $\chi = u - w$, the harmonic function $\chi$ belongs to $H^{1/2}(\Omega)$ and $\chi = 0 $ on $\Gamma$. Hence $\chi = 0$ in $\Omega$ and $u = w$ in $\Omega$ which means that  $u \in H^{s+1/2}(\Omega)$.
\end{proof}

Using Theorem \ref{d02-240118-t1ex}, we give a second proof of Theorem \ref{isoHsbTheo}, similar to the one given in  \cite{J-K}.

\begin{theorem}
Let $\Omega$ be a Lipschitz bounded open subset of $\mathbb{R}^N$ and \\ $1/2< s<3/2$. Then if $f\in H^{s-2}(\Omega)$, the problem $(\mathscr{L}_D^0)$ has a unique solution $u\in H^s(\Omega)$ with the following estimate
$$
 \Vert u \Vert_{H^s(\Omega)}\ \leq \ C\,\Vert f \Vert_{H^{s-2}(\Omega)}.
$$
\end{theorem}

\begin{proof}
We extend the data $f$ by zero outside $\Omega$ and denote by $\widetilde{f}$. We can see that $\widetilde{f}\in H^{s-2}(\mathbb{R}^N)$ and the   solution  $\mathscr{E}*\widetilde{f}$ of Laplace's equation belongs to $H^{s-2}(\mathbb{R}^N)$ where $\mathscr{E}$ is the fundamental solution of $-\Delta$. The restriction $(\mathscr{E}*\widetilde{f})|_\Omega$, belongs to $H^s(\Omega) $ and its trace $\gamma(\mathscr{E}*\widetilde{f})|_\Omega\in H^{s-1/2}(\Gamma).$ Now we consider the following problem
$$
\Delta w = 0\ \textrm{ in } \Omega\ \ \ \textrm{ and} \ \ \ w = \gamma(\mathscr{E}*\widetilde{f}|_\Omega)\ \textrm{ on }\ \Gamma.
$$
From Theorem \ref{d02-240118-t1ex} the problem above has a unique solution $w\in H^s(\Omega)$ with the corresponding estimate. The solution of Problem $(\mathscr{L}_D^0)$ is then given by
$$
u\ = \ \mathscr{E}*\widetilde{f}|_\Omega-w
$$
and belongs to $H^s(\Omega)$ with the estimate 
$$
  \Vert u\Vert_{H^s(\Omega)}\ \leq \ C\, \Vert f \Vert_{H^{s-2}(\Omega)}.
$$
\end{proof}
 
\subsection{Ne\v{c}as Property, second version}
We will now improve Theorem \ref{CoruH10DeltaudualH1/2}  as follows.

\begin{corollary}\label{d02-110118-th1a}
Let 
\begin{equation*}
u\in H^1(\Omega)\quad \mathrm{ with} \quad \Delta u  \in\ [H^{1/2}(\Omega)]'.
\end{equation*}
i) If $u\in H^1(\Gamma)$, then $\partial_\textit{\textbf n} u\in L^2(\Gamma)$ and we have the following estimate
\begin{equation}\label{estimNeuNecu}
 \Vert\partial_\textit{\textbf n} u\Vert_{ \textit{\textbf L}^2(\Gamma)}\ \leq\ C\Big(\inf_{k\in  \mathbb{R}}\left|\left |u + k \right |\right |_{H^1(\Gamma)} + \Vert\Delta u\Vert_{[H^{1/2}(\Omega)]'}\Big). 
\end{equation}
\noindent ii)  If $\partial_\textit{\textbf n} u\in L^2(\Gamma)$, then $u\in H^1(\Gamma)$ and we have the following estimate
\begin{equation*}\label{d02-100118-e3a}
 \inf_{k\in  \mathbb{R}}\Vert u + k\Vert _{H^1(\Gamma)}  \ \leq\ C\Big(\Vert \partial_\textit{\textbf n} u\Vert_{ \textit{\textbf L}^2(\Gamma)} +  \Vert\Delta u\Vert_{[H^{1/2}(\Omega)]')}\Big).
\end{equation*}
\noindent iii)  If $u\in H^1(\Gamma)$ or $\partial_{\textit{\textbf n}} u\in L^2(\Gamma)$, then $u\in H^{3/2}(\Omega)$.
\end{corollary}

\begin{proof} i) Let $u_0 \in H^{3/2}(\Omega)$ be the solution given by Theorem \ref{IsogdansH1Gamma} and satisfying
$$
\Delta u_0 = 0\ \textrm{ in } \Omega\ \ \ \textrm{ and} \ \ \ u_0 = u\ \textrm{ on }\ \Gamma.
$$
Using Theorem \ref{NecProt}, we know that $\partial_\textit{\textbf n} u_0 \in L^2(\Gamma)$ with the estimate
\begin{equation*}\label{estimNeuNecu0}
\Vert\partial_\textit{\textbf n} u_0\Vert_{ \textit{\textbf L}^2(\Gamma)}\ \leq\ C(\Omega) \inf_{k\in  \mathbb{R}}\left|\left |u_0 + k \right |\right |_{H^1(\Gamma)} .
\end{equation*} 
Since the function $v = u - u_0\in H^1_0(\Omega)$ and its Laplacian belongs to $[H^{1/2}(\Omega)]'$, we conclude that  $\partial_\textit{\textbf n} v \in L^2(\Gamma)$ by using Theorem \ref{CoruH10DeltaudualH1/2} with
\begin{equation*}\label{estimNeuNecv}
\Vert\partial_\textit{\textbf n} v\Vert_{ \textit{\textbf L}^2(\Gamma)}\ \leq\ C(\Omega)\Vert\Delta v\Vert_{[H^{1/2}(\Omega)]'}.
\end{equation*}
So $\partial_{\textit{\textbf n}} u \in L^2(\Gamma)$ with the estimate \eqref{estimNeuNecu}. \smallskip

\noindent  ii) Let $z_0 \in H^{1}(\Omega)$ be the solution satisfying
$$
\Delta z_0 = \Delta u\ \textrm{ in } \Omega\ \ \ \textrm{ and} \ \ \ z_0 = 0\ \textrm{ on }\ \Gamma.
$$
Point i) above implies that  $\partial_\textit{\textbf n} z_0 \in L^2(\Gamma)$. Setting $v = u - z_0$. Since by assumption $\partial_\textit{\textbf n} u \in L^2(\Gamma)$, we have also $\partial_\textit{\textbf n} v \in L^2(\Gamma)$. Now, as the harmonic function $v $ belongs to $H^{1}(\Omega)$, we deduce by using again Theorem \ref{NecProt} that its trace satisfies $v\in H^1(\Gamma)$.  So we have also $u\in H^1(\Gamma)$. To prove the corresponding estimate, we proceed as above.\smallskip

\noindent  iii)  Suppose that $u\in H^1(\Gamma)$. We know that there exist $u_0 \in  H^{3/2}_{0}(\Omega)$ satisfying $\Delta u_0 = \Delta u$ in $\Omega$ and $u_1 \in  H^{3/2}(\Omega)$ such that $\Delta u_1 = 0$ in $\Omega$ with $u_1= u$ on $\Gamma$.  Setting $z = u_0 + u_1$ which is in $H^{3/2}(\Omega)$, the function $u - z$ is harmonic and belongs to $H^{1}_0(\Omega)$. Hence $u = z$ and $u\in H^{3/2}(\Omega)$.

Now, suppose that $\partial_\textit{\textbf n} u \in L^2(\Gamma)$. Using Point ii) above we deduce that $u\in H^1(\Gamma)$ and then $u\in H^{3/2}(\Omega)$.
\end{proof}

\begin{remark}\upshape  The condition $\Delta u  \in\ [H^{1/2}(\Omega)]'$ is sufficient and maybe not necessary. However if we replace it by the condition $\sqrt\varrho \Delta u \in L^2(\Omega)$, then the conclusion of Point i) above no longer applies, as we can see by taking the function  $v$ given in Step 2 of Proposition \ref{ConterexampleH1demitrace1b}.

\end{remark}

 \appendix
 \section{Open problems} 

In this work we mainly focused on Laplace equation with Dirichlet boundary condition. And our analysis is mainly based  on $L^2$ theory. However, the theory and the results developed in this paper can be extended to other types of elliptic equations or systems and also to different boundary conditions, as well as the $L^p$ framework. We give below some examples of models for which it would be interesting, we believe, to study the questions of optimal regularity when the domain is only Lipschitz.  In many applications, such as fluid mechanics or electromagnetism, the domain $\Omega$ is indeed not very regular.

\subsection{Laplace equation with Dirichlet boundary condition}\label{Dir}\smallskip

\noindent{\bf 1. $L^p$-{\bf theory}}.  In the 80's, Ne\v{c}as posed the question of solving the problem $(\mathscr{L}_D)$ with the homogeneous boundary condition $g = 0$ on Lipschitz domains, when the RHS $f\in W^{-1,\, p}(\Omega)$. The answer to this question is partially given in the paper of Jerison and  Kenig \cite{J-K}:\medskip

\noindent{\bf Negative results.}  If $N \geq 3$, then for any $p > 3$ (resp. $ p > 4$ if $N = 2$), there is a Lipschitz domain $\Omega$ and $f\in \mathscr{C}^{\infty}(\overline{\Omega})$ such that the solution $u$ of Problem $(\mathscr{L}_D)$ with the homogeneous boundary condition $g = 0$ does not belong to $W^{1,\, p}(\Omega)$.\medskip

\noindent {\bf Positive results.} There exists $q > 3$ when $N \geq 3$ (resp. $q > 4$ when $N = 2$) such that if $q' < p < q$, then the problem $(\mathscr{L}_D)$ has a unique solution $u\in W^{1,\, p}(\Omega)$ satisfying the estimate
$$
\Vert u \Vert_{W^{1,\, p}(\Omega)}\leq C \Vert f \Vert_{W^{-1,\, p}(\Omega)}.
$$
Moreover, if $\Omega$ is $\mathscr{C}^{1}$, we can take $q = \infty$ as stated above. \medskip

In our opinion, the results above for $N \geq 3$, concerning the solvability in $W^{1, p}(\Omega)$, are available only for $N = 3$. This leads us to the following conjecture:\medskip

\noindent{\bf Conjecture 1.  Solvability in }$ W^{1, p}_0(\Omega)$. \smallskip

\noindent {\bf  i)} Let $\Omega$ be a bounded Lipschitz  domain of $\mathbb{R}^N$ with $N \geq 2$. Then, there exists $p_0(\Omega) < 2N/(N+1)$, depending only on $\Omega$,   such that if $p_0(\Omega)  < p <  [p_0(\Omega)]'$ and $f \in W^{-1,\, p}(\Omega)$ the problem $(\mathscr{L}_D^0)$ has a unique solution $u\in W^{1,\, p}(\Omega)$ satisfying the estimate
$$
\Vert u \Vert_{W^{1,\, p}(\Omega)}\leq C \Vert f \Vert_{W^{-1,\, p}(\Omega)}.
$$
\noindent {\bf  ii)} If $N \geq 2$, then for any $p > 2N/(N-1)$, there is a Lipschitz domain $\Omega$ and $f\in W^{-1,\, p}(\Omega)$ such that the solution $u$ of Problem $(\mathscr{L}_D^0)$ does not belong to $W^{1,\, p}(\Omega)$.\medskip
 
\noindent{\bf Conjecture 2.  Solvability in }$ L^{p}_{s,0}(\Omega)$, {\bf with} $ 1/p \leq s \leq 1+1/p$. \medskip

We know that the following operators are  isomorphisms: \\
i)  $ \Delta : H^{s}_0(\Omega) \rightarrow H^{s-2}(\Omega)$ for any $s$ satisfying $1/2 < s < 3/2$,\\
ii)  $\Delta : H^{3/2}_0(\Omega)  \rightarrow [H^{1/2}_{00}(\Omega)]'$ and  $\Delta : H^{1/2}_{00}(\Omega)\rightarrow H^{3/2}_0(\Omega)]'$.\medskip

We claim that for any $p_0 < p < p_0' $, with the same exponent $p_0$ as above,  the following operators are  isomorphisms: \\
i) $\Delta : L^{p}_{s,0}(\Omega)\rightarrow L^{s - 2,p}(\Omega)$ for any $s$ satisfying $1/p < s < 1+1/p$,\\
ii) from $ \Delta :  L^{p}_{1+1/p, 0}(\Omega) \rightarrow [L^{ p'}_{1/p', 0}(\Omega)]'$ and  $  \Delta : L^{p'}_{1/p', 0}(\Omega) \rightarrow[L^{1+1/p, p}_{0}(\Omega)]'$,
where the above spaces are defined in Section \ref{secuniqueness} and
$$
  L^{ p'}_{1/p', 0}(\Omega) = \{v \in L^{1/p', p'}(\Omega); \; \frac{v}{\varrho^{1/p'}} \in L^{p'}(\Omega)\}.
$$
\noindent{\bf 2.}  {\bf Very weak solutions}. The concept of very weak solutions developed in Lions-Magenes \cite{Lions} and also in Ne\v{c}as' book is quite appropriate when the domain $\Omega$ is sufficiently regular. In this case, using a duality argument, it is possible to solve the problem $(\mathscr{L}_D^H)$ in $H^{-s}(\Omega) $ when the boundary data $g$ belongs to $H^{-s-1/2}(\Gamma)$, with $s \ge 0$ which depends on the regularity of $\Omega$. The case $s = 0$ for  Lipschitz domain is treated in \cite{AAB} with some additional assumptions on $g$. \medskip

\noindent{{\bf Problem 1.} So, what is the maximal value of $s$ to find solutions in $H^{-s}(\Omega)$ (or in some weighted $L^2$ Sobolev space with respect to the distance to the boundary)? Similar question holds for the non homogeneous problem $(\mathscr{L}_D^0)$. Particularly, do we have an "optimal  choice" for the data $f$?  \medskip

\noindent{\bf 3.} {\bf The} $\mathrm{div(a\, } \mathrm{grad})$ {\bf operator}.\medskip

\noindent{\bf Problem 2.} What happens now if we replace the Laplacian by the operator 
$$
\mathrm{div(a\, } \mathrm{grad}), 
$$
with the function $a$ \\
i) not necessarily continuous, but possibly $VMO$ (Vanishing Mean Oscillation) and satisfying the inequalities $0 < a_* \leq a \leq a^*$, where $a_*$ and $a^*$ are constants?\\
ii) or equal to $\varrho^\alpha$ with $0 < \alpha \leq 1$, where $\varrho$ is the distance to the boundary of $\Omega$?  \medskip

\noindent{\bf 4. Optimal regularity in $\mathscr{C}^{1, 1}$ domains. } In Section \ref{Inhomogeneous Problem} we have seen that in the case of a Lipschitz domain $\Omega$, the optimal regularity $H^{3/2}(\Omega)$ for the solution of Problem $(\mathscr{L}_D^0)$ is obtained when the RHS $f$ is in the dual space $[H^{1/2}_{00}(\Omega)]'$. \medskip

\noindent{\bf Problem 3.}  When the domain is of class $\mathscr{C}^{1,1}$, do we have the regularity $H^{5/2}(\Omega)$ if $f$ belongs to $H^{1/2}_{00}(\Omega)$? \medskip

\subsection{Laplace equation with Neumann boundary condition}\label{Neu}\smallskip

In \cite{Fabes3} the authors showed that for any  bounded Lipschitz domain $\Omega$ and any $f\in [L^{p'}_{s+1/p'}(\Omega)]' $ and $h\in [B^{p'}_s(\Gamma)]'$ satisfying the compatibility condition $\langle f, 1\rangle_\Omega = \langle h,  1\rangle_\Gamma$, the problem
$$
(\mathscr{L}_N)\ \ \ \  \Delta v = f \quad \ \mbox{in}\ \Omega \quad
\mbox{and } \quad \partial_{\textit{\textbf n}} v= h \ \ \mbox{on }\Gamma,
$$
has a unique (modulo additive constants) solution $v \in L^{p}_{1-s+1/p}(\Omega)$ if the pair $s, p$ satisfies some conditions similar to those given in \cite{J-K} for the Dirichlet case. As in this latter case, these conditions must be modified in the same spirit as in Section \ref{secuniqueness}. Furthermore, the spaces considered for the data $f$, which are not subspaces of $\mathscr{D}'(\Omega)$, are not appropriate. On the other hand,  the above Neumann condition, as defined in \cite{Fabes3}, is related to the linear and continuous form $f$ and therefore varies with $f$. For simplicity, let us consider the case $p = 2$ and recall that for any vector field $\textit{\textbf F}\in\textit{\textbf H}\,^{-s+1/2}(\Omega)$, the authors define its normal component as follows: 
\begin{equation}\label{d02-190118-e2}
\forall\mu\in H^s(\Gamma),\ \ \ \left<\, \textit{\textbf F}\cdot\textit{\textbf n}_{f},\mu\,\right>_{H^{-s}(\Gamma)\times H^{s}(\Gamma)}\ =\ \left<\,f,\varphi\,\right>_\Omega+\left<\,\textit{\textbf F},\nabla\varphi\,\right>_\Omega,
\end{equation}
where $f\in \left[H^{s+1/2}(\Omega)\right]'$  is any extension of  $\textrm{div}\,\textit{\textbf F}\in H^{-s-1/2}(\Omega)$ and $\varphi \in H^{s+1/2}(\Omega)$ is an extension (in the trace sense) of $\mu$.

With this notion of normal component, which differs from the classical one, for which
an additional condition on divergence is required, the authors proved in \cite{Fabes3} that for any $0 < s < 1$, $f\in [H^{s+1/2}(\Omega)]' $ and $h\in H^{-s}(\Gamma)$ satisfying the compatibility condition $\langle f, 1\rangle_\Omega = \langle h,  1\rangle_\Gamma $ the problem: Find $v\in H^{-s + 3/2}(\Omega)$ satisfying
\begin{equation}\label{forvarNeu}
\forall \varphi \in H^{s+1/2}(\Omega), \quad \langle  \nabla v, \nabla \varphi \rangle_\Omega =  -\langle  f, \varphi\rangle_\Omega + \langle h, \varphi\rangle_\Gamma
\end{equation}
admits a unique solution, up to an additive constant. This solution $v$ satisfies the equation $\Delta v = f_{\vert \Omega}$ in $\Omega$ and $\nabla v\cdot\textit{\textbf n}_f = h \ \ \mbox{on }\Gamma$ in the sense of \eqref{d02-190118-e2}.

Let us consider now the following variational formulation:  Find $v\in H^{-s + 3/2}(\Omega)$ satisfying
\begin{equation}\label{forvarNeuF}
\forall \varphi \in H^{s+1/2}(\Omega), \quad \langle  \nabla v, \nabla \varphi \rangle_\Omega =  \langle  \textit{\textbf F}, \nabla\varphi\rangle_\Omega + \langle h, \varphi\rangle_\Gamma
\end{equation}
where $\textit{\textbf F}\in\textit{\textbf H}\,^{-s+1/2}(\Omega)$ and $h\in H^{-s}(\Gamma)$ satisfies $\langle h, 1\rangle = 0$. As above, \eqref{forvarNeuF}  admits a unique solution, up to an additive constant. Moreover \eqref{forvarNeuF} is equivalent to the following Neumann problem:  
$$
(\mathscr{L}'_N)\ \ \ \  \Delta v = \mathrm{div}\, \textit{\textbf F} \quad \ \mbox{in}\ \Omega \quad
\mbox{and } \quad (\nabla v -  \textit{\textbf F}) \cdot\textit{\textbf n} = h \ \ \mbox{on }\Gamma,
$$ 
Note that since $\textit{\textbf F}$ and $\nabla v$ are not sufficiently regular, $\textit{\textbf F} \cdot\textit{\textbf n}$ and $\nabla v\cdot\textit{\textbf n}$ have no sense contrarily to the difference $ (\nabla v -  \textit{\textbf F}) \cdot\textit{\textbf n}$.

And we can easily prove that for any extension $f$ of  $\textrm{div}\,\textit{\textbf F}\in H^{-s-1/2}(\Omega)$, we have the following relation:
$$
 \nabla v\cdot\textit{\textbf n}_f = h + \textit{\textbf F}\cdot\textit{\textbf n}_{f}  \ \ \mbox{on }\Gamma.
 $$ 
 Beside, for any $f\in [H^{s+1/2}(\Omega)]'$ there exist $\textit{\textbf F}\in \textit{\textbf H}\,^{-s+1/2}(\Omega)$ and $g\in H^{-s}(\Gamma)$ such that 
 \begin{equation*}\label{Decompf}
\forall \varphi \in H^{s+1/2}(\Omega), \quad \langle  f, \varphi \rangle_\Omega =  \langle  \textit{\textbf F}, \nabla\varphi\rangle_\Omega + \langle g, \varphi\rangle_\Gamma
\end{equation*}
To prove the above property, we can use the following theorem.\medskip

\noindent{\bf Theorem A.1.}
{\it Let $f\in  \left[H^{s+1/2}(\Omega)\right]'$. Then, there exists a unique $u\in H^{-s+3/2}(\mathbb{R}^{N})$ satisfying 
$$
u-\Delta u=0\quad  in \; \Omega'\quad  with \quad  \Omega'= \mathbb{R}^N\setminus\overline{\Omega}
$$ 
and a unique $h\in H^{-s}(\Gamma)$ such that
\begin{eqnarray*}\label{numlemmadual2}
	\forall\varphi\in H^{s+1/2}(\Omega),\quad \left\langle f,\varphi\right\rangle_{\Omega} = \int_{\Omega}u\varphi \, dx + \langle \nabla u, \nabla \varphi\rangle_\Omega +\left\langle h,\varphi\right\rangle_{\Gamma}.
\end{eqnarray*}
In particular
$$
f_{|_{\Omega}} = u-\Delta u\quad\mbox{in}\;\Omega
$$
and \begin{eqnarray*}
	\left\langle h, \varphi\right\rangle_{\Gamma}=-\left\langle \partial_\textit{\textbf n} u , \varphi\right\rangle_{\partial\Omega'}.
\end{eqnarray*}
}

The proof of this theorem will be given in a forthcoming paper.\medskip

Return now to the study of the Neumann problem $ (\mathscr{L}_N)$. The previous observations show that the two approaches, \eqref{forvarNeu} and \eqref{forvarNeuF} are equivalent. However the second one seems more convenient concerning the definition and the sense of the normal component of non regular vector fields. Another reason for the interest of this choice is given in the next theorem.\medskip

\noindent{\bf Theorem A.2.}
{\it   For any 
\begin{equation*}\label{HypdivFFcdotn}
\textit{\textbf F}\in\textit{\textbf H}\,^{1/2}(\Omega)\quad with \quad div \,\textit{\textbf F}\in  [H^{1/2}(\Omega)]' \quad and \quad\textit{\textbf F}\cdot \textit{\textbf n}\in L^2(\Gamma)
\end{equation*} 
and for any $h\in L^2(\Gamma)$ satisfying the compatibility condition $\langle h,  1\rangle_\Gamma = 0$, there exists a unique $u\in H^{3/2}(\Omega) $, up an additive constant, such that
\begin{equation*}\label{NeumanGeneralFc} 
\Delta u = \mathrm{div}\,\textit{\textbf F}\quad in \; \Omega \quad and \quad  \nabla u\cdot \textit{\textbf n} = \textit{\textbf F}\cdot \textit{\textbf n} +  h\quad  on\;  \Gamma .
 \end{equation*}
}
\begin{proof} Clearly we have the existence of solution $u\in H^1(\Omega)$ and also in $H^{3/2-s}(\Omega)$ for any $s < 1$. The regularity $H^{3/2}(\Omega) $ is an immediate consequence of Corollary \ref{d02-110118-th1a}.
\end{proof} 

\noindent{\bf Conjecture 3.  Solvability in }$ L^{s, p}(\Omega)$, {\bf with} $ 1/p < s < 1+1/p$. \medskip

We claim that for any $p_0 < p < p_0' $, with the same exponent $p_0$ as in Appendix A.1,  and for any $\textit{\textbf F}\in L^{p}_{-s+1/p}(\Omega) $ and $h\in B^{p}_{-s}(\Gamma)$, with $\langle h, 1\rangle = 0$, Problem $ (\mathscr{L}'_N)$ admits   a unique solution $u\in L^{p}_{-s+ 1+ 1/p}(\Omega)$, up to an additive constant.\medskip

\noindent{\bf Problem 4.} What happens for the extreme values $s = 1/p$ and $s = 1+1/p$?

\subsection{Biharmonic Problem}\label{Biha} 
As mentioned above, some ideas and arguments used for Laplace equation may be appropriate for the investigation of the Biharmonic problem, with different boundary conditions. \medskip

\noindent{\bf 1. Biharmonic problem with Dirichlet boundary conditions.} Let us first consider the case of Dirichlet boundary conditions:
\begin{equation*}
 (\mathscr{B}_D)\quad
 \begin{cases}
 \Delta^2 u\ =\ f \quad\textrm{ in }\ \Omega,\\
   u\ =\ g_0 \quad\textrm{ on }\ \Gamma,\\
  \partial_{\textit{\textbf n}} u\ =\ g_1 \quad\textrm{ on }\ \Gamma.
 \end{cases}
 \end{equation*}
 For $\Omega$ of class $\mathscr{C}^{0,1}$, it is proved in \cite{Dahlberg} that for $f= 0$ and any pair
\begin{equation*}\label{d06-140318-e1}
 (g_0, g_1) \in H^1(\Gamma)\times  L^2(\Gamma),
\end{equation*}
there  exists a unique solution $u\in H^{3/2}(\Omega)$ to  Problem $(\mathscr{B}_D)$ satisfying $\sqrt\varrho\,\nabla^2 u\in \textit{\textbf L}^2(\Omega)$.  \medskip

\noindent{{\bf Problem 5.} It would be interesting to see if, as for Laplace equation, one could obtain the $H^{5/2}$-regularity and for which choice of data one can have this result. \medskip

Concerning the problem $(\mathscr{B}_D)$ with $g_0 = g_1= 0$, Adolfsson and Pipher \cite{AP} have established the existence of a solution in $H^{2+s}(\Omega)$  if $f\in H^s(\Omega)$ and $-1/2 < s < 1/2$. They also showed in the same paper a similar result when $f = 0$ and $(g_0, g_1)$ belong to some Whitney array spaces denoted by $WA^2_{3/2+s}(\Gamma)$ following the characterization:
$$
g_0 \in H^1(\Gamma),\; g_1\in H^{1/2+s}(\Gamma) \quad \mathrm{and}\quad \nabla_\tau g_0 + g_1\textit{\textbf n} \in \textit{\textbf H}\,^{1/2+s}(\Gamma).
$$

\noindent{{\bf Problem 6.}  Is it possible to obtain the optimal regularity $H^{5/2}(\Omega)$, for appropriate data and corresponding to the case $s = 1/2$?\\

\noindent{\bf 2. Biharmonic problem with Navier boundary conditions.} The case of Navier boundary conditions 
\begin{equation*}
 (\mathscr{B}_{Na})\quad
 \begin{cases}
 \Delta^2 u\ =\ f \quad\textrm{ in }\ \Omega,\\
   u\ =\ g_0 \quad\textrm{ on }\ \Gamma,\\
 \Delta u =\ g_1 \quad\textrm{ on }\ \Gamma,
 \end{cases}
 \end{equation*}
 is completely open and particularly interesting. \medskip
 
 \noindent{{\bf Problem 7.}  Which assumptions on $f, g_0$ and $g_1$ are appropriate to get solution in $H^s(\Omega$? And for which corresponding values of $s$?\medskip
 
\noindent{\bf 3. Biharmonic problem with "Neumann"  boundary conditions.} For a bounded Lipschitz domain $\Omega$, with connected boundary, Verchota investigated in \cite{Verchota2} the following Neumann problem  
\begin{equation*}
 (\mathscr{B}_{Ne})\quad
 \begin{cases}
 \Delta^2 v\ =\ f \quad\textrm{ in }\ \Omega,\\\\
   \nu \Delta v + (1 - \nu)\frac{\partial^2 v}{\partial \textit{\textbf n}^2}\ = \ \Lambda_0 \quad\textrm{ on }\ \Gamma,\\\\

 \frac{\partial \Delta v } {\partial \textit{\textbf n}} + \frac{1 - \nu}{2}\frac{\partial}{\partial \boldsymbol{\tau}_{ij}}(\frac{\partial^2 v}{\partial \textit{\textbf n}\, \partial\boldsymbol{\tau}_{ij}}) \ = \ \Lambda_1 \quad\textrm{ on }\ \Gamma,
 \end{cases}
 \end{equation*}
 where $\nu$ is a constant known as the Poisson ratio, whose value corresponds to a particular physical situation. Here $\boldsymbol{\tau}_{ij} = n_i \textit{\textbf e}_j - n_j\textit{\textbf e}_i $ is an orthogonal vector field to the outward normal $\textit{\textbf n}$. He showed that if $-1/(N - 1) \leq \nu < 1$, $2- \varepsilon < p  < 2 + \varepsilon$ for some $\varepsilon > 0$ and
 $$
 f= 0, \quad \Lambda_0 \in L^p(\Gamma) \quad \mathrm{and}\quad  \Lambda_1 \in W^{-1, p}(\Gamma),
 $$
 then there exist solutions  to the Neumann problem satisfying in addition the estimate
 $$
 \Vert \nabla^2 u\Vert_{L^p(\Gamma)} \leq C (\Vert \Lambda_0  u \Vert_{L^p(\Gamma)} + \Vert \Lambda_1  u \Vert_{W^{-1, p}(\Gamma)}).
 $$
However, it is not specified in which Sobolev space belong the solutions. \medskip

\noindent{{\bf Problem 8.}  It would therefore be interesting on one hand to give more details on the solutions and, on the other hand, to study the properties of the  Steklov-Poincar\'e operator (Dirichlet-to-Neumann) corresponding to the Bi-Laplacian.

\subsection{Stokes Problem, Elasticity Equations}\label{SE}\smallskip 

One of the first works on Stokes system in Lipschitz domains was done by Fabes, Kenig and Verchota \cite{FKV}. They established the existence of a solution for the homogeneous problem in the case of boundary Dirichlet condition in $L^2(\Gamma)$ or $H^1(\Gamma)$ (see also the papers \cite{BS},  \cite{Shen}). Stokes operator with Neumann  boundary conditions is studied in \cite{MMW} (see also the book \cite{MW}).  \medskip

\noindent{{\bf Problem 9.}  Can the obtained results for the Laplacian in the present work be extended to the Stokes operator?  What happens for Navier, Navier-type or pressure boundary conditions?  \medskip

\noindent{{\bf Problem 10.} What about the elasticity equations? 

\medskip

\noindent{\small{\textit{Acknowledgments.} 
We are deeply indebted towards Professor  Martin Costabel, Professor David Jerison, Professor  Carlos Kenig and Professor S\'ebastien Tordeux.
Our exchanges have been extremely  fruitful for us.
Thanks a lot for them.
}
\medskip

\noindent{\bf Data Availability.} Data sharing not applicable to this article as no datasets were generated or analysed during
the current study.


\begin{thebibliography}{10}


\bibitem{Ada}
\textsc{Adams, R.A}: Sobolev spaces, Academic Press, New-York, (1978)

\bibitem{AP}
 \textsc{Adolfsson, V., Pipher, J.}: The Inhomogeneous Dirichlet Problem for $\Delta^2$ in Lipschitz Domains. Journal of Functional Analysis {\bf 159}, 137--190 (1998)

\bibitem{AAB} 
\textsc{Aib\`eche, A., Amrouche, C., Bahouli, B.}: Traces Characterizations for Sobolev Spaces on Lipschitz Domains of $\mathbb{R}^2$.  Comptes Rendus Math\'ematiques, Acad\'emie des Sciences, Paris, {\bf 361} (2023), 587--597

\bibitem{AG} 
 \textsc{Amrouche, C., Girault, V.}: Decomposition of vector spaces and application to
the Stokes problem in arbitrary dimension. Czechoslovak
Math. Journal., {\bf 119}(44), 109--140 (1994) 

\bibitem{AM}
\textsc{Amrouche, C. and Moussaoui M}. Maximal regularity of Dirichlet problem  for the Laplacian in Lipschitz domains, ArXiv (2025), http://arxiv.org/abs/2509.08543\smallskip

\bibitem{ARB}
\textsc{Amrouche, C. and Rodr\'iguez-Bellido, M.A}. Stationary Stokes, Oseen and Navier-Stokes Equations with Singular Data, \emph{Arch. Rational. Mech. Anal.}, {\bf 199}, (2011), 597--651\smallskip

\bibitem{ACK} 
\textsc{Asekritova, I, Cobos, F., Kruglyak, N.} Interpolation of closed subspaces and invertibility of operaors. Zeit. Anal. Anwend. {\bf  34} (2015), no. 1, 1?15. 


\bibitem{BBX} 
\textsc{Bacuta, C., Bramble, J.H and Xu, J.} Regularity estimates for elliptic boundary value problems in Besov spaces. Mathematics of Computation. Vol. {\bf 72}, $N^0$ 244, pp. 1577--595

\bibitem{BL} 
\textsc{Bergh, J. and Lofstr$\mathrm{\ddot{o}}$m, J.}: Interpolation spaces, an introduction, Springer-Verlag, 1976

\bibitem{BS}  
\textsc{Brown, R.M, Shen, Z.}: Estimates for Stokes Operator in Lipschitz domains. Indiana Univ. Math. Journal. $\boldsymbol{ 44-4}$, 1183--1206 (1995) 


\bibitem{Cos}
 \textsc{Costabel, M.}: Boundary integral operators on Lipschitz domains: elementary results. SIAM J. Math. Anal., Vol. {\bf 19}, no. 3, 613--626 (1988) 
 
\bibitem{Dahl77} 
\textsc{Dahlberg, B.E.J}: Estimates for harmonic measure. Archive Rational Mech. Anal.  {\bf 65}, 275--283 (1977) 

\bibitem{Dahl} 
\textsc{Dahlberg, B.E.J}: Weighted norm inequalities for the Lusin area integral and the nontangential  maximal function for functions harmonic  in a Lipschitz domain. Studia Math. {\bf LXVII},  297--394 (1980)

\bibitem{Dahlberg}
\textsc{Dahlberg, B.E.J, Kenig, C.E,  Pipher, J,  Verchota, G.C}:  Area integral estimates for higher order elliptic equations and systems. Ann. Inst. Fourier, {\bf 47}, No. 5, 1425--1461 (1997)


\bibitem{D}  
\textsc{Dauge, M.}: Elliptic boundary value problems on corner domains. Smoothness and asymptotics of solutions. Lecture Notes in Mathematics, 1341. Springer-Verlag, Berlin, (1988). viii+259 pp.

\bibitem{FJ}
\textsc{Fabes, E.B, Jodeit, M.Jr}: On the Spectra of a Hardy Kernel, Journal of Functional Analysis {\bf 21}, 187--194 (1976)


\bibitem{Fabes2}
\textsc{Fabes, E.B, Jodeit, M.Jr, Lewis, J.E}: Double layer potentials for domains with corners and edges. Indiana Univ. Math., {\bf 26}, No. 2,  95--114 (1977)

\bibitem{Fabes1}
\textsc{Fabes, E.B., Jodeit, M.Jr., Rivi\`ere, N.M}: Potential techniques for boundary value problems on $C^1$ domains. Acta Math., {\bf 141}, No. 2, 165--186 (1978)

\bibitem{FKV}
\textsc{Fabes, E.B,  Kenig, C.E, Verchota, G.C}: The Dirichlet problem for the Stokes system on Lipschitz domains. Duke Math. Journal, {\bf 57-3}, 769--793 (1998)

\bibitem{Fabes3}
\textsc{Fabes, E.B,  Mendez, O.,  Mitrea, M.}: Boundary layers on Sobolev-Besov spaces and Poisson's equation for the Laplacian in Lipschitz domains. Journal of Functional Analysis, {\bf 159},  323--368 (1998)

\bibitem{GK} 
\textsc{Geymonat, G., Krasucki, F.}: On the existence of the Airy function in Lipschitz domains.
Application to the traces of $H^2(\Omega)$. C. R. Acad. Sci. Paris Ser. $\boldsymbol{ I\,330}$, 355--360 (2000)

\bibitem{Gri} 
 \textsc{Grisvard, P.}: Elliptic Problems in Nonsmooth Domains, Pitman, Boston, (1985).

\bibitem{Gri1} 
\textsc{Grisvard, P.}:  Alternative de Fredholm relative au probl\`eme de Dirichlet dans un polygone ou un poly\`edre. Bolletino U.M.I, {\bf (4)}-5, 132--164 (1972)

\bibitem{Gri2} 
\textsc{Grisvard, P.}: Alternative de Fredholm relative au probl\`eme de Dirichlet dans un poly\`edre.  Ann. Sc. Norm. Sup. Pisa, {\bf (2)}-3, 359--388 (1975)
 
\bibitem{IK} 
\textsc{Ivanov, S, Kalton, N.}: Interpolation of subspaces and applications to exponential bases in Sobolev spaces. Algebra i Analiz {\bf 13}, (2001), n 2, 93--115
 
\bibitem{Jer1}
\textsc{Jerison, D., Kenig, C.E.}: The Neumann problem on Lipschitz domains.  Bull. Amer. Math. Soc. {\bf 4}, No. 2,  203--207 (1981)    


\bibitem{Jer2}
\textsc{Jerison, D., Kenig, C.E.}:  The Dirichlet problem in non smooth domains. Ann. of Math. {\bf 113}, No. 2, 367--382 (1981)   


\bibitem{J-K}
\textsc{Jerison, D., Kenig, C.E.}: The Inhomogeneous Dirichlet Problem in Lipschitz domains. Journal of Functional Analysis {\bf 130}, 161--219 (1995)

\bibitem{Lions}
\textsc{Lions, J.L.,  Magenes, E.}: Probl\`{e}mes aux limites non homog\`{e}nes et applications, Vol. 1, Dunod,  Paris, (1969)

\bibitem{MMS} 
\textsc{Maz'ya, V.,  Mitrea, M., Shaposhnikova, T.}: The Dirichlet problem in Lipschitz domains for higher order elliptic systems with rough coefficients. J. Anal. Math., {\bf 110}, 167--239  (2010)

\bibitem{McL}
\textsc{McLean, W.}: Strongly elliptic systems and boundary integral equations. University Press. Cambridge, (2000)  


\bibitem{Mik}
 \textsc{Mikhailov, S.E.}: About traces, extensions and co-normal derivative operators on Lipschitz domains. Integral Methods in Science and Engineering: Techniques and applications, Constanda C., Popenko S. (Eds): Birk$\mathrm{\ddot{a}}$user, Boston,  151--162 (2007)

\bibitem{Mitrea2}
\textsc{Mitrea, I., Mitrea, M.}: Multi-layer potentials and boundary problems for higher-order elliptic systems in Lipschitz domains. Lecture Notes in Mathematics {\bf 2063}, Springer. New York, (2013) 

\bibitem{Mitrea1}
\textsc{ Mitrea, D., Mitrea, M., Taylor, M.}: Layer potentials, the Hodge Laplacian, and global boundary problems in non smooth Riemannian manifolds.  Memoirs of the American Mathematical Society, {\bf 713}, American Mathematical Society. Providence, (2001) 

\bibitem{MMW} 
 \textsc{Mitrea, D.,  Monniaux,  S.,  Wrigh, M.}: The Stokes operator with Neumann boundary conditions in Lipschitz domains. Journal of Mathematical Sciences, Vol. {\bf 176}, No. 3,  409--457 (2011)
 

\bibitem{MW} 
\textsc{Mitrea, D.,  Wright, M.}: Boundary value problems for the Stokes system in arbitrary Lipschitz domains. Ast\'erisque No.{\bf 344} (2012)

\bibitem{Necas}
\textsc{Ne\v{c}as, J.}: Direct methods in the theory of elliptic equations. Springer. New York, (2012) 

\bibitem{Shen}
\textsc{Shen, Z.}: A note on the Dirichlet problem for the Stokes system in Lipschitz domains. Proc. AMS, {\bf 123-3},  801--811 (1995)

\bibitem{Stein}
\textsc{Stein, E.M.}: Singular Integrals and Differentiability Properties of Functions, Princeton University Press,  Princeton (1970)

\bibitem{Tar}
\textsc{Tartar, L}. An introduction to Sobolev spaces and interpolation spaces.  Lecture Notes of the Unione Matematica Italiana, 3. Springer, Berlin; UMI, Bologna, 2007

\bibitem{Temam}
\textsc{Temam, R.}: The Navier-Stokes Equations. North Holland (1977)

\bibitem{Tri}
\textsc{Triebel, H.}: Interpolation theory, function spaces, differential operators. North-Holland Mathematical Library, 18. North-Holland Publishing Co., Amsterdam-New York, (1978). 528 pp.

\bibitem{Verchota1}
\textsc{Verchota, G.C.}: Layer potentials and regularity for the Dirichlet problem for Laplace's equation in Lipschitz domains. Journal of Functional Analysis, {\bf 59},  572--611 (1984)

\bibitem{Verchota2}
\textsc{Verchota, G.C.}: The biharmonic Neumann problem  in Lipschitz domains. Acta Math. {\bf 194}, 217--279 (2005).

\bibitem{Wend} 
\textsc{Wendland, W.L.} Martin Costabel's version of the trace theorem revisited. \emph{Math. Methods Appl. Sci.} {\bf 40} (2017), no. 2, 329--334.



\end{thebibliography}
\end{document}